\documentclass[12pt,a4paper,leqno]{article}
\usepackage[utf8x]{inputenc}
\usepackage[francais]{babel}
\usepackage[T1]{fontenc}
\usepackage{lmodern}
\usepackage{marvosym}
\usepackage{amsmath}
\usepackage{amssymb, amscd, amsthm, mathrsfs}
\usepackage[all]{xy}
\usepackage[dvips]{graphicx}

\usepackage{makeidx}

\makeindex
\newtheorem{theo}{{Th\'eor\`eme}}[section]
\newtheorem{coro}[theo]{{Corollaire}}
\newtheorem{lemma}[theo]{{Lemme}}
\newtheorem{prop}[theo]{Proposition}

\theoremstyle{remark}
\newtheorem{remark}[theo]{\textbf{Remarque}}

\theoremstyle{definition}
\newtheorem{defn}[theo]{D\'efinition}
\newtheorem{example}[theo]{Exemple}
\newtheorem{recette}[theo]{\textbf{Recette}}

\newcommand{\ra}{\rightarrow}

\newcommand{\ol}{\overline}
\newcommand{\immouv}[1][r]
   {\ar@{}[#1] |*[o][F]{\hbox{%
         \vrule width 1.5mm height 0pt depth 0pt%
         \vrule width 0pt height .75mm depth .75mm%
         }}
         \ar@{^{(}->}[#1]}

\newcommand{\cA}{\mathcal{A}}

\newcommand{\cC}{\mathcal{C}}

\newcommand{\cE}{\mathcal{E}}
\newcommand{\cF}{\mathcal{F}}
\newcommand{\cG}{\mathcal{G}}
\newcommand{\cH}{\mathcal{H}}

\newcommand{\cK}{\mathcal{K}}

\newcommand{\cM}{\mathcal{M}}

\newcommand{\cO}{\mathcal{O}}
\newcommand{\cP}{\mathcal{P}}
\newcommand{\cR}{\mathcal{R}}

\newcommand{\A}{\mathbb A}
\newcommand{\C}{\mathbb C}
\newcommand{\B}{\mathbb B}

\newcommand{\I}{\mathbb I}

\newcommand{\N}{\mathbb N}

\newcommand{\Q}{\mathbb Q}
\newcommand{\R}{\mathbb R}

\newcommand{\Z}{\mathbb Z}

\newcommand{\D}{\mathbf{D}}

\newcommand{\bM}{\mathbf M}
\newcommand{\bR}{\mathbf R}

\newcommand{\fD}{\mathfrak D}

\newcommand{\fF}{\mathfrak F}

\newcommand{\fK}{\mathfrak K}

\newcommand{\rH}{\mathrm H}

\newcommand{\m}{\mathrm m}
\newcommand{\rP}{\mathrm P}

\newcommand{\sW}{\mathscr W}

\newcommand{\sX}{\mathscr X}

\renewcommand{\ker}{\mathrm {Ker} }

\DeclareMathOperator{\coker}{\mathrm Coker }

\renewcommand{\Re}{\mathrm Re}

\newcommand{\Dir}{\mathbf{Dir}}

\DeclareMathOperator{\Hom}{\mathrm Hom}

\DeclareMathOperator{\GL}{\mathrm GL}
\DeclareMathOperator{\SL}{\mathrm SL}

\DeclareMathOperator{\LA}{\mathrm LA}

\DeclareMathOperator{\Fil}{\mathrm{Fil}}

\DeclareMathOperator{\Gal}{\mathrm Gal}

\DeclareMathOperator{\ord}{\mathrm ord}

\DeclareMathOperator{\Spf}{\mathrm{Spf}}
\DeclareMathOperator{\Spm}{\mathrm{Spm}}

\DeclareMathOperator{\Sym}{\mathrm{Sym}}

\DeclareMathOperator{\plim}{\varprojlim}

\newcommand{\alg}{\mathrm{alg}}
\newcommand{\an}{\mathrm{an}}

\newcommand{\con}{\mathrm{cong}}
\newcommand{\cont}{\mathrm{cont}}
\newcommand{\cycl}{\mathrm{cycl}}
\newcommand{\dR}{\mathrm{dR}}

\newcommand{\Inf}{\mathrm{inf}}

\newcommand{\res}{\mathrm{res}}

\newcommand{\z}{\zeta}

\begin{document}
\title{Le système d'Euler de Kato en famille (I)}
\author{ Shanwen \textsc{WANG}}
\date{}
\maketitle {\renewcommand{\thefootnote}{\fnsymbol{footnote}}

\pagestyle{myheadings}

\begin{abstract}
Ce texte est le premier article d'une série d'articles sur une généralisation de système d'Euler de Kato. Il est consacr\'e \`a la construction d'une famille de syst\`emes d'Euler de Kato et à la construction d'une famille de lois de réciprocité sur l'espace des poids $\sW$, qui interpolent les objets classiques.
\end{abstract}
\renewcommand{\abstractname}{Abstract}
\begin{abstract}
This article is the first article of a serie of articles on the generalization of Kato's Euler system. The main subject of this article is to construct a family of Kato's Euler systems and a family of Kato's explicit reciprocity laws over the weight space $\sW$, which interpolate the corresponding classical objects. 
\end{abstract}

\tableofcontents
\section{Introduction }
 \subsection{Notations}\label{notation}
On note $\overline\Q$ la cl\^oture alg\'ebrique de $\Q$ dans $\C$,
et on fixe, pour tout nombre premier~$p$, une cl\^oture
alg\'ebrique $\overline\Q_p$ de $\Q_p$, ainsi qu'un plongement de
$\overline\Q$ dans $\overline\Q_p$.

Si $N\in\N$, on note $\zeta_N$ la racine $N$-i\`eme de l'unit\'e
$e^{2i\pi/N}\in\overline\Q$ , et on note
$\Q^{\rm cycl}$ l'extension cyclotomique de $\Q$,
r\'eunion des $\Q(\zeta_N)$, pour $N\geq 1$, ainsi que $\Q^{\rm cycl}_p$ l'extension cyclotomique de $\Q_p$, r\'eunion de $\Q_p(\z_N)$, pour $N\geq 1$.
 \subsubsection*{Objets ad\'eliques}
 Soient $\cP$ l'ensemble des premiers de $\Z$ et $\hat{\Z}$ le compl\'et\'e profini de $\Z$, alors
\index{qi si wo le} 
$\hat{\Z}=\prod_{p\in\cP}\Z_p$. Soit $\Q\otimes\hat{\Z}$ l'anneau
des ad\`eles finis de $\Q$. Si
$x\in\Q\otimes\hat{\Z}$, on note $x_p$ (resp. $x^{]p[}$) la
composante de $x$ en $p$ (resp. en dehors de $p$). Notons
$\hat{\Z}^{]p[}=\prod_{l\neq p}\Z_l$. On a donc
$\hat{\Z}=\Z_p\times\hat{\Z}^{]p[}$. Cela induit les d\'ecompositions
suivantes: pour tout  $d\geq 1$,
\[
\bM_d(\Q\otimes\hat{\Z})=\bM_d(\Q_p)\times\bM_d(\Q\otimes\hat{\Z}^{]p[})
\text{ et }
\GL_d(\Q\otimes\hat{\Z})=\GL_d(\Q_p)\times\GL_d(\Q\otimes\hat{\Z}^{]p[}).\]
On d\'efinit les sous-ensembles suivants de $\Q\otimes\hat{\Z}$ et de
$\bM_2(\Q\otimes\hat{\Z})$:

\begin{align*}\hat{\Z}^{(p)}=\Z_p^{*}\times\hat{\Z}^{]p[} &\text{ et
}
\bM_{2}(\hat{\Z})^{(p)}=\GL_2(\Z_p)\times\bM_2(\hat{\Z}^{]p[}), \\
(\Q\otimes\hat{\Z})^{(p)}=\Z_p^{*}\times(\Q\otimes\hat{\Z}^{]p[})
&\text{ et }
\bM_{2}(\Q\otimes\hat{\Z})^{(p)}=\GL_2(\Z_p)\times\bM_2(\Q\otimes\hat{\Z}^{]p[}).
\end{align*}
\subsubsection*{Formes modulaires}
Soient $A$ un sous-anneau de $\C$ et $\Gamma$ un sous-groupe
d'indice fini de $\SL_2(\Z)$. On note $\cM_k(\Gamma,\C)$ le
$\C$-espace vectoriel des formes modulaires de poids $k$ pour
$\Gamma$. On note aussi $\cM_{k}(\Gamma,A)$ le sous $A$-module de
$\cM_k(\Gamma,\C)$ des formes modulaires dont le $q$-d\'eveloppement
est \`a coefficients dans $A$. On pose
$\cM(\Gamma,A)=\oplus_{k=0}^{+\infty}\cM_k(\Gamma,A)$, et on note
$\cM_k(A)$ (resp. $\cM(A)$) la r\'eunion des $\cM_k(\Gamma,A)$
(resp. des $\cM(\Gamma,A)$), o\`u $\Gamma$ d\'ecrit tous les
sous-groupes d'indice fini de $\SL_2(\Z)$. On peut munir l'alg\`ebre
$\cM(\C)$ d'une action de $\GL_2(\Q)_{+}=\{\gamma\in\GL_2(\Q)|\det\gamma>0\}$ de la fa\c{c}on suivante:
\begin{equation}\label{etoi}
f*\gamma=(\det\gamma)^{1-k}f_{|_k}\gamma, \text{ pour }
f\in\cM_k(\C) \text{ et } \gamma\in\GL_2(\Q)_{+},
\end{equation}
o\`u $f_{|_k}\gamma$ est l'action modulaire usuelle de $\GL_2(\R)_{+}$.


\begin{defn}Soient $N\geq 1$ et $\Gamma_N=\{\bigl(\begin{smallmatrix}a & b\\ c & d\end{smallmatrix}\bigr)\in\SL_2(\Z), \bigl(\begin{smallmatrix}a & b\\ c & d\end{smallmatrix}\bigr)\equiv \bigl(\begin{smallmatrix}1 & 0\\ 0 & 1\end{smallmatrix}\bigr)\mod N\}$. Le groupe $\Gamma_N$ est un sous-groupe de $\SL_2(\Z)$ d'indice fini.
On dit qu'un sous-groupe $\Gamma$ de $\SL_2(\Z)$ est un sous-groupe
de congruence s'il contient $\Gamma_N$ pour un certain $N\geq 1$.
\end{defn}
\begin{example}
Les sous-groupes
$\Gamma_0(N)=\{\gamma\in\SL_2(\Z)|\gamma\equiv(\begin{smallmatrix}*&*\\0&*\end{smallmatrix})\mod
N\}$ et
$\Gamma_1(N)=\{\gamma\in\SL_2(\Z)|\gamma\equiv(\begin{smallmatrix}1&*\\0&1\end{smallmatrix})\mod
N\}$ sont des sous-groupes de congruences.
\end{example}

On pose:
\[\cM^{\con}_k(A)=\bigcup\limits_{\substack{\Gamma \text { sous-groupe de congruence } }}\cM_k(\Gamma,A)\text{ et }
\cM^{\con}(A)=\bigcup_k\cM_k^{\con}(A).\]

Soit $K$ un sous-corps de $\C$ et soit $\ol{K}$ la cl\^oture alg\'ebrique de $K$. On note $\Pi_K$ le groupe des
automorphismes de $\cM(\bar{K})$ sur $\cM(\SL_2(\Z),K)$; c'est un
groupe profini.
On note
$\Pi_{\Q}^{'}$ le groupe des automorphismes
de $\cM(\overline\Q)$ engendr\'e par $\Pi_{\Q}$ et
$\GL_2(\Q)_{+}$.
Plus g\'en\'eralement, si $S\subset\cP$ est fini, on note $\Pi_{\Q}^{(S)}$ le sous-groupe de $\Pi_{\Q}^{'}$ engendr\'e par $\Pi_{\Q}$ et $\GL_2(\Z^{(S)})_{+}$, o\`u $\Z^{(S)}$ est le sous-anneau de $\Q$ obtenu en inversant tous les nombres premiers qui n'appartiennent pas \`a $S$.
Si $f\in\cM(\ol{K})$, le groupe de galois $\cG_K$ agit sur les coefficients du $q$-d\'eveloppement de $f$; ceci nous fournit une section de $\cG_K\ra \Pi_K$, notée par $\iota_K$. 

Le groupe des automorphismes de $\cM^{\con}(\Q^{\cycl})$ sur $\cM(\SL_2(\Z),\Q^{\cycl})$ est le groupe $\SL_2(\hat{\Z})$, le compl\'et\'e profini de $\SL_2(\Z)$ par rapport aux sous-groupes de congruence. D'autre part, soit $f\in\cM^{\con}(\Q^{\cycl})$, le groupe $\cG_{\Q}$ agit sur les coefficients du $q$-d\'eveloppement de $f$ \`a travers son quotient $\Gal(\Q^{\cycl}/\Q)$ qui est isomorphe \`a $\hat{\Z}^{*}$ par le caract\`ere cyclotomique. On note $H$ le groupe des automorphismes de $\cM^{\con}(\Q^{\cycl})$ sur $\cM(\SL_2(\Z),\Q)$.  La sous-alg\`ebre $\cM^{\con}(\Q^{\cycl})$ est stable par $\Pi_{\Q}$ qui agit \`a travers $H$. Le groupe $H$ est isomorphe à $\GL_2(\hat{\Z})$ et on a le diagramme commutatif de groupes suivant (c.f. par exemple $\cite{Wang}$Théorème. $2.2$):

\[\xymatrix{
1\ar[r]&\Pi_{\bar{\Q}}\ar[r]\ar[d]&\Pi_{\Q}\ar[r]\ar[d]&\cG_{\Q}\ar[r]\ar[d]^{\chi_{\cycl}}\ar@{.>}@/^/[l]^{\iota_\Q}&1\\
1\ar[r]&\SL_{2}(\hat{\Z})\ar[r]&\GL_2(\hat{\Z})\ar[r]^{\det}&\hat{\Z}^{*}\ar[r]\ar@{.>}@/^/[l]^{\iota}&1                     },\]
où la section $\iota_\Q$ de $\cG_{\Q}$ dans $\Pi_{\Q}$ d\'ecrite plus haut envoie $u\in \hat{\Z}^{*}$ sur la matrice $(\begin{smallmatrix}1&0\\0&u\end{smallmatrix})\in \GL_2(\hat{\Z})$.

\subsection{Le syst\`eme d'Euler de Kato}
En bref, un syst\`eme d'Euler est une collection de classes de cohomologie v\'erifiant une relation de distribution. La construction du syst\`eme d'Euler de Kato dans \cite{KK}, \cite{PC1} ou \cite{Wang} est
comme suit:

\`A partir des unités de Siegel, on construit une distribution
alg\'ebrique $z_{\textbf{Siegel}}$ sur $(\Q\otimes\hat{\Z})^2-(0,0)$ \`a
valeurs dans $\Q\otimes(\cM(\bar{\Q})[\frac{1}{\Delta}])^{*}$, o\`u $\Delta=q\prod_{n\geq 1}(1-q^n)^{24}$ est la forme modulaire de poids $12$. La
distribution $z_{\textbf{Siegel}}$ est invariante sous l'action du groupe
$\Pi_{\Q}^{'}$.
La th\'eorie de Kummer $p$-adique nous fournit un \'el\'ement
\[z_{\textbf{Siegel}}^{(p)}\in \rH^1(\Pi^{'}_{\Q},\fD_{\alg}((\Q\otimes\hat{\Z})^{2}-(0,0),\Q_p(1))).\]
Par cup-produit et
restriction \`a $\Pi_{\Q}^{(p)}\subset\Pi'_\Q$
et $\bM_2(\Q\otimes\hat{\Z})^{(p)}\subset((\Q\otimes\hat{\Z})^{2}-(0,0))^2$,
 on obtient une distribution alg\'ebrique:
\[z_{\textbf{Kato}}\in\rH^2(\Pi^{(p)}_{\Q},\fD_{\alg}(\bM_2(\Q\otimes\hat{\Z})^{(p)},\Q_p(2))).\]
En modifiant $z_{\textbf{Kato}}$ par un op\'erateur $(c^2-\langle c,1\rangle)(d^2-\langle1,d\rangle)$
(c.f. \cite{Wang} $\S 2.3.2$) qui fait dispara\^\i tre les d\'enominateurs,
 on obtient une distribution alg\'ebrique \`a valeurs dans $\Z_p(2)$
(que l'on peut donc voir comme une mesure, notée par $z_{\textbf{Kato},c,d}$), et une torsion \`a
la Soul\'e nous fournit enfin un \'el\'ement
\[z_{\textbf{Kato},c,d}(k,j)\in\rH^2(\Pi^{(p)}_{\Q},\fD_{\alg}(\bM_2(\Q\otimes\hat{\Z})^{(p)},V_{k,j})),\]
o\`u $V_{k,j}=\Sym^{k-2}V_p\otimes\Q_p(2-j)$, o\`u $V_p$ est la repr\'esentation standard de dimension
$2$ de $\GL_2(\Z_p)$.

\subsection{Une famille de syst\`emes d'Euler de Kato sur l'espace des poids }\label{introfam}
Dans la suite, on désigne par $\I_0(p)$ le sous-groupe d'Iwahori de $\GL_2(\Z_p)$
\[\{\gamma\in\GL_2(\Z_p)|\gamma\equiv(\begin{smallmatrix}*&*\\0&*\end{smallmatrix})\mod
p\}.\] On note $\bM_2^{(p)}=\bM_2(\Q\otimes\hat{\Z})^{]p[}\times \I_0(p)$.
  
 On note $\sW$ l'espace rigide analytique associé à l'algèbre d'Iwasawa $\Z_p[[\Z_p^*]]$: c'est l'\emph{espace des poids} et on dispose d'une inclusion de $\Z\subset\mathscr{W}$ en envoyant $k$ sur le caractère algébrique $(z\mapsto z^{k-2})$ sur $\Z_p^*$.
 Dans $\S \ref{const}$, on introduit la notion de \emph{caractère universel} (c.f. Définition \ref{carac}) pour l'espace des poids.
On construit une famille de représentations de Banach $\D_{0,\rho^{\mathbf{univ}}_{j}}$ de $\I_0(p)$ sur l'espace $\sW$, où l'indice $\rho^{\mathbf{univ}}_{j}$ désigne la torsion par un $1$-cocycle $\rho^{\mathbf{univ}}_{j}$ associé au caractère universel $\kappa^{\mathbf{univ}}$ de l'espace des poids et au caractère $\det^{-j}$ avec $j\geq 1$, qui admet une section globale $\nu_j$ interpolant les vecteurs de plus haut poids dans $V_{k,j+2}$  via une application de spécialisation $\I_0(p)$-équivariante $\mathbf{Sp}_{k,j}$. 

On note $\tilde{\Gamma}_0(p)$ l'image inverse de $\GL_2(\hat{\Z})^{]p[}\times\I_0(p)$ dans $\Pi_{\Q}$ via l'application $\Pi_\Q\ra\GL_2(\hat{\Z})$. Si $V$ est un $\Z$-module, on note $\fD_0(\bM_2^{(p)},V)$ le $\Z$-module des mesures sur $\bM_2^{(p)}$ à valeurs dans $V$.
 En appliquant la technique de "torsion à la Soulé" à l'élément $z_{\mathbf{Kato},c,d}$ et à la famille de représentations $(\D_{0,\rho^{\mathbf{univ}}_{j}},\nu_j)$ munie d'un vecteur universel de plus haut poids $\nu_j$, on construit une famille d'éléments de Kato sur l'espace des poids 
 \[z_{\mathbf{Kato},c,d}(\nu_j) \in\rH^2(\tilde{\Gamma}_0(p), \fD_0(\bM_2^{(p)},\D_{0, \rho^{\mathbf{univ}}_{j}}(2))).\] 
 Comme $\mathbf{Sp}_{k,j}: \D_{0,\rho^{\mathbf{univ}}_{j}}\ra V_{k,j+2} $ est $\I_0(p)$-équivariante, on en déduit un morphisme de $\rH^2(\tilde{\Gamma}_0(p), \fD_0(\bM_2^{(p)},\D_{0, \rho^{\mathbf{univ}}_{j}}(2)))$ dans $\rH^2(\tilde{\Gamma}_0(p), \fD_0(\bM_2^{(p)},V_{k,j}))$, qui sera encore noté par $\mathbf{Sp}_{k,j}$.  Notre premier résultat est que $z_{\mathbf{Kato},c,d}(\nu_j)$ interpole les $z_{\mathbf{Kato},c,d}(k,j)$.
\begin{theo}$($Théorème \ref{fam-poids}$)$ Si $1\leq j\in\N$, alors pour tout entier $k\geq 1+j$, on a 
\[\textbf{Sp}_{k,j}(z_{\mathbf{Kato},c,d}(\nu_j))=z_{\mathbf{Kato},c,d}(k,j) .\]
\end{theo}
\begin{remark}D'autres éléments de Kato en famille ont été construits par Fukaya $\cite{Fu}$ (via la $K$-théorie) et par Delbourgo $\cite{DD}$ (via les symboles modulaires). Par contre, notre construction est consiste à reprendre la construction de Kato \cite{KK} et de Colmez \cite{PC1} pour réaliser la déformation.   
\end{remark}
\subsection{La distribution d'Eisenstein}
Les séries d'Eisenstein-Kronecker satisfont des relations de distribution (c.f. Lemme~\ref{lemma15}). Cela permet, si $c,d\in\Z_p^*$, de construire deux distributions algébriques 
$z_{\mathbf{Eis},d}(k)$ et $z'_{\mathbf{Eis},c}(k)$ sur $(\Q\otimes\hat{\Z})^2$ à valeurs dans $\cM_k^{\con}(\Q_p^{\cycl})$ (c.f. Proposition~\ref{eiskj}), invariantes sous l'action de $\GL_2(\Q\otimes\hat{\Z})$ à multiplication près par une puissance de déterminant. 

Soient $k\geq 2$ et $1\leq j\leq k-1$ deux entiers. On définit
\[z_{\mathbf{Eis},c,d}(k,j)=\frac{1}{(j-1)!}z'_{\mathbf{Eis},c}(k-j)\otimes z_{\mathbf{Eis},d}(j)\in\fD_{\alg}(\bM_2(\Q\otimes\hat{\Z}),\cM_k(\Q^{\cycl})),\]   
en utilisant le fait que le produit de deux formes modulaires de poids $i$ et $j$ est une forme modulaire de poids $i+j$ et en identifiant 
$(\Q\times\hat{\Z})^2\times (\Q\otimes\hat{\Z})^2$ avec $\bM_2(\Q\otimes\hat{\Z})$.

On peut, si $j$ est fixé, interpoler ces distributions par une distribution analytique sur l'espace des poids. On dira qu'une série formelle $F(q)=\sum_{n\in\Q_+}A_nq^n\in \cup_{M\in\Z}\cO(\sW)[[q_M]]$ est une famille $p$-adique de formes modulaires sur l'espace des poids $\sW$, si pour presque tout $k\in\N\subset\sW$, l'évaluation $\mathbf{Ev}_{k}(F(q))$ de $F(q)$ en point $k$ soit le $q$-développement d'une forme modulaire classique de poids $k$. Dans $\S~\ref{construction_fam_Eis}$, si $c\in\Z_p^*$ et si $j\geq 1$,  pour chaque $(\alpha,\beta)\in (\Q/\Z)^2$, on construit une famille $p$-adique de séries d'Eisenstein $F_{c,\alpha,\beta}(\kappa^{\mathbf{univ}},j)$ sur l'espace des poids, qui interpole les séries d'Eisenstein-Kronecker $F_{c,\alpha,\beta}^{(k-j)}$ en $k$ (c.f. Lemme~\ref{special}). 

On note $\cM(\overline{\Q}\otimes\cO(\sW))$ le $\ol{\Q}\otimes\cO(\sW)$-module des familles de formes modulaires $p$-adiques sur $\sW$.
De la construction de ces familles $F_{c,\alpha,\beta}(\kappa^{\mathbf{univ}},j)$ et de la relation de distribution entre les $F_{c,\alpha,\beta}^{(k-j)}$, on déduit une relation de distribution pour ces familles $F_{c,\alpha,\beta}(\kappa^{\mathbf{univ}},j)$, qui se traduit en une distribution algébrique $z'_{\mathbf{Eis,c}}(\kappa^{\mathbf{univ}},j)$ (c.f. Théorème.~\ref{efamdis}).
Ceci nous permet de construire une distribution (c.f. ~$\S~\ref{discom}$) 
\[z_{\mathbf{Eis},c,d}(\kappa^{\mathbf{univ}},j)\in\fD_{\alg}(\bM_2^{(p)},\cM(\overline{\Q}\otimes\cO(\sW))),\]
pour $1\leq j\in\Z$, qui interpole les distributions d'Eisenstein $z_{\mathbf{Eis},c,d}(k,j)$ en $k$. 

 \subsection{La loi de réciprocité explicite de Kato en famille}

 Soit $\fK^+=\Q_p\{\frac{q}{p}\}$ l'alg\`ebre des fonctions analytiques sur la boule $v_p(q)\geq 1$ \`a coefficients dans $\Q_p$;  c'est un anneau principal complet pour la valuation $v_{p,\fK}$ d\'efinie par la formule:
\[ v_{p,\fK}(f)=\inf_{n\in\N}v_p(a_n), \text{ si } f=\sum_{n\in\N}a_n(\frac{q}{p})^n\in \fK^+ .\]
 On note $\fK$ le complété du corps des fractions de l'anneau $\fK^+$ pour la valuation $v_{p,\fK}$,  $\ol{\fK}$ une clôture algébrique de $\fK$, ainsi que $\fK_{\infty}\subset\ol{\fK}$ la sous-extension algébrique de $\fK$ en rajoutant les $\z_M$ et les racines $M$-ièmes $q^{\frac{1}{M}}$ de $q$, pour tout $M\geq 1$.  On note $\ol{\fK}^{+}$ (resp. $\fK_{\infty}^+$) la clôture intégrale de $\fK^+$ dans $\ol{\fK}$ (resp. $\fK_{\infty}$).
 
 On note $\cG_{\fK}$ le groupe de Galois de $\ol{\fK}$ sur $\fK$, $P_{\Q_p}$ le groupe de Galois de $\ol{\Q}_p\fK_{\infty}$ sur $\fK$, et $P_{\Q_p}^{\cycl}$ le groupe de Galois de $\fK_{\infty}$ sur $\fK$.  L'application qui à une forme modulaire associe son $q$-développement, nous fournit une inclusion de $\cM(\ol{\Q})$ dans $\ol{\Q}_p\fK_\infty^+$ équivariante sous l'action de $P_{\Q_p}$. On en déduit ainsi un morphisme $P_{\Q_p}\ra \Pi_\Q$,  qui induit un morphisme $\cG_\fK\ra \Pi_\Q$ et un morphisme "de localisation" $\rH^i(\Pi_\Q, W)\ra \rH^i(\cG_\fK, W)$ pour tout $\Pi_\Q$-module $W$ et tout $i\in \N$. 
 
 La loi de r\'eciprocit\'e explicite de Kato consiste \`a
relier l'\'el\'ement $z_{\mathbf{Kato},c,d}$, qui vit dans la cohomologie
du groupe $\Pi^{(p)}_{\Q}$, \`a une distribution
construite \`a partir du produit de deux s\'eries d'Eisenstein
(le produit scalaire de Petersson d'une forme primitive avec
un tel produit fait appara\^itre les valeurs spéciales de la fonction $L$ de $f$,
et c'est cela qui permettrait de construire la fonction $L$ $p$-adique).
Ceci se fait en plusieurs \'etapes:\\
$\bullet$ On commence par "localiser" notre classe de cohomologie
\`a $\cG_{\fK}$ et par
\'etendre les coefficients de
$V_{k,j}$ \`a $\B_{\dR}^+\otimes V_{k,j}$, o\`u
$\B_{\dR}^+=\B_{\dR}^+(\bar\fK^+)$ est un anneau de Fontaine.
\\
$\bullet$ On constate que l'image de $z_{\mathbf{Kato},c,d}(k,j)$ sous l'application "de localisation" \[\rH^2(\Pi^{(p)}_{\Q},\fD_{\alg}(\bM_2(\Q\otimes\hat{\Z})^{(p)},V_{k,j}))
\ra\rH^2(\cG_{\fK},\fD_{\alg}(\bM_2(\Q\otimes\hat{\Z})^{(p)},\B_{\dR}^+\otimes V_{k,j}))\]
est l'inflation d'un $2$-cocycle sur $P^{\cycl}_{\Q_p}$
\`a valeurs dans $\fD_{\alg}(\bM_2(\Q\otimes\hat{\Z})^{(p)},
\B_{\dR}^+\otimes V_{k,j})$.
Les m\'ethodes de descente presque \'etale de Tate~\cite{Ta} et
Sen~\cite{Sen}, revisit\'ees par Faltings~\cite{Fa}
(cf.~aussi Andreatta-Iovita~\cite{AI}) permettraient de montrer
que c'est toujours le cas, mais une preuve directe
pour l'\'el\'ement de Kato est donnée dans  $\cite{Wang}\S~5.2$.
\\
$\bullet$ On construit
une application exponentielle duale (c.f.~\cite{Wang} $\S~5.1$):
\[\exp^{*}_{\mathbf{Kato}}:\rH^2(P_{\Q_p},\fD_{\alg}(\bM_2(\Q\otimes\hat{\Z})^{(p)},\B_{\dR}^+\otimes V_{k,j}))\ra\rH^0(P_{\Q_p},\fD_{\alg}(\bM_2(\Q\otimes\hat{\Z})^{(p)},\cK_{\infty}^+)),\]
où $\cK_{\infty}^+$ est la réunion des $\cK_M^+$ pour tout $M\geq 1$ avec $\cK_M^+$ le complété $q$-adique de la clôture intégrale $\fK_M^+$ de $\fK^+$ dans $\fK_M=\fK[\z_M,q_M]$; 
et on calcule l'image de $z_{\mathbf{Kato},c,d}(k,j)$.

La définition de l'application $\exp^*_{\mathbf{Kato}}$ et le calcul de l'image de $z_{\mathbf{Kato},c,d}(k,j)$ reposent sur la méthode de Tate-Sen-Colmez et sur une description explicite (c.f. \cite{Wang} prop.4.10) de la cohomologie de $P_{\fK_M}=\Gal(\fK_{Mp^{\infty}}/\fK_M)$ avec $M\geq 1$ tel que $v_p(M)=m\geq v_p(2p)$ et $\fK_{Mp^\infty}=\cup_{n\geq 1}\fK_{Mp^n}$. 

Le groupe $P_{\fK_M}$ est un groupe analytique $p$-adique de rang $2$, isomorphe à 
\[\rP_m=\{(\begin{smallmatrix}a& b\\ c& d\end{smallmatrix})\in\GL_2(\Z_p): a=1, c=0, b\in p^m\Z_p, d\in 1+p^m\Z_p \}.\]
On note $(u,v)\in (p^m\Z_p)^2$ l'élément $(\begin{smallmatrix}1& u\\ 0& e^v\end{smallmatrix})$
de $\rP_m$. 
Soit $V$ une représentation \emph{analytique} (c.f. Définition \ref{def-an}) de $\rP_m$, on dispose des dérivations $\partial_{m,i}:V\ra V$, $i=1,2$, définis par 
\[\partial_{m,i}=\lim_{n\ra+\infty}\frac{\gamma_i^{p^n}-1}{p^n}, \text{ où } \gamma_1=(p^m,0), \gamma_2=(0,p^m).\]
La proposition suivante est une version entière de $\cite{Wang}$ prop. 4.10.

\begin{prop}[Proposition ~\ref{analytic}]
Soit $V$ une repr\'esentation analytique de $\rP_m$ munie d'un $\Z_p$-réseau $T$ stable sous l'action de $\rP_m$. Alors
\begin{itemize}
\item[$(1)$] tout \'el\'ement de $ \rH^2(\rP_m,V) $ est repr\'esentable par un $2$-cocycle analytique à $p^{2m}$-torsion près;
\item[$(2)$] on a $ \rH^2(\rP_m,V)\cong (V/(\partial_{1},\partial_{2}-1))$,
et l'image d'un $2$-cocycle analytique
\[((u,v),(x,y))\ra c_{(u,v),(x,y)}=\sum_{i+j+k+l\geq 2}c_{i,j,k,l}u^iv^jx^ky^l,\]
avec $p^{(i+j+k+l)m}c_{i,j,k,l}\in T$, par cet isomorphisme, est celle de
$\delta^{(2)}(c_{(u,v),(x,y)})=p^{2m}(c_{1,0,0,1}-c_{0,1,1,0})$ à $p^{2m}$-torsion près.
\end{itemize}
\end{prop}

En utilisant la proposition ci-dessus, on construit une application exponentielle duale en famille (c.f.~$\S~\ref{constructiondeloi}$) $\exp^{*}_{\mathbf{Kato},\nu}$ de \[\rH^2(P_{\Q_p},\fD_{\alg}(\bM_2^{(p)},\B_{\dR}^{+}(\fK^+_{Mp^{\infty}})\hat{\otimes}\D_{0,\rho^{\mathbf{univ}}_{j-2}}(\sW) ))\] dans $\fD_{\alg}(\bM_2^{(p)},\cK_{\infty}^{+}\hat{\otimes}_{\Z_p}\cO(\sW) )$.
On calcule l'image de $z_{\mathbf{Kato},c,d}(\nu_j)$ sous l'application exponentielle duale $\exp^*_{\mathbf{Kato},\nu}$ et obtient finalement notre r\'esultat principal, qui montre que l'application $\exp^{*}_{\mathbf{Kato},\nu}$ est une application exponentielle duale en famille:
\begin{theo}\label{theo}Si $ j\geq 1$ et si $c,d\in\Z_p^{*}$, alors pour tout entier $k\geq j+1 $,
on a:\\
$(1)$ $\exp^{*}_{\mathbf{Kato},\nu}(z_{\mathbf{Kato},c,d}(\nu_j))=z_{\mathbf{Eis},c,d}(\kappa^{\mathbf{univ}},j);$\\
$(2)$ $\mathbf{Ev}_{k}\circ \exp^{*}_{\mathbf{Kato},\nu}(z_{\mathbf{Kato},c,d}(\nu_j))=\exp^{*}_{\mathbf{Kato}}\circ \textbf{Sp}_{k,j}(z_{\mathbf{Kato},c,d}(\nu_j))$.
\end{theo}
\begin{remark}Dans $\cite{Fu}$, Fukaya a construit une fonction L $p$-adique en deux variables sur une famille de formes modulaires ordinaires en appliquant sa $K_2$-série de Coleman à son élément universel de Kato. Par ailleurs, Panchishkin $\cite{AP}$ a donné une construction de la fonction L $p$-adique en deux variables sur une famille de formes modulaires non-ordinaires sans utiliser le système d'Euler de Kato. Récemment, Bellaïche \cite{Be} a construit une fonction L $p$-adique en deux variables sur la courbe de Hecke ("eigencurve") en utilisant la théorie de symboles surconvergents. 
Les résultats de cet article peuvent aider à construire une fonction L $p$-adique en deux variables sur la courbe de Hecke en utilisant le système d'Euler de Kato, ceci est le sujet de \cite{Wang1}.
\end{remark}

\subsection*{Remerciements:}Cet article est une continuation de ma thèse sous la direction de P. Colmez. Je voudrais lui exprimer ma profonde gratitude pour sa générosité, sa disponibilité et aussi pour le temps qui m'a accordé. 
Je tiens également 
à remercier A. Iovita pour sa suggestion de travailler avec l'espace des distributions pour construire la famille de représentations de Banach.  
Une partie de cet article est faite pendant mon séjour à Lausanne pendant Janvier-Juin 2011, je remercie sincèrement P. Michel et la chaire de TAN de EPFL pour leur hospitalité.

\section{Famille de systèmes d'Euler de Kato sur l'espace $\sW$}
\subsection{Groupe rigide analytique $p$-adique et caract\`ere universel}

Soit $G$ un groupe abélien discret ou un pro-$p$ groupe abélien,  topologiquement de type fini. La $\Z_p$-algèbre de groupe $\Z_p[G]$ (resp. la $\Z_p$-algèbre de groupe complété $\Z_p[[G]]$ si $G$ est un pro-$p$-groupe), est naturellement une algèbre de Hopf (resp. une algèbre de Hopf complète), dont la structure de Hopf est induite par la structure de groupe de $G$. Alors on peut lui associer un groupe rigide analytique $\sX$ sur $\Q_p$. Les $\C_p$-points de $\sX$ sont l'ensemble $\Hom_{\cont}(G, \C_p^*)$, des caractères continus sur $G$ à valeurs dans $\C_p$.

\begin{defn}\label{carac}Soit $G$ un groupe ci-dessus et soit $\sX$ le groupe rigide analytique associé. L'inclusion canonique $G\ra \Z_p[G]\ra \cO(\sX)$ induit un caractère $\kappa^{\mathbf{univ}}$ de $G$ dans $\cO(\sX)^*$. On l'appelle le \emph{caractère universel} du groupe rigide analytique $\sX$.
\end{defn}
Soit $U$ un ouvert affinoïde de $\sX$. 
Si $\kappa\in U(\C_p)$ est un caractère continu sur $G$ à valeurs dans $\C_p^*$, on définit une application d'évaluation en point $\kappa$
\[\mathbf{Ev}_\kappa:\cO(U)\ra \C_p.\]  
Par définition, on a le lemme suivant:
\begin{lemma}Si $\kappa\in\sX(\C_p)$ est un caractère continu sur $G$ à valeurs dans $\C_p^*$, on a alors 
\[\kappa=\mathbf{Ev}_\kappa\circ \kappa^{\mathbf{univ}}.\]
\end{lemma}
\begin{example}\label{espacepoids}
On définit l'espace rigide analytique associé à l'algèbre $\Z_p[[T]]$ en utilisant le schéma formel affine (c.f.\cite[\S 7]{DJ}). 

On note $\Spf \Z_p[[T]]$ le schéma formel. On note $A_n=\Z_p[[T]][\frac{T^n}{p}]$ le sous-anneau de $\Z_p[[T]]\otimes\Q_p$ et $B_n$ le complété de $A_n$ pour la topologie $(p,T)$-adique. On constate que $B_n$ est aussi le complété $p$-adique de $A_n$:
\[B_n=\{\sum_{m=0}^{+\infty}a_m(\frac{T^n}{p})^m: a_m\in\Z_p[[T]] \text{ tel que } \lim_mv_p(a_m)=+\infty \}.\] Enfin, on note $C_n=B_n\otimes \Q_p$. Ceci nous donne un système projectif $\{C_n\}$ muni de morphismes de transition $C_{n+1}\ra C_n$ induits par les inclusions $A_{n+1}\subset A_n$.  Les $C_n$ sont des $\Q_p$-algèbres affinoïdes et le morphisme de transition identifie $\Spm(C_{n+1})$ comme un sous-domain affinoïde de $\Spm(C_{n})$. En fait, $C_n$ est l'anneau des fonctions analytiques sur le disque $v_p(T)\geq \frac{1}{n}$. 

 On définit l'espace rigide analytique $(\sW_0,\cO)$ associé au schéma formel $\Spf \Z_p[[T]]$ comme la réunion croissante admissible des affinoïdes $\Spm(C_n)$;
l'algèbre des sections globales $\cO(\sW_0)$ s'identifie à la limite projective $\plim_n C_n$, qui est $\cR^+$ l'anneau des fonctions analytiques sur le disque ouvert $v_p(T)>0$.

On note $\Lambda=\Z_p[[\Z_p^{*}]]$ l'alg\`ebre d'Iwasawa, qui est
un anneau semi-local r\'egulier, noethérien, de dimension de Krull
$2$. L'espace rigide analytique $\mathscr{W}$ sur
$\Q_p$, qui lui est associé est appel\'e l'\emph{espace des poids}. Les $\C_p$-points de $\mathscr{W}$ est l'ensemble $\Hom_{\cont}(\Lambda, \C_p)$, des caractères continus sur $\Z_p^{*}$ \`a valeurs dans $\C_p^{*}$. 

On note $\omega$ le caractère de Teichmüller. On a une d\'ecomposition $\Z_p^{*}\cong \mu_{p-1}\times (1+p\Z_p)$ donnée par la formule $z\mapsto (\omega(z), \langle z\rangle)$, qui induit la décomposition $\mathscr{W}=\Hom(\mu_{p-1},\C_p^{*})\times\Hom_{\cont}(\Gamma, \C_p^{*}) $, o\`u $\Gamma=(1+p\Z_p)$. 
L'espace $\mathscr{W}$ admet d'un recouvrement admissible $\{W_n\}_{n\geq 1}$, où
$W_n=\Hom(\mu_{p-1},\C_p^{*})\times\Hom_{\cont}(1+p^n\Z_p, \C_p^{*})$.
En fixant un g\'en\'erateur $\gamma\in\Gamma$, on a un isomorphisme 
\[\Hom_{\cont}(\Gamma, \C_p^{*})\cong \{t\in\C_p, v_p(t-1)>0 \}, \]
en envoyant  un caractère continu $\kappa$ sur $\kappa(\gamma)$.
En fixant un g\'en\'erateur $(\z,\gamma)$ de $\mu_{p-1}\times(1+p\Z_p)$, le caract\`ere universel $\kappa^{\mathbf{univ}}$ de $\mathscr{W}$ est donné par la formule $z\mapsto X_1^{\ord_{\z}\omega(z)}T_1^{\log_\gamma \langle z\rangle}$,  où on remplace les variables $\z$ et $\gamma$ par $X_1$ et $T_1$ respectivement.
Comme $\log_\gamma\langle z\rangle=\frac{\log\langle z\rangle}{\log\gamma}\in \Z_p$, $T^{\log_\gamma\langle z\rangle}\in\Z_p[[T-1]]$.


\end{example}
\begin{defn} Soit $\mathscr{X}$ un $\Q_p$-espace rigide.
\begin{itemize}
\item[(1)] Un sous-ensemble $Z\subset \mathscr{X}$ est dit Zariski-dense si pour tout sous-ensemble analytique $U\subset \mathscr{X}$ tel que $Z\subset U$, alors $U=\mathscr{X}$. 
 \item[(2)] Soit $Z\subset \mathscr{X}$ un sous-ensemble Zariski-dense, tel que pour tout $z\in Z$ et tout voisinage ouvert affinoïde $V$ de $z$ dans $\mathscr{X}$, $V\cap Z$ est Zariski-dense dans chaque composante irréductible de $V$ contenant $z$, on dira alors que $Z$ est très Zariski-dense dans $\mathscr{X}$.
 \end{itemize}
 \end{defn}
  Un exemple important est que l'ensemble $\N$ (resp. $\Z$) est très Zariski-dense dans l'espace des poids $\sW$.

\subsection{Famille de représentations de Banach sur $\sW$ et ses tordues }\label{const}

Soit $K$ un sous-corps complet de $\C_p$ et soit $(A,|\cdot|_A)$ une $K$-algèbre de Banach.
\begin{defn}  Soit $M$ un $A$-module; une norme sur $M$ est une application $|\cdot|: M\ra \R_{+}$ telle que 
\begin{itemize}
\item[(1)] $|m|=0$ si et seulement si $m=0$;
\item[(2)] $|m+n|\leq \max\{ |m|, |n|\}$, si $m, n\in M$;
\item[(3)] $| am|\leq  |a|_A|m|$ si $a\in A$ et $m\in M$. 
\end{itemize}
On dit qu'un $A$-module $M$ muni d'une norme est de Banach si $M$ est complet pour cette norme.
\end{defn}

\begin{example}\label{ex1}
\begin{itemize}
\item[(1)] Soit $I$ un ensemble. On note $\cC_A^b(I)$ l'ensemble des familles bornées $x=(x_i)_{i\in I }$ d'éléments de $A$. On munit $\cC_A^b(I)$ de la norme $|x|=\sup_{i\in I}|x_i| $, ce qui en fait un $A$-module de Banach. On note $\cC^0_A(I)$ le sous $A$-module de $\cC_A^b(I)$ des suites  $x=(x_i)_{i\in I }$ d'éléments de $A$ tendant vers $0$ suivant le filtre des complémentaires des parties finies (ce que nous écrivons $x_i\ra 0$ pour $i\ra\infty$). C'est un $A$-module de Banach comme sous-$A$-module fermé d'un $A$-module de Banach.  Si $A=K$ est de valuation discrète et si $|M|=|K|$,  tout $K$-espace de Banach est de la forme $\cC_K^0(I)$.

\item[(2)] Soient $(M,|\cdot|_M),(N,|\cdot|_N)$ deux $A$-modules de Banach. On note $\Hom_A(M,N)$ le $A$-module des morphismes continus de $A$-modules. On le munit de la norme suivante
\[     |f|=\sup_{0≠m\in M} \frac{|f(m)|_N}{|m|_M} , \text{ si } f\in \Hom_A(M,N),      \]  
ce qui en fait un $A$-module de Banach. 
En particulier, si $N=A$, $\check{M}=\Hom_{A}(M,A)$ est le $A$-module de Banach dual de $M$. 
\item[(3)] Soit $X$ un espace topologique et soit $M$ un $A$-module de Banach. 
On note $\cC^0(X,M)$ le $A$-module des fonctions continues sur $X$ à valeurs dans $M$ muni de la topologie compact-ouvert (i.e. convergent uniformément sur tout sous-ensemble compact de $X$).  
Soit $\Gamma\subset X$ un sous-ensemble compact et soit $V$ un sous $A$-module ouvert de $M$; on note $U_{\Gamma, V}$ le sous-ensemble de $\cC^0(X,M)$ 
des fonctions continues à support dans $\Gamma$ et à valeurs dans $V$. Ceci fournit  une base d'ouverts si $\Gamma$ décrit les sous-ensembles compact de $X$ et $V$ décrit les sous-ensembles ouverts de $M$. Si $X$ est compact, le $A$-module topologique $\cC^0(X,M)$  est un $A$-module de Banach, qui est isomorphe au $A$-module de Banach $\cC^0(X,K)\hat{\otimes}_KM$.
\end{itemize}
\end{example}

\begin{defn}Soit $M$ un $A$-module de Banach.
\begin{itemize}
\item[(1)]Une base de Banach de $M$ est une famille bornée $(e_i)_{i\in I}$ de $M$ telle que tout élément $x$ de $M$ peut s'écrire de manière unique sous la forme d'une série convergente $x=\sum_{i\in I }a_i e_i$, 
où $a_i$ sont des éléments de $A$ tendant vers $0$ suivant le filtre des complémentaires des parties finies.
\item[(2)]Une base orthonormale de $M$ est une base de Banach $(e_i)_{i\in I}$ de $M$ telle que l'application $(a_i)_{i\in I}\mapsto \sum_{i\in I} a_i e_i$, de $\cC^0_A(I)$ dans $M$, est une isométrie. 
\end{itemize}
\end{defn}
On dit que $M$ est orthonormalisable si $M$ admet une base orthonormale.
\begin{example}
\begin{itemize}\label{example-Banach}
\item[(1)] Si $i\in\N$, on note $\ell(i)$ le plus petit entier $n$ vérifiant $p^n>i$. On a donc 
\[\ell(0)=0 \text{ et } \ell(i)=[\frac{\log i}{\log p}] +1, \text{ si } i\geq 1 .\]
Si $r\in\N$, on note $\cC^r(\Z_p,A)$ le $A$-module des \emph{fonctions de classe $\cC^r$} à valeurs dans $A$ (c.f.\cite{PC2} I.5). Comme $\Z_p$ est compact, on a un isomorphisme de $A$-modules $\cC^r(\Z_p,A)\cong \cC^r(\Z_p, K)\hat{\otimes}_KA$.  Par ailleurs, le $K$-espace vectoriel $\cC^r(\Z_p,K)$ est un $K$-espace de Banach muni d'une base orthonormale $\{p^{r\ell(n)}\binom{z}{n}; n\geq 0\}$.
Ceci implique que $\cC^r(\Z_p,A)$ est orthonormalisable. En particulier, si $r=0$, $\cC^0(\Z_p,A)$ est le $A$-module des fonctions continues sur $\Z_p$ à valeurs dans $A$.
\item[(2)]Si $h\in \N$, on note $\LA_h(\Z_p, A)$ l'espace des fonctions $\varphi: \Z_p\ra A$ dont la restriction de $\varphi$ à $a+p^h\Z_p$ est la restriction d'une fonction $A$-analytique sur le disque fermé $\{x\in \C_p; v_p(x-a)\geq h\}$, quel que soit $a\in \Z_p$; c'est aussi un $A$-module de Banach orthonormalisable  et on a $\LA_h(\Z_p, A)\cong \LA_h(\Z_p,K)\hat{\otimes}_KA$.  On a une inclusion de $A$-modules $\LA_h(\Z_p,A)\subset\cC^r(\Z_p,A)$ pour tout $r\geq 0$ et $h\in\N$.
\item[(3)]Soit $M$ un $A$-module de Banach orthonormalisable et soit $\{e_i\}_{i\in I}$ une base orthonormale de $M$. On a un isomorphisme de $A$-modules:
\[\check{M}\cong\Hom_A(\cC_A^0(I), A)\cong \Hom_K(\cC_K^0(I),A).\]
On note $M^{\vee}$ l'image inverse de $\Hom_K(\cC_K^0(I),K)\hat{\otimes }_K A$ dans $\check{M}$ sous l'isomorphisme ci-dessus, qui en fait un sous-$A$-module de Banach.  De plus,  $M^\vee$ est orthonormalisable. 
\end{itemize}
\end{example}

\begin{defn}Soit $G$ un groupe topologique. 
Une $A$-représentation de Banach de $G$ est un $A$-module de Banach orthonormalisable muni d'une action $A$-linéaire continue de $G$.
\end{defn}

\begin{example}
\begin{itemize}
\item[(1)]Soit $V_0$ une $K$-représentation de Banach de $G$ (i.e. un $K$-espace de Banach muni d'une action $K$-linéaire continue de $G$). On étend l'action de $G$ sur $V_0$ par $A$-linéarité en une action sur le $A$-module de Banach $V=V_0\hat{\otimes}_K A$, ce qui est une $A$-représentation de Banach de $G$.
\item[(2)]Soit $V$ une $A$-représentation de Banach de $G$ à droite. On suppose que $V$ est munie d'une structure d'anneau. On note $V^*$ le groupe des unités de $V$. Soit $\eta:G\ra V^{*}$ un $1$-cocycle (i.e. une fonction continue sur $G$ à valeurs dans $V$ telle que $\eta(h_1h_2)=\eta(h_1)*h_2\eta(h_2)$).  On note $V(\eta)$ la $A$-représentation de Banach $V$ tordue par le $1$-cocycle $\eta$.
\item[(3)] Soit $V$ une $A$-représentation de Banach de $G$ à droite. Sa $A$-duale $ \check{V}=\Hom_A(V, A)$ est munie de l'action de $G$ à gauche donnée par la formule
\[\gamma f (v)=f(v*\gamma ), \text{ si } \gamma\in G, f\in \check{V}, v\in V. \]
Le sous-$A$-module $V^{\vee}$ (c.f. Exemple \ref{example-Banach} (3)) de $\check{V}$ est stable sous cette action, et donc une $A$-représentation de Banach de $G$ à gauche. 
\end{itemize}
\end{example}

\begin{defn}Soit $G$ un groupe topologique et soit $\mathscr{X}$ un espace rigide.
 Une famille $\cF$ de représentations de Banach de $G$ sur $\mathscr{X}$ est la donnée d'un faisceau sur $\mathscr{X}$ tel que 
\begin{itemize}
\item[(1)]pour tout ouvert affinoïde $V\subset\mathscr{X}$, $\cF(V)$ soit une $\cO(V)$-représentation de Banach de $G$;
\item[(2)]si $V'\subset V$ sont des ouverts affinoïdes de $\mathscr{X}$, l'application canonique 
\[\cF(V)\hat{\otimes}_{\cO(V)}\cO(V')\ra \cF(V')\] soit un isomorphisme de $\cO(V')$-représentations de Banach de $G$. 
\end{itemize}
\end{defn}
 
Soit $\I_0(p)=\{\gamma\in\GL_2(\Z_p)|\gamma\equiv(\begin{smallmatrix}*&*\\ 0 & *\end{smallmatrix})\mod p\}$ le sous-groupe d'Iwahori de $\GL_2(\Z_p)$; c'est un groupe $p$-adique. Rappelons que $V_{k,j+2}=\Sym^{k-2}V_p\otimes \Q_p(-j)$, où $V_p=\Q_p e_1\oplus\Q_p e_2$ est la représentation standard à droite de $\GL_2(\Z_p)$ donnée par la formule: $e_1*\gamma= ae_1+be_2, e_2*\gamma=ce_1+de_2$, si $\gamma=(\begin{smallmatrix}a& b\\ c& d\end{smallmatrix})\in \GL_2(\Z_p)$, est une représentation algébrique de $\I_0(p)$ muni du vecteur $e_1^{k-2}t^{-j}$ de plus haut poids $k$ pour le sous-groupe de Borel inférieur.
Dans le reste de ce paragraphe, si $1\leq j\in \N$, on construira une famille $\D_{0,\rho^{\mathbf{univ}}_{j}}$ de représentations de Banach de $\I_0(p)$ sur l'espace des poids $\sW$, qui est une interpolation $p$-adique en poids des représentations algébriques $V_{k,j+2}$ de $\I_0(p)$ (Il y a une autre interpolation en utilisant l'espace des fonctions analytiques par Chenevier \cite{Ch}.) La construction se découpe en trois étapes:

$\bullet$ On définit une action modulaire de $\I_0(p)$ sur $\Z_p$ par la formule 
$\gamma z=\frac{b+dz}{a+cz} $ si $\gamma=(\begin{smallmatrix}a& b\\ c& d\end{smallmatrix})\in\I_0(p)$, et une action continue de $\I_0(p)$ sur $\cC^0(\Z_p,\Z_p)$ à droite par la formule: 
$ f*\gamma(z)=f(\gamma z)$, si $\gamma\in\I_0(p)$ et $f\in \cC^0(\Z_p,\Z_p)$.  Ceci permet de lui associer un faisceau constant de représentations de Banach $\cC^0_{\sW}$ de $\I_0(p)$ sur $\sW$ défini par: si $U$ un ouvert affinoïde de $\sW$, on pose $\cC^0_{\sW}(U)=\cC^0(\Z_p, \cO(U))$.  Son anneau des sections globales est la limite projective $\plim_n\cC^0(W_n)\cong \plim_n\cC^0(\Z_p,\Q_p)\hat{\otimes}\cO(W_n)$.

$\bullet$ On note $\cC^0_\sW(\sW)^*$ le groupe des unités de l'anneau $\cC^0_\sW(\sW)$ des sections globales de $\cC^0_{\sW}$.
Soit $\rho:\I_0(p)\ra \cC^0_\sW(\sW)^{*}$ un $1$-cocycle de $\I_0(p)$ (i.e. $\rho(\gamma_1\gamma_2)=\rho(\gamma_1)(\rho(\gamma_2)*\gamma_1)$). On note $\cC^{0,\rho}_\sW$ la famille de représentations de Banach $\cC^0_{\sW}$ tordue par le $1$-cocycle $\rho$, c'est-\`a-dire, soit $U\subset\sW$ un ouvert affinoïde; si $f\in\cC^0_\sW(U)$ et si $f^{\rho}\in \cC^{0,\rho}_\sW(U)$ correspondant à $f$, 
on a $  f^{\rho}*\gamma=\rho(\gamma)( f*\gamma)^{\rho}$,  si $\gamma\in \I_0(p)$. Alors $\cC_\sW^{0,\rho}$ est encore une famille de repr\'esentations de Banach de $\I_0(p)$ sur l'espace des poids.
En particulier, on peut tordre  $\cC^0_\sW$ par un caractère de $\I_0(p)$ vu comme un $1$-cocycle. 

Pour tout entier $n\geq 1$, on note $r_n\geq 1$ le plus petit entier tel que $(p-1)p^{r_n}>n$. Si $a\in\Z_p^{*}$ et si $c\in p\Z_p$,  on définit une fonction sur $\Z_p$ à valeurs dans $\cO(\sW)$
\[f_{a,c}(z)=\kappa^{\mathbf{univ}}(a+cz).\]
\begin{lemma}\label{1-cocycle} La fonction 
$f_{a,c}(z)$ appartient à $\cC^0_\sW(\sW)$. De plus, elle est une unité de l'anneau $\cC^0_\sW(\sW)$.
\end{lemma}
\begin{proof}
On a un isomorphisme $\cO(W_n)\cong\Z_p[\Delta]\otimes C_n$,
où $C_n$ est le sous-anneau de $\Q_p[[T_1-1]]$ consistant des fonctions analytiques sur le disque $v_p(T_1-1)\geq \frac{1}{n}$ et $\Delta$ est un groupe cyclique d'ordre $p-1$ engendré par $X_1$. Dans la suite, on identifie $\cO(\sW_n)$ avec $\Z_p[\Delta]\otimes C_n$. La fonction $f_{a,c}$ s'écrit sous la forme $X_1^{\ord_\z\omega(a+cz)}T_1^{\log_{\gamma}(\langle a+cz\rangle)}$.

Comme $v_p(T_1-1)\geq \frac{1}{n}$, la fonction $\sum\binom{z}{k}(T_1-1)^k$ est dans $\cC^0(\Z_p,\cO(W_n))$. On en déduit que la série 
\[T_1^{\log_{\gamma}(\langle a+cz\rangle)}=\sum_{k\in\N}\binom{ \log_{\gamma}(\langle a+cz\rangle)}{k}(T_1-1)^k\]
est dans $\cC^0(\Z_p,\cO(W_n))$ pour tout $n$ et donc appartient à $\cC^0_\sW(\sW)$.

Comme $\kappa^{\mathbf{univ}}$ est un caractère sur $\Z_p^*$, la fonction $\kappa^{\mathbf{univ}}((a+cz)^{-1})\in \cC^0_{\sW}(\sW)$ est l'inverse de $f_{a,c}(z)$.
\end{proof}

Si $1\leq j\in\N$, on définit une fonction $\rho^{\mathbf{univ}}_{j}$ sur $ \I_0(p)$ à valeurs dans $\cC^{0}_{\sW}(\sW)^*$ par la formule : 
\[g=(\begin{smallmatrix}a&b\\ c& d \end{smallmatrix})\mapsto \rho^{\mathbf{univ}}_{j}(g)=f_{a,c}(z) (\det g)^{-j},  \text{ si } g=(\begin{smallmatrix}a&b\\c&d\end{smallmatrix})\in \I_0(p)  .\]
Si $g_1=(\begin{smallmatrix}a_1&b_1\\ c_1& d_1 \end{smallmatrix})$ et $g_2=(\begin{smallmatrix}a_2&b_2\\ c_2& d_2 \end{smallmatrix})$, on a 
 \[\rho^{\mathbf{univ}}_{j}(g_1g_2)=f_{a_1a_2+b_1c_2, c_1a_2+d_1c_2 }(z)\det(g_1g_1)^{-j}=\rho^{\mathbf{univ}}_{j}(g_2)*g_1\rho^{\mathbf{univ}}_{j}(g_1);\]
ceci implique que $\rho^{\mathbf{univ}}_{j}$ est un $1$-cocycle. 

Si $k\in \N$, l'application d'évaluation $\mathbf{Ev}_{k}$ induit un $1$-cocycle algébrique $\rho_{k,j}$ sur $\I_0(p)$ à valeurs dans $\cC^0(\Z_p, \Q_p)$ donné par la formule:
\[\rho_{k,j}(g)=\mathbf{Ev}_{k}( \rho^{\mathbf{univ}}_{j}(g))= (a+cz)^{k-2} (\det g)^{-j}.\]
On note $\cC^{0}_{k,j}$ la $\Q_p$-représentation de Banach $\cC^0(\Z_p,\Q_p)$ tordue par le $1$-cocycle $\rho_{k,j}$ et  $\cC^{0,\rho^{\mathbf{univ}}_{j}}_{\sW}$ la famille de représentations de Banach $\cC^0_\sW$  tordue par le $1$-cocycle $\rho_{j}^{\mathbf{univ}}$ de $\I_0(p)$. Alors, l'application d'évaluation $\mathbf{Ev}_{k}$ induit une application d'évaluation de $\cC^{0,\rho^{\mathbf{univ}}_{j}}_{\sW}$ dans  $\cC^{0}_{k,j}$, notée par $\mathbf{Ev}_{k,j}$. Un calcul direct montre le lemme suivant:
\begin{lemma}L'application $\mathbf{Ev}_{k,j}: \cC^{0,\rho^{\mathbf{univ}}_{j}}_{\sW}\ra \cC^{0}_{k,j}$ est $\I_0(p)$-équivariante.
\end{lemma}

$\bullet$
On définit une famille duale $\D_{0,\rho^{\mathbf{univ}}_{j}}$  de $\cC_{\sW}^{0,\rho^{\mathbf{univ}}_{j}}$ de représentations de Banach sur $\sW$: 
si $U\subset\sW$ est un ouvert affinoïde,  $\D_{0,\rho^{\mathbf{univ}}_{j}}(U)$ est défini par 
$\D_{0,\rho^{\mathbf{univ}}_{j}}(U)=(\cC_{\sW}^{0,\rho^{\mathbf{univ}}_{j}}(U))^{\vee}$. 

On note l'accouplement $\cO(U)$-linéaire $\I_0(p)$-équivariant
 \[\langle,\rangle: \D_{0,\rho^{\mathbf{univ}}_{j}}(U) \times\cC_{\sW}^{0,\rho^{\mathbf{univ}}_{j}}(U) \ra \cO(U) \] par l'intégration $\int_{\Z_p}$:
 \[\langle \gamma\cdot\mu,f\rangle=\int_{\Z_p} f (\gamma\cdot\mu)=\int_{\Z_p} (f*\gamma) \mu, \text{ si } f\in \cC_{\sW}^{0,\rho^{\mathbf{univ}}_j}(U), \gamma\in \I_0(p) \text{ et } \mu\in \D_{0,\rho^{\mathbf{univ}}_j}(U).\]
 On note $\D_{0,k,j}$ la représentation $\Q_p$-duale de $\I_0(p)$ de $\cC^0_{k,j}$. Remarquons qu'on a une inclusion de $\cC^0(\Z_p,\Q_p)$ dans $\cC^0(\Z_p, \cO(U))$.
 Si $k\in\Z$ appartient à $U$, alors
 l'application $\mathbf{Ev}_{k,j}$ induit une application, notée encore par $\mathbf{Ev}_{k,j}$, 
 \[\D_{0,\rho^{\mathbf{univ}}_{j}}(U)\ra \D_{0,k,j}; \int_{\Z_p} f\textbf{Ev}_{k,j}(\mu)= \textbf{Ev}_{k}(\int_{\Z_p} f \mu )\text{ pour } f\in \cC^0_{k,j}, \mu \in  
 \D_{0,\rho^{\mathbf{univ}}_{j}}(U). \]

\begin{lemma}\label{Ev}L'application $\mathbf{Ev}_{k,j}:\D_{0,\rho^{\mathbf{univ}}_{j}}(U)\ra \D_{0,k,j}$ est $\I_0(p)$-équivariante.
\end{lemma} 
\begin{proof}Si $f\in \cC^0_{k,j}, \mu \in  \D_{0,\rho^{\mathbf{univ}}_{j}}(U)$ et $g\in \I_0(p)$, on a 
\begin{equation*}
\begin{split}
\int_{\Z_p}f  \mathbf{Ev}_{k,j}(g\cdot\mu)&= \mathbf{Ev}_{k}(\int_{\Z_p} f  (g\cdot\mu))=\mathbf{Ev}_{k}(\int_{\Z_p}(f*g) \mu) 
\\&=  \int_{\Z_p}(f*g) \textbf{Ev}_{k,j}(\mu)= \int_{\Z_p}f   [g\cdot\textbf{Ev}_{k,j}(\mu)].
\end{split}
\end{equation*}  

\end{proof}
Le faisceau $\D_{0,\rho^{\mathbf{univ}}_{j}}$ nous donne une interpolation des représentations algébriques $V_{k,j+2}$ de $\I_0(p)$ pour $k\in \N$  au sens suivant:

La fonction $f:\Z_p\ra V_{k,j+2}$,  donnée par 
\[z\mapsto f(z)=(e_1+ze_2)^{k-2}t^{-j}, \] 
est une fonction continue sur $\Z_p$ à valeurs dans $V_{k,j+2}$. 
 On définit une application linéaire continue
$\pi_{k,j}:\D_{0,k,j}\ra V_{k,j+2}$ par l'intégration: $\mu\mapsto \int_{\Z_p} f(z)\mu $.

\begin{lemma}\label{iota}
L'application $\pi_{k,j}$ est $\I_0(p) $-équivariante.
\end{lemma}
\begin{proof}Ce lemme se démontre par le calcul suivant: si $\gamma=(\begin{smallmatrix} a&  b\\ c& d \end{smallmatrix})\in \I_0(p)$, alors 
$\rho_{k,j}(\gamma)=(a+cz)^{k-2}(\det\gamma)^{-j}$ et 
\begin{equation*}
\begin{split}
\int_{\Z_p}(e_1+ze_2)^{k-2} t^{-j}(\gamma\cdot\mu)&=\sum_{l=0}^{k_2}(\int_{\Z_p} \binom{k-2}{l}z^l (\gamma\cdot\mu)) e_1^{k-2-l}e_2^l t^{-j}
\\ &=\sum_{l=0}^{k_2}\left(\int_{\Z_p} \binom{k-2}{l}(\gamma z)^l \rho_{k,j}(\gamma) \mu\right) e_1^{k-2-l}e_2^l t^{-j}
\\ &=\int_{\Z_p}(ae_1+be_2+z(ce_1+de_2) )^{k-2}  (\det(\gamma)t)^{-j}\mu 
\\&= ( \int_{\Z_p}(e_1+ze_2)^{k-2} t^{-j}\mu)*\gamma.
\end{split}
\end{equation*}

\end{proof}
Si $U$ est un ouvert affinoïde de $\sW$ et si $k\in\Z$ appartient à  $U$, on définit une application de spécialisation
$\textbf{Sp}_{k,j}: \D_{0,\rho^{\mathbf{univ}}_{j}}(U)\ra \D_{0,k,j}\ra V_{k,j+2}$ en composant l'application d'évaluation $\textbf{Ev}_{k,j}$ et l'application $\pi_{k,j}$, qui est $\I_0(p)$-équivariante par les lemmes $\ref{Ev}$ et $\ref{iota}$.

\begin{remark}\label{intervect} La masse de Dirac $\delta_0$ en $0$ fournit une section globale $\nu_j$ de $\D_{0,\rho^{\mathbf{univ}}_{j}}(\sW)$ qui interpole le vecteur de plus haut poids $e_1^{k-2}t^{-j}$ dans $V_{k,j+2}$; en effet, 
$\int_{\Z_p}(e_1+ze_2)^{k-2}t^{-j}\delta_0=e_1^{k-2}t^{-j}$, ce qui se traduit par la formule $\textbf{Sp}_{k,j} (\nu_j)= e_1^{k-2} t^{-j} $. 
\end{remark}

\subsection{Famille de systèmes d'Euler de Kato sur l'espace des poids}
\subsubsection{Torsion à la Soulé}


Soit $G$ un groupe  localement profini, agissant continûment à droite sur un espace topologique localement profini $X$. Soit $V$ une $A$-représentation de Banach de $G$ à droite, où $A$ est une $\Q_p$-algèbre de Banach. On note $\cC^0(X, V)$ le $A$-module des fonctions continues sur $X$ à valeurs dans $V$ et $\fD_0(X,V)$ le $A$-module des mesures sur $X$ à valeurs dans $V$.  On munit $\cC^0(X,V)$ et $\fD_0(X,V)$ d'actions de $G$ à droite comme suit:
si $g\in G$, $x\in X$,  $\phi(x)\in\cC^0(X,V)$, et $\mu\in\fD_0(X,V)$, alors 
\[ \phi*g(x)= \phi(x*g^{-1})*g \text{ et } \int_X \phi(x)(\mu*g)=\left(\int_X(\phi*g^{-1}) \mu\right)*g. \]
Par contre, si $G$ agit continûment à gauche sur $X$ et $V$, on munit $\cC^0(X,\Z)$ et
$\fD_{0}(X,V)$ d'actions de $G$ \`a gauche comme suit:
si $g\in G, x\in X,\phi\in\cC^0(X,\Z), \mu\in\fD_{0}(X,V),$ alors
\begin{equation}\label{actiondisgauche} (g\cdot\phi)(x)=\phi(g^{-1}x) \text{ et } \int_{X}\phi(x)(g\cdot\mu)=g\cdot\bigl(\int_{X}(g^{-1}\cdot\phi)\mu\bigr).
\end{equation}

Soit $\mu\in \fD_0(X,\Z_p)$ et soit $M$ un $\Q_p$-espace de Banach muni d'un base de Banach $E=\{e_i\}_{i\in I}$. On a une inclusion de $\fD_0(X,\Z_p)\subset \fD_0(X,M)$. Soit $f\in \cC^0(X,M)$; l'intégration $\int_X f(x)\mu$ est définie comme suit: on décompose $f= \sum_{i\in I} f_i e_i$ sous la base de Banach $E$ de $M$ avec $f_i\in \cC^0(X, \Q_p)$ et l'intégration $\int_X f(x)\mu$ est donné par la formule  $\int_X f(x)\mu= \sum_{i\in I}e_i  \int_X f_i(x)\mu$.
\begin{prop}\label{TorsionSou}
Soit $f\in\cC^0(X,V)^G $. La multiplication d'une mesure $\mu\in\fD_0(X,\Z_p)$ par la fonction $f$ induit un morphisme $G$-équivariant à droite de $\fD_0(X,\Z_p)$ dans $\fD_0(X, V)$. 
\end{prop}
\begin{proof}
Il suffit de vérifier que c'est $G$-équivariant. Soit $\{e_i\}_{i\in I}$ une $\Q_p$-base de Banach de $V$. On décompose $f(x)$ sous la forme $\sum_{i\in I} f_i(x)e_i$ avec $f_i(x)\in\cC^0(X, \Q_p)$. Soit $\phi$ une fonction continue sur $X$ à valeurs dans $\Q_p$. 
Si on considère $\mu*g \in \fD_0(X,\Z_p)$, alors l'intégration $\int_X\phi(x)(f(x)\otimes \mu*g)$ s'écrit sous la forme:
\begin{equation*}
 \int_X\phi(x)(f\otimes(\mu*g))= \sum_{i\in I}e_i\int_X \phi(x*g) f_i( x*g)\mu=\int_{X} \phi(x*g)f(x*g)\mu  
\end{equation*}
Par ailleurs, l'intégration $\int_X\phi(x)(f\otimes\mu)*g$ s'écrit sous la forme
\begin{equation*}
\begin{split}
\int_X\phi(x) (f\otimes\mu)*g
&=(\int_X\phi(x*g) (f(x)\otimes\mu))*g=\int_X\phi(x*g)((f(x)*g)\otimes\mu).
\end{split}
\end{equation*}
Comme $f\in\cC^0(X,V)^G$, on a $f(x*g)=f(x)*g$. Ceci permet de conclure. 

\end{proof}
Si l'action de groupe est à gauche, on a aussi un énoncé analogue.  

\subsubsection{Application au système d'Euler de Kato }
On note $\hat{\Gamma}_0(p)=\GL_2(\hat{\Z})^{]p[}\times \I_0(p)$ et  $\tilde{\Gamma}_0(p)$ l'image inverse de $\hat{\Gamma}_0(p)$ dans $\Pi_{\Q}$. 
On note $\bM_2^{(p)}=\bM_2(\Q\otimes\hat{\Z})^{]p[}\times \I_0(p)$. Si $x\in \bM_2^{(p)}$, on note $x_p$ sa composante à la place $p$.
 
Rappelons que l'on a construit une section globale $\nu_j\in \D_{0, \rho^{\mathbf{univ}}_j}(\sW)$ qui interpole les vecteurs de plus haut poids $e_1^{k-2}t^{-j}$ de représentations algébriques 
$V_{k,j+2}$ de $\I_0(p)$. Pour tout ouvert affinoïde $U$ de $\sW$,  la fonction $x\mapsto x_p\cdot \nu_j$ est une fonction continue sur $\bM_2^{(p)}$ à valeurs dans $\D_{0,\rho^{\mathbf{univ}}_j}(U)$ invariante sous l'action de $\tilde{\Gamma}_0(p)$. De manière explicite, si $x_p=(\begin{smallmatrix}a&b\\ c &d\end{smallmatrix})\in\I_0(p)$, on a  $x_p\cdot\nu_j=\rho^{\mathbf{univ}}_{j}(x_p)\delta_{\frac{b}{a}}$, où $\delta_{\frac{b}{a}}$ est la masse de Dirac en $\frac{b}{a}$. 
Elle interpole la fonction continue $x\mapsto (e_1^{k-2}t^{-j})*x_p$ sur $\bM_2^{(p)}$ à valeurs dans $V_{k,j+2}$ via l'application de spécialisation.

Considérons la multiplication d'une mesure $\mu\in \fD_0(\bM_2^{(p)},\Z_p(2))$ par la fonction $x\mapsto x_p\cdot\nu_j$ (resp. $x\mapsto (e_1^{k-2}t^{-j})*x_p$). Ceci nous donne une mesure $(x_p\cdot \nu_j)\otimes \mu$ (resp. $(e_1^{k-2}t^{-j})*x_p\otimes\mu$) sur $\bM_2^{(p)}$ à valeurs dans $\D_{0, \rho^{\mathbf{univ}}_j}(U)(2)$ (resp. $V_{k,j+2}$).

Les fonctions $x\mapsto x_p\cdot\nu_j$ et  $x\mapsto(e_1^{k-2}t^{-j})*x_p$ satisfont la condition "invariante sous l'action de $\tilde{\Gamma}_0(p)$" par construction; on déduit le corollaire ci-dessous de la proposition $\ref{TorsionSou}$.
\begin{coro}\label{sou}
\begin{itemize}
\item[(1)] La multiplication d'une mesure  $\mu\in \fD_0(\bM_2^{(p)},\Z_p(2))$ par la fonction $x\mapsto(e_1^{k-2}t^{-j})*x_p$ induit un morphisme $\tilde{\Gamma}_0(p)$-équivariant à droite de $\Z_p$-modules de  $\fD_0(\bM_2^{(p)},\Z_p(2))$ dans $ \fD_0(\bM_2^{(p)},V_{k,j})$;
\item[(2)] La multiplication d'une mesure  $\mu\in \fD_0(\bM_2^{(p)},\Z_p(2))$ par la fonction $x\mapsto x_p\cdot\nu_j$ induit un morphisme $\tilde{\Gamma}_0(p)$-équivariant à gauche de $\Z_p$-modules de  $\fD_0(\bM_2^{(p)},\Z_p(2))$ dans $ \fD_0(\bM_2^{(p)},\D_{0, \rho^{\mathbf{univ}}_{j}}(\sW)(2))$.
\end{itemize}
\end{coro}

 D'après Kato \cite{KK} et Colmez\cite{PC1} (voir aussi \cite{Wang}), on dispose d'une mesure  $z_{\mathbf{Kato},c,d}$ construite à partir des unités de Siegel et appartenant à  $\rH^2(\Pi_\Q, \fD_0(\bM_2(\Q\otimes\hat{\Z})^{(p)},\Z_p(2)))$. L'application de restriction de $\Pi_\Q$ à $\tilde{\Gamma}_0(p)$ et l'application de restriction de $\fD_0(\bM_2(\Q\otimes\hat{\Z})^{(p)},\Z_p(2))$ dans $\fD_0(\bM_2^{(p)},\Z_p(2))$ nous fournissent une application 
 \[\rH^2(\Pi_\Q, \fD_0(\bM_2(\Q\otimes\hat{\Z})^{(p)},\Z_p(2)))\ra \rH^2(\tilde{\Gamma}_0(p), \fD_0(\bM_2^{(p)},\Z_p(2))).\]
 On note l'image de $z_{\mathbf{Kato},c,d}$ sous cette application de la même manière. 
 D'après le corollaire $\ref{sou}$, quel que soit $U\subset\sW$ ouvert affinoïde, la multiplication par  la fonction $x\mapsto x_p\cdot\nu_j$ (resp. $x\mapsto(e_1^{k-2}t^{-j})*x_p$) induit un morphisme naturel:
 \begin{equation*}
 \begin{split}
 &\rH^2(\tilde{\Gamma}_0(p), \fD_0(\bM_2^{(p)},\Z_p(2)))\ra   \rH^2(\tilde{\Gamma}_0(p), \fD_0(\bM_2^{(p)},\D_{0, \rho^{\mathbf{univ}}_{j}}(U)(2) ) )\\
 (\text{resp.} &\rH^2(\tilde{\Gamma}_0(p), \fD_0(\bM_2^{(p)},\Z_p(2)))\ra   \rH^2(\tilde{\Gamma}_0(p), \fD_0(\bM_2^{(p)},V_{k,j} ) )  ).
\end{split}
 \end{equation*}
 On définit, 
  \begin{equation*}
 \begin{split}
 &z_{\mathbf{Kato},c,d}(\nu_j)=(x_p\cdot\nu_j)\otimes z_{\mathbf{Kato},c,d} \in \rH^2(\tilde{\Gamma}_0(p), \fD_0(\bM_2^{(p)},\D_{0, \rho^{\mathbf{univ}}_{j}}(U)(2)) )\\
 (\text{resp.} &z_{\mathbf{Kato},c,d}(k,j)=(e_1^{k-2}t^{-j})*x_p\otimes z_{\mathbf{Kato},c,d} \in \rH^2(\tilde{\Gamma}_0(p), \fD_0(\bM_2^{(p)},V_{k,j}) ) ),
\end{split}
 \end{equation*}
 où $\tilde{\Gamma}_0(p)$ agit sur $\D_{0, \rho^{\mathbf{univ}}_{j}}(U)(2) $ à travers son quotient $\I_0(p)$. Comme l'application de spécialisation $\textbf{Sp}_{k,j}: \D_{0,\rho^{\mathbf{univ}}_{j}}(U)\ra V_{k,j+2}$ est $\tilde{\Gamma}_0(p)$-équivariante, pour tout ouvert $U\subset \sW$,  elle induit une application de spécialisation 
 \[\textbf{Sp}_{k,j}: \rH^2(\tilde{\Gamma}_0(p), \fD_0(\bM_2^{(p)},\D_{0, \rho^{\mathbf{univ}}_{j}}(U)(2)) ) \ra \rH^2(\tilde{\Gamma}_0(p), \fD_0(\bM_2^{(p)},V_{k,j}) ).\]
 On déduit le théorème suivant de la construction de $\nu_j$ (c.f. Remarque $\ref{intervect}$) et du corollaire $\ref{sou}$, et il faut faire attention que l'application de spécialisation transfère une action de groupe à gauche en celle à droite: 
 \begin{theo}\label{fam-poids}
 Si $1\leq j\in\N$, alors pour tout entier $k\geq j+1$, on a \[\textbf{Sp}_{k,j}(z_{\mathbf{Kato},c,d}(\nu_j))=z_{\mathbf{Kato},c,d}(k,j) .\]
\end{theo}

\section{Une famille de distributions d'Eisenstein}
\subsection{Séries d'Eisenstein-Kronecker}
\subsubsection{Séries d'Eisenstein-Kronecker}
Les r\'esultats dans ce paragraphe peuvent se trouver dans le livre de Weil $\cite{Weil}$, voir aussi $\cite{PC1}$, $\cite{Wang}$.

\begin{defn}Si $(\tau,z)\in\cH\times\C$, on pose $q=e^{2i\pi\tau}$ et
$q_z=e^{2i\pi z}$. On introduit l'op\'erateur
$\partial_{z}:=\frac{1}{2i\pi}\frac{\partial}{\partial
z}=q_z\frac{\partial}{\partial q_z}$. On pose aussi $e(a)=e^{2i\pi
a}$. Si $k\in\N$, $\tau\in\cH$, et $z,u\in\C$, la s\'erie d'Eisenstein-Kronecker est
\[\rH_k(s,\tau,z,u)=\frac{\Gamma(s)}{(-2i\pi)^k}(\frac{\tau-\bar{\tau}}{2i\pi})^{s-k}\sideset{}{'}\sum_{\omega\in\Z+\Z\tau}\frac{\overline{\omega+z}^k}{|\omega+z|^{2s}}e(\frac{\omega\bar{u}-u\bar{\omega}}{\tau-\bar{\tau}}),\]
qui converge pour $\Re(s)>1+\frac{k}{2}$, et poss\`ede un
prolongement
 m\'eromorphe \`a tout le plan complexe avec des p\^oles simples
en $s=1$ ( si $k=0$ et $u\in\Z+\Z\tau$) et $s=0$ ( si $k=0$
et $z\in\Z+\Z\tau$ ). Dans la formule ci-dessus $\sideset{}{'}\sum$
signifie ( si $z\in\Z+\Z\tau$ ) que l'on supprime le terme
correspondant \`a $\omega=-z$. De plus, elle v\'erifie l'\'equation fonctionnelle:
\[\rH_k(s,\tau,z,u)=e(\frac{z\bar{u}-u\bar{z}}{\tau-\bar{\tau}})\rH_k(k+1-s,\tau,u,z).\]
\end{defn}
Si $k\geq 1$, on d\'efinit les fonctions suivantes:
\begin{align*}
E_k(\tau,z)=\rH_k(k,\tau,z,0), F_k(\tau,z)=\rH_k(k,\tau,0,z).
\end{align*}
Les fonctions $E_k(\tau,z)$ et $F_k(\tau,z)$ sont p\'eriodiques
en $z$ de p\'eriode $\Z\tau+\Z$. De plus, on a:
\[E_{k+1}(\tau,z)=\partial_z E_k(\tau,z), \text{ si } k\in\N \text{ et } E_0(\tau,z)=\log|\theta(\tau,z)| \text{ si }z\notin \Z+\Z\tau,\]
o\`u $\theta(\tau,z)$ est donn\'ee par le produit infini:
\[\theta(\tau,z)=q^{1/12}(q_z^{1/2}-q_z^{-1/2})\sideset{}{}\prod_{n\geq 1}((1-q^n q_z)(1-q^nq_z^{-1})).\]
On note $\Delta=(\partial_z
\theta(\tau,z)|_{z=0})^{12}=q\sideset{}{}\prod_{n\geq
1}(1-q^n)^{24}$ la forme modulaire de poids $12$.

Soient $(\alpha,\beta)\in (\Q/\Z)^2$ et $(a,b)\in\Q^2$ qui a pour image
$(\alpha,\beta)$ dans $(\Q/\Z)^2$. Si $k=2$ et $(\alpha,\beta)\neq(0,0)$, ou si $k\geq 1$ et $k\neq 2$,  on d\'efinit:
\[ E_{\alpha,\beta}^{(k)}=E_k(\tau,a\tau+b) \text{ et } F_{\alpha,\beta}^{(k)}=F_k(\tau,a\tau+b).\]
La s\'erie $H_2(s,\tau,0,0)$ converge pour $\Re(s)>2$, mais pas pour $s=2$;
si $k=2$ et $(\alpha,\beta)=(0,0)$, on d\'efinit \[E_{0,0}^{(2)}=F_{0,0}^{(2)}:=\lim_{s\ra 2}H_2(s,\tau,0,0).\]
\begin{lemma}\label{lemma15}Les fonctions $E_{\alpha,\beta}^{(k)}, F_{\alpha,\beta}^{(k)} $ satisfont les relations de distribution suivantes, quel que soit l'entier $f\geq 1$:
\begin{align}\label{rel}
\sideset{}{}\sum_{f\alpha'=\alpha,f\beta'=\beta}E_{\alpha',\beta'}^{(k)}=f^{k}E_{\alpha,\beta}^{(k)}&\text{\
\ et }
&\sideset{}{}\sum_{f\alpha'=\alpha,f\beta'=\beta}F_{\alpha',\beta'}^{(k)}=f^{2-k}F_{\alpha,\beta}^{(k)}\\
\label{rel1}
\sideset{}{}\sum_{f\beta'=\beta}E_{\alpha,\beta'}^{(k)}(\frac{\tau}{f})=f^{k}E_{\alpha,\beta}^{(k)}&\text{\
\  et }
&\sideset{}{}\sum_{f\beta'=\beta}F_{\alpha,\beta'}^{(k)}(\frac{\tau}{f})=fF_{\alpha,\beta}^{(k)}.
\end{align}
\end{lemma}
\begin{prop}\label{EF}
\begin{itemize}
\item[(1)]$E_{0,0}^{(2)}=F_{0,0}^{(2)}=\frac{-1}{24}E_2^{*},$
o\`u
\[E_2^{*}=\frac{6}{i\pi(\tau-\bar{\tau})}+1-24\sideset{}{}\sum_{n=1}^{+\infty}\sigma_1(n)q^n\]
est la s\'erie d'Eisenstein non holomorphe de poids $2$ habituelle.

\item[(2)] Si $N\alpha=N\beta=0$, alors
\begin{itemize}
\item[(a)] $\tilde{E}_{\alpha,\beta}^{(2)}=E_{\alpha,\beta}^{(2)}-E_{0,0}^{(2)}\in\cM_2(\Gamma_N,\Q(\zeta_N))$
et $E_{\alpha,\beta}^{(k)}\in\cM_k(\Gamma_N,\Q(\z_N))$ si $k\geq 1
$ et $k\neq 2$.
\item[(b)] $F_{\alpha,\beta}^{(k)}\in\cM_k(\Gamma_N,\Q(\z_N))$ si $k\geq
1, k\neq 2$ ou si $k=2,(\alpha,\beta)\neq(0,0)$.
\end{itemize}
\end{itemize}
\end{prop}
\begin{prop}\label{eis} Si $\gamma=\bigl(\begin{smallmatrix}a & b\\ c &
d\end{smallmatrix}\bigr)\in\GL_2(\hat{\Z}), k\geq 1$ et
$(\alpha,\beta)\in(\Q/\Z)^2$, on a:
\[E_{\alpha,\beta}^{(k)}*\gamma=E_{a\alpha+c\beta,b\alpha+d\beta}^{(k)}\text{ et } F_{\alpha,\beta}^{(k)}*\gamma=F_{a\alpha+c\beta,b\alpha+d\beta}^{(k)}.\]
\end{prop}

\begin{defn}Soit $A$ un anneau. Une série de Dirichlet formelle à coefficients dans $A$ est une série de la forme $\sum_{n\in\frac{1}{N}\N}a_n n^{-s}$, où $N\in\N$ et $(a_n)$ désigne une suite d'éléments dans $A$. On note $\Dir(A)$ le $A$-module des s\'eries de Dirichlet formelles dont les coefficients sont dans $A$. 

\end{defn}

Soit $\alpha\in\Q/\Z$. On d\'efinit les s\'eries de Dirichlet formelles $\z(\alpha, s)$ et
$\z^{*}(\alpha,s)$, appartenant \`a $\Dir(\Q^{\cycl})$, par les formules:
\[\z(\alpha,s)=\sum\limits_{\substack{n\in\Q_{+}^{*}\\ n=\alpha \mod\Z}}n^{-s} \text{ et } \z^{*}(\alpha,s)=\sum_{n=1}^{\infty}e^{2i\pi\alpha n}n^{-s}.\]
\begin{remark}Les séries de Dirichlet $\z(\alpha,s)$ et $\z^*(\alpha, s)$ convergent pour $\Re(s)>1$ et elles se prolongent analytiquement sur le plan complexe en fonctions méromorphes avec au plus un pôle simple en $s=1$. Ces prolongements définissent la fonction zêta de Hurwitz $\z(\alpha,s)$ et la fonction L de Dirichlet $\z^*(\alpha, s)$ respectivement.
\end{remark}
La proposition suivante d\'ecrit le $q$-d\'eveloppement d'une série d'Eisenstein (c.f. \cite{KK}).
\begin{prop}\label{q-deve}
\begin{itemize}
\item[(i)]Si $k\geq 1, k\neq 2$ et $\alpha, \beta\in\Q/\Z$, alors le $q$-d\'eveloppement $\sum\limits_{n\in\Q_{+}}a_nq^n$ de $E_{\alpha,\beta}^{(k)}$ est donn\'e par
\begin{equation*}
\sum_{n\in\Q_{+}^{*}}\frac{a_n}{n^s}=\z(\alpha,s)\z^{*}(\beta,s-k+1)+(-1)^{k}\z(-\alpha,s)\z^{*}(-\beta,s-k+1).
\end{equation*}
De plus, on a:

Soit $k\neq 1 $. $a_0=0$ $($resp. $a_0=\z^{*}(\beta,1-k))$ si $\alpha\neq 0$ $($resp. $\alpha=0)$.\\
Soit $k=1$. On a $a_0=\z(\alpha,0) ~($resp.
$a_0=\frac{1}{2}(\z^{*}(\beta,0)-\z^{*}(-\beta,0)))$ si $\alpha\neq
0~($resp. $\alpha=0)$.
\item[(ii)]Si $k\geq 1$ et $\alpha, \beta\in\Q/\Z~($si $k=2$, $(\alpha,\beta)\neq (0,0))$, alors le $q$-d\'eveloppement
$\sum\limits_{n\in\Q_{+}}a_nq^n$ de $F_{\alpha,\beta}^{(k)}$ est donn\'e par
\begin{equation*}
\sum_{n\in\Q_{+}^{*}}\frac{a_n}{n^s}=\z(\alpha,s-k+1)\z^{*}(\beta,s)+(-1)^{k}\z(-\alpha,s-k+1)\z^{*}(-\beta,s).
\end{equation*}
De plus, 
\begin{itemize}
\item[]$a_0=\z(\alpha,1-k)$, la valeur spéciale de la fonction zêta de Hurwitz, si $k\neq 1$;
\item[] $a_0=\z(\alpha,0)~($resp.
$a_0=\frac{1}{2}(\z^{*}(\beta,0)-\z^{*}(-\beta,0)))$ si $\alpha\neq
0~($resp. $\alpha=0)$ et si $k=1$.
\end{itemize}
\end{itemize}
\end{prop}

\subsubsection{Les distributions $z_{\mathbf{Eis}}(k,j)$ et $z_{\mathbf{Eis},c,d}(k,j)$}
Dans ce paragraphe, on rappelle la construction des distributions d'Eisenstein $z_{\mathbf{Eis}}(k,j)$ et $z_{\mathbf{Eis,c,d}}(k,j)$ construites à partir d'un produit de deux séries d'Eisenstein dans $\cite{{PC1}}$ et $\cite{Wang}$. 

 Soient $X=(\Q\otimes\hat{\Z})^2, G=\GL_2(\Q\otimes\hat{\Z})$ et $V=\cM_k^{\con}(\Q^{\cycl})$.
Alors $X$ est un espace topologique localement profini et $G$ est un groupe localement profini agissant contin\^ument \`a droite sur $X$ par la multiplication de matrices.

L'action à droite de $G$ sur $V$, not\'ee par $*$, provient de l'action de
$\Pi_{\Q}^{'}$ sur $V$, et l'action de $\Pi_{\Q}$
se factorise \`a travers son quotient $\GL_2(\hat{\Z})$. Comme tout
$\gamma\in\GL_2(\Q\otimes\hat{\Z})$ peut s'\'ecrire sous la
forme $\gamma=g_1\bigl(\begin{smallmatrix}r_0 & 0\\ 0 &
r_0\end{smallmatrix}\bigr)\bigl(\begin{smallmatrix}1 & 0\\ 0 &
e\end{smallmatrix}\bigr)g_2$ avec
$g_1,g_2\in\GL_2(\hat{\Z}),r_0\in\Q^{*}_{+}$ et $e$ un entier $\geq 1$, il suffit de donner respectivement les formules pour $\gamma\in
\GL_2(\hat{\Z})$ ou $\gamma=\bigl(\begin{smallmatrix}r_0&0\\0&r_0\end{smallmatrix}\bigr)$
ou $\gamma=\bigl(\begin{smallmatrix}1&0\\0&e\end{smallmatrix}\bigr)$
:
\begin{itemize}
\item[(1)]Comme
$\bigl(\begin{smallmatrix}r_0 & 0\\ 0 & r_0\end{smallmatrix}\bigr)$
et $\bigl(\begin{smallmatrix}1 & 0\\ 0 & e\end{smallmatrix}\bigr)$
apparaissent dans $\GL_2(\Q)_{+}$, ses actions sont données par la formule (\ref{etoi})  dans \S \ref{notation}. 
\item[(2)]Si $\gamma\in\GL_2(\hat{\Z})$, en utilisant la d\'ecomposition $\GL_2(\hat{\Z})=\cup_{d\in\hat{\Z}^*}\SL_2(\hat{\Z})(\begin{smallmatrix}1&0\\0&d\end{smallmatrix})$, on d\'ecompose l'action de $\gamma$ en deux parties.
Comme on est en poids $k$, l'action de $\SL_2(\hat{\Z})$ est l'action $|_k$. 
L'action de $(\begin{smallmatrix}1&0\\0&d\end{smallmatrix})$ est via un rel\`evement $\sigma_d$ dans $\cG_{\Q}$ agissant sur les coefficients 
du $q$-développement. Dans le cas des séries d'Eisenstein, la proposition $\ref{eis}$ explicite les formules. 
\end{itemize}
On a le théorème suivant, qui traduit les relations de distribution pour les séries d'Eisenstein en terme de distributions (c.f. \cite[Théorème 2.12]{Wang}):  

\begin{theo}Si $k\geq 1$, il existe une distribution alg\'ebrique $z_{\mathbf{Eis}}(k)$ $($resp. $z'_{\mathbf{Eis}}(k))\in
\fD_{\alg}((\Q\otimes\hat{\Z})^2,\cM_k^{\con}(\Q^{\cycl}))$ v\'erifiant:
quels que soient $r\in\Q^{*}$ et $(a,b)\in\Q^2$, on a
\begin{align*}
\int_{(a+r\hat{\Z})\times(b+r\hat{\Z})}z_{\mathbf{Eis}}(k)=r^{-k}E^{(k)}_{r^{-1}a,r^{-1}b}\\
\bigl( resp. \int_{(a+r\hat{\Z})\times(b+r\hat{\Z})}z_{\mathbf{Eis}}(k)=r^{-k}\tilde{E}^{(k)}_{r^{-1}a,r^{-1}b} \text{ si } k=2 \bigr)\\
\int_{(a+r\hat{\Z})\times(b+r\hat{\Z})}z'_{\mathbf{Eis}}(k)=r^{k-2}F^{(k)}_{r^{-1}a,r^{-1}b}.
\end{align*}
De plus, si $\gamma\in\GL_2(\Q\otimes\hat{\Z})$, alors
\[z_{\mathbf{Eis}}(k)*\gamma=z_{\mathbf{Eis}}(k) \text{ et } z'_{\mathbf{Eis}}(k)*\gamma=|\det\gamma|^{1-k}z'_{\mathbf{Eis}}(k).\]
\end{theo}
On peut identifier
$(\Q\otimes\hat{\Z})^2\times(\Q\otimes\hat{\Z})^2$ avec
$\bM_2(\Q\otimes\hat{\Z})$ via 
$((a,b),(c,d))\mapsto\bigl(\begin{smallmatrix}a & b\\ c &
d\end{smallmatrix}\bigr)$. En utilisant le fait que le produit de deux formes modulaires de poids $i$ et $j$ est une forme modulaire de poids $i+j$,
on définit pour $k\geq 2$ et $1\leq j\leq k-1$, 
\[z_{\mathbf{Eis}}(k,j)=\frac{1}{(j-1)!}z'_{\mathbf{Eis}}(k-j)\otimes z_{\mathbf{Eis}}(j)\in\fD_{\alg}(\bM_2(\Q\otimes\hat{\Z}),\cM_k(\overline{\Q})).\]

La distribution $z_{\mathbf{Eis,c,d}}(k,j)$ dans $\cite{Wang}$ est construite à partir de combinaisons linéaires de séries d'Eisenstein. On rappelle sa construction ci-dessous. 

Soit $\langle\langle\cdot\rangle\rangle:\Z_p^{*}\ra\hat{\Z}^*$ l'inclusion naturelle  en envoyant $x$ sur $\langle\langle x\rangle\rangle=(1,\cdots, x,1,\cdots)$, o\`u $x$ est \`a la place $p$. Consid\'erons l'inclusion de $\hat{\Z}^{*}$ dans $\GL_2(\hat{\Z})$ en envoyant $d$ sur $(\begin{smallmatrix}d&0\\ 0& d\end{smallmatrix})$. D'apr\`es la proposition \ref{eis}, cela d\'efinit une action de $d\in\hat{\Z}^{*}$ sur les s\'eries d'Eisenstein par les formules: si $k\geq 1$ et $(\alpha,\beta)\in (\Q/\Z)^2$, on a
\begin{equation}
d\cdot E_{\alpha,\beta}^{k}=E_{d\alpha, d\beta}^{(k)}=E_{\alpha,\beta}^{(k)}*(\begin{smallmatrix}d&0\\ 0& d\end{smallmatrix}) \text{ et } d\cdot F_{\alpha,\beta}^{(k)}=F_{d\alpha,d\beta}^{(k)}=F_{\alpha,\beta}^{(k)}*(\begin{smallmatrix}d&0\\ 0& d\end{smallmatrix}),
\end{equation}
o\`u l'action de $*$ est celle de $\GL_2(\hat{\Z})$ sur les s\'eries d'Eisenstein.

Consid\'erons l'injection de $\cM_k^{\con}(\ol{\Q})$ dans $\cM_k^{\con}(\ol{\Q}_p)$.
 On peut d\'efinir une variante de s\'eries d'Eisenstein \`a coefficients dans $\ol{\Q}_p$ comme ci-dessous:  si $c\in\Z_p^{*}$, on pose
 \begin{equation}\label{variant}
 \begin{split}
 E^{(k)}_{c,\alpha,\beta}&=\left\{
 \begin{aligned}&c^2E^{(k)}_{\alpha,\beta}-c^{k}E_{\langle\langle c\rangle\rangle\alpha,\langle\langle c \rangle\rangle\beta}^{(k)}; \text{ si } k\geq 1 \text{ et } k\neq 2,
 \\
 &c^2\tilde{E}_{\alpha,\beta}^{(2)}-c^2\tilde{E}_{\langle\langle c\rangle\rangle\alpha,\langle\langle c \rangle\rangle\beta}^{(2)}; \text{ si }k=2;
         \end{aligned} \right. \\
 F^{(k)}_{c,\alpha,\beta}&=c^2F_{\alpha,\beta}^{(k)}-c^{2-k}F_{\langle\langle c \rangle\rangle\alpha, \langle\langle c \rangle\rangle\beta}^{(k)} \text{ si } k\geq 1 \text{ et } k\neq 2, \text{ ou si } (\alpha,\beta)\neq (0,0)\text{ et } k=2.
 \end{split}
 \end{equation}
Elles sont des combinaisons lin\'eaires de s\'eries d'Eisenstein. 
\begin{prop}\label{eiskj}$($\cite{Wang} Proposition 2.14$)$ 
Soit $c\in\Z_p^{*}$. Si $k\geq 1$, il existe une distribution alg\'ebrique $z_{\mathbf{Eis}, c}(k)$ $($resp. $z^{'}_{\mathbf{Eis},c})$ $\in \fD_{\alg}((\Q\otimes\hat{\Z})^2, \cM_k^{\con}(\Q_p^{\cycl}))$ v\'erifiant: quel que soient $r\in\Q^*$ et $(a,b)\in\Q^2$, on a
\begin{equation*}
\begin{split}
\int_{(a+r\hat{\Z})\times(b+r\hat{\Z})}z_{\mathbf{Eis},c}(k)=r^{-k}E^{(k)}_{c,r^{-1}a,r^{-1}b}
(\text{resp.} \int_{(a+r\hat{\Z})\times(b+r\hat{\Z})}z^{'}_{\mathbf{Eis},c}(k)=r^{k-2}F^{(k)}_{c,r^{-1}a,r^{-1}b}.)
\end{split}
\end{equation*}
De plus, si $\gamma\in\GL_2(\Q\otimes \hat{\Z})$, alors
on a \[z_{\mathbf{Eis},c}(k)*\gamma=z_{\mathbf{Eis},c}(k) \text{ et } z_{\mathbf{Eis},c}^{'}*\gamma=|\det\gamma|^{1-k}z_{\mathbf{Eis},c}^{'}(k).\]
\end{prop}

Soient $c,d\in\Z_p^{*}$. Si $k\geq 2$ et $1\leq j\leq k-1$, on définit une distribution $z_{\mathbf{Eis},c,d}(k,j)$ appartient \`a $\fD_{\alg}(\bM_2(\Q\otimes\hat{\Z}),\cM_k(\Q^{\cycl}_p))$ par la formule:
   \[z_{\mathbf{Eis},c,d}(k,j)=\frac{1}{(j-1)!}z'_{\mathbf{Eis},c}(k-j)\otimes z_{\mathbf{Eis},d}(j).\] 
 Si $\gamma\in\GL_2(\Q\otimes\hat{\Z})$, on a $z_{\mathbf{Eis},c,d}(k,j)_{|_k}\gamma=|\det\gamma|^{j-1}z_{\mathbf{Eis},c,d}(k,j).$

\subsection{Famille $p$-adique de séries d'Eisenstein}\label{construction_fam_Eis}
\subsubsection{La fonction zêta $p$-adique }

Soit $\alpha\in\Q/\Z$; on note $\{\alpha\}$ le plus petit nombre rationnel positif tel que $\{\alpha\}\equiv \alpha\mod \Z$, et on note $\ord_p(\alpha)$ (resp. $\ord\alpha$) l'ordre de $\alpha$ dans le groupe $\Q_p/\Z_p$ (resp. $\Q/\Z$ ).

Dans la suite, on suppose que $\alpha\in \Q/\Z$ vérifie $\ord_p(\alpha)\geq p$.  
Si $x\in \Z_p$, on note $\ord(\alpha)(\{\alpha\}+x)$ par $x_{\alpha}$, qui appartient à $\Z_p^*$. 
 
Si $c\in\Z_p^*$, on note $\mu_c$ la mesure sur $\Z_p$ 
  dont la transformée d'Amice $\cA_{\mu_c}$ est donnée par $\frac{c^2}{T}-\frac{c}{(1+T)^{c^{-1}}-1}$. Par ailleurs, le caractère universel de l'espace des poids $\kappa^{\mathbf{univ}}:\Z_p^*\ra \cO(\sW)$ est une fonction continue sur $\Z_p^*$ à valeurs dans $\cO(\sW)$.  Si $1\leq j\in \Z$, on définit la fonction zêta $p$-adique de Hurwitz comme la fonction holomorphe sur $\sW$  
 \[\z_{p,c}(\kappa^{\mathbf{univ}},\alpha,j)=-\int_{\Z_p}\kappa^{\mathbf{univ}}(x_\alpha) \langle x_\alpha\rangle^{1-j}\mu_c.\] 
 
 \begin{lemma}
Si $ j,k$ sont deux entiers $\geq 1$, alors on a 
 \[ \mathbf{Ev}_{k+j}(\z_{p,c}(\kappa^{\mathbf{univ}},\alpha,j))=\ord(\alpha)^{k-1}\omega(0_{\alpha})^{j-1}(c^2\z(\alpha, 1-k)-c^{2-k}\z(\langle\langle c\rangle\rangle\alpha, 1-k)). \]
 \end{lemma}
 \begin{proof}De la définition, on a bien $\textbf{Ev}_{k+j}(\z_{p,c}(\kappa^{\mathbf{univ}},\alpha,j))=-\omega(0_{\alpha})^{j-1}\int_{\Z_p}x_{\alpha}^{k-1}\mu_c$.
 On se ramène à vérifier que $\int_{\Z_p}x_{\alpha}^k\mu_c= -\ord(\alpha)^k (c^2\z(\alpha, -k)-c^{1-k}\z(\langle\langle c\rangle\rangle\alpha,-k))$.
 
 On utilise la transformée d'Amice $\cA_{\mu_c}$ de $\mu_c$ pour calculer l'intégration $\int_{\Z_p}x_\alpha^{k} \mu_c$. On suppose que $c^{-1}\in\N$. Si on pose $T=e^t-1$, on a alors 
 \begin{equation*}
 \begin{split}
 \int_{\Z_p}x_\alpha^{k} \mu_c(x)&=\ord(\alpha)^k(\frac{d}{dt})^{k}(\int_{\Z_p}e^{(\{\alpha\}+x)t}\mu_c)|_{t=0}\\
 &=\ord(\alpha)^k(\frac{d}{dt})^{k}(\int_{\Z_p}(T+1)^{(\{\alpha\}+x)}\mu_c)|_{t=0}\\ &=\ord(\alpha)^k(\frac{d}{dt})^{k}((T+1)^{\{\alpha\}}\cA_{\mu_c})|_{t=0}.
 \end{split}
 \end{equation*}
 On note $f_{\alpha}$ la fonction $t\mapsto(\frac{c^2}{e^t-1}-\frac{c}{e^{c^{-1}t}-1})\cdot e^{\{\alpha\}t}$, qui est $\cC^{\infty}$ sur $\R_{+}$. Comme pour tout $n\in\N$, $t^nf_\alpha(t)$ tend vers $0$ quand $t$ tend vers $+\infty$, on a 
 \begin{equation*}
 (\frac{d}{dt})^{k}((T+1)^{\{\alpha\}}\cA_{\mu_c})|_{t=0}=f^{(k)}_{\alpha}(0)=(-1)^kL(f_{\alpha}(t), -k),
 \end{equation*}
 où $L(f_{\alpha}(t), s)$ est définie par la formule $\frac{1}{\Gamma(s)}\int_0^{+\infty}f_{\alpha}(t) t^{s-1}dt $. 
 
 D'autre part, on a $\z(\{\alpha\}, s)=\frac{1}{\Gamma(s)}\int_{0}^{+\infty}\frac{e^{(1-\{\alpha\})t}}{e^t-1}t^{s-1}dt$. 
 Ceci implique que 
 \[L(f_{\alpha}(t), s)=c^2\z(1-\{\alpha\}, s)-c^{s+1}\z(1-c\{\alpha\}, s).\] On en déduit une égalité algébrique 
 \begin{equation*}
 \begin{split}
 \int_{\Z_p}x_{\alpha}^k\mu_c&=\ord(\alpha)^k(-1)^k(c^2\z(1-\{\alpha\}, -k)-c^{1-k}\z(1-c\{\alpha\}, -k))\\ 
 &=-\ord(\alpha)^k(c^2\z(\{\alpha\}, -k)-c^{1-k}\z(c\{\alpha\}, -k)),\end{split}
\end{equation*} 
où la dernière égalité se déduit du fait que $\z(\{\alpha\}, -k)=(-1)^{k}\z(1-\{\alpha\}, -k)$. 
Si $c\in\Z_p^*$, on choisit une suite $(c_n)$ d'éléments de $\N$ qui converge vers $c$ (resp. $\langle\langle c\rangle\rangle$) dans $\Z_p^{*}$ (resp. $\hat{\Z}$),
de telle sorte que la suite $(c_n\{\alpha\})$ tend vers $\{\langle\langle c\rangle\rangle \alpha \}$. Ceci permet de conclure.
 \end{proof}

\subsubsection{Une famille $p$-adique de formes modulaires}\label{espace}

Si $\alpha\in\Q/\Z$, on définit la série de Dirichlet formelle $\z(\alpha, \kappa^{\mathbf{univ}}, j, s)$ appartenant à $\Dir(\cO(\sW))$ par la formule:
\[\z(\alpha, \kappa^{\mathbf{univ}}, j, s)=\sum_{n\in\Q_+^*, n\equiv \alpha[\Z]}n^{-s}(n\ord(\alpha))^{1-j}\kappa^{\mathbf{univ}}(n\ord(\alpha)).\]
Par définition, on a $\textbf{Ev}_{k+j}(\z(\alpha, \kappa^{\mathbf{univ}}, j, s))=\ord(\alpha)^{k-1}\z(\alpha, s-k+1)$.

Si $1\leq j\in\N$ et $(\alpha,\beta)\in(\Q/\Z)^2$ avec $\ord_p\alpha\geq p$, on pose $\tilde{F}_{\alpha,\beta}(\kappa^{\mathbf{univ}},j)=\sum_{m\in\Q_+^{*}} a_mq^m$ dont sa série de Dirichlet formelle associée est 
\begin{equation}\label{famq-dev}
\z(\alpha, \kappa^{\mathbf{univ}}, j, s)\z^*(\beta,s )+\kappa^{\mathbf{univ}}(-1)(-1)^{2-j}\z(-\alpha, \kappa^{\mathbf{univ}}, j, s)\z^*(-\beta,s ).
\end{equation}

Si $c\in\Z_p^*$, on définit deux séries formelles appartenant à $\cup_{M\in\N}\cO(\sW)[[q_M]]$: 
\begin{equation}
\begin{split}
\tilde{F}_{c,\alpha,\beta}(\kappa^{\mathbf{univ}},j)&=c^2\tilde{F}_{\alpha,\beta}(\kappa^{\mathbf{univ}},j)- \kappa^{\mathbf{univ}}(c^{-1})c^j\tilde{F}_{\langle\langle c\rangle\rangle\alpha,\langle\langle c\rangle\rangle\beta}(\kappa^{\mathbf{univ}},j) \text{ et }
\end{split}
\end{equation}
\begin{equation}
\begin{split}
F_{c,\alpha,\beta}(\kappa^{\mathbf{univ}},j)&=\omega(0_\alpha)^{1-j} \cdot \z_{p,c}(\kappa^{\mathbf{univ}},\alpha,j)+\tilde{F}_{c,\alpha,\beta}(\kappa^{\mathbf{univ}},j).
\end{split}
\end{equation}

\begin{defn}Une famille $p$-adique de formes modulaires sur l'espace des poids $\sW$ est la donnée d'une série formelle $F(q)=\sum_{n\in \Q_{+}} A_nq^n\in \cup_{M\in\Z}\cO(\sW)[[q_M]]$ telle que pour presque tout point $k\in\N\subset \sW$, $\textbf{Ev}_{k}(F(q))$ est le $q$-développement d'une forme modulaire classique.
\end{defn}

\begin{remark}\label{dense}Soit $F$ une famille $p$-adique de formes modulaires sur l'espace des poids. Si $Z$ est un sous-ensemble très Zariski-dense de $\sW$, tel que  $\textbf{Ev}_k(F)=0$ pour tout $k\in Z$,  alors $F=0$.
\end{remark}
\begin{lemma}\label{special}
Si $k,j$ sont deux entiers $\geq 1 $, on a
 \[\mathbf{Ev}_{k+j}(F_{c,\alpha,\beta}(\kappa^{\mathbf{univ}},j))=( \ord\alpha)^{k-1}(c^2F_{\alpha,\beta}^{(k)}-c^{2-k}F_{\langle\langle c\rangle\rangle\alpha,\langle\langle c\rangle\rangle\beta}^{(k)}).\]
 Par conséquent, la série formelle $F_{c,\alpha,\beta}(\kappa^{\mathbf{univ}},j)$ est une famille de formes modulaires sur l'espace des poids.
\end{lemma}
\begin{proof}C'est une conséquence directe de la définition de $F_{c,\alpha,\beta}(\kappa^{\mathbf{univ}},j)$ et de la proposition \ref{q-deve}.
\end{proof}

\begin{lemma}\label{distifam}Soit $(\alpha,\beta)\in(\Q/\Z)^2$ tel que $\ord_p\alpha\geq p$, la famille $F_{c,\alpha,\beta}(\kappa^{\mathbf{univ}},j)$ satisfait les relations de distribution suivantes, quel que soit l'entier $f\geq 1$:
\begin{equation}
\begin{split}
\sum_{f\alpha'=\alpha; f\beta'=\beta}F_{c,\alpha',\beta'}(\kappa^{\mathbf{univ}},j)&=f F_{c,\alpha,\beta}(\kappa^{\mathbf{univ}},j), 
\\ \sum_{ f\beta'=\beta}F_{c,\alpha,\beta'}(\kappa^{\mathbf{univ}},j)(q^{\frac{1}{f}})&=f F_{c,\alpha,\beta}(\kappa^{\mathbf{univ}},j)(q).
\end{split}
\end{equation}
\end{lemma}
\begin{proof}On donne une démonstration pour la première relation et la deuxième se démontre de la même façon.
Si $f\alpha'=\alpha$ et $f\alpha''=\alpha$, on a $\ord(\alpha')=\ord(\alpha'')=\ord(\alpha)f$. En utilisant le lemme \ref{lemma15} et le lemma $\ref{special}$, on déduit la relation de distribution pour $\textbf{Ev}_{k}(F_{\alpha,\beta,c}(\kappa^{\mathbf{univ}},j))$ pour tout  $k>2+j$:
\begin{equation*}
\begin{split}
\textbf{Ev}_{k}(\sum_{f\alpha'=\alpha; f\beta'=\beta}F_{c,\alpha',\beta'}(\kappa^{\mathbf{univ}},j))&=\sum_{f\alpha'=\alpha; f\beta'=\beta}(\ord(\alpha)f)^{k-j-1}  F_{c,\alpha',\beta'}^{(k-j)}\\
&=(\ord(\alpha)f)^{k-j-1}f^{2-k+j} F_{c,\alpha,\beta}^{(k-j)}\\
&=f \textbf{Ev}_{k}(F_{c,\alpha,\beta}(\kappa^{\mathbf{univ}},j)).
\end{split}
\end{equation*}
On conclut le lemme par le remarque $\ref{dense}$.
\end{proof}
On définit une action de $\hat{\Gamma}_0(p)$ sur $F_{c,\alpha,\beta}(\kappa^{\mathbf{univ}},j)$ par la formule:
\[  F_{c,\alpha,\beta}(\kappa^{\mathbf{univ}},j)*\gamma=F_{c, a_0\alpha+c_0\beta,b_0\alpha+d_0\beta}(\kappa^{\mathbf{univ}},j) , \text{ si } \gamma=\bigl(\begin{smallmatrix}a_0 & b_0\\ c_0 &
d_0\end{smallmatrix}\bigr)\in\hat{\Gamma}_0(p),\]
qui commute avec l'application d'évaluation $\textbf{Ev}_k$. 
\subsubsection{La distribution $z_{\mathbf{Eis},c,d}(\kappa^{\mathbf{univ}},j)$}\label{discom}

Dans ce paragraphe, on construit une famille de distributions d'Eisenstein qui interpole en poids $k$ la restriction de la distribution d'Eisenstein $z_{\mathbf{Eis,c,d}}(k,j)$ sur $\bM_2(\Q\otimes\hat{\Z})^{(p)}$ à $\bM_2^{(p)}$, où $\bM_2^{(p)}=\bM_2(\Q\otimes\hat{\Z})^{]p[}\times \I_0(p)$ .

La relation de distribution (c.f. Lemme $\ref{distifam}$) pour $F_{c,\alpha,\beta}(\kappa^{\mathbf{univ}},j)$ se traduit en terme de distribution:
\begin{theo}\label{efamdis}Si $1\leq j\in\N$, il existe une distribution alg\'ebrique  
\[z'_{\mathbf{Eis},c}(\kappa^{\mathbf{univ}},j)\in
\fD_{\alg}((\Q\otimes\hat{\Z})^{(p)}\times (\Q\otimes\hat{\Z}),\cM^{\con}(\Q^{\cycl}\otimes\cO(\sW)))\] v\'erifiant:
soient $ 0\neq r\in \Q^*$ et $(a,b)\in \Q^2$ tels que $a\in \Z_p^*$ et $v_p(r)\geq 1$, on a
\[\int_{(a+r\hat{\Z})\times(b+r\hat{\Z})}z'_{\mathbf{Eis},c}(\kappa^{\mathbf{univ}}-j)=r^{-j}
\ord(r^{-1}a)^{j-1}\kappa^{\mathbf{univ}}(\frac{r}{\ord(r^{-1}a)})F_{c,r^{-1}a,r^{-1}b}(\kappa^{\mathbf{univ}},j).\]
De plus, si $\gamma\in\hat{\Gamma}_0(p)$, on a $z'_{\mathbf{Eis},c}(\kappa^{\mathbf{univ}},j)*\gamma=z'_{\mathbf{Eis},c}(\kappa^{\mathbf{univ}},j)$.
\end{theo}

 En utilisant le fait que le produit d'une forme modulaire classique et d'une famille $p$-adique de formes modulaires  est encore une famille $p$-adique de formes modulaires, on d\'efinit
\[z_{\mathbf{Eis},c,d}(\kappa^{\mathbf{univ}},j)=\frac{1}{(j-1)!}z'_{\mathbf{Eis},c}(\kappa^{\mathbf{univ}},j)\otimes z_{\mathbf{Eis},d}(j)\in\fD_{\alg}(\bM_2^{(p)},\cM(\overline{\Q}\otimes\cO(\sW))).\]

D'après les constructions dans les th\'eor\`emes (\ref{eiskj}) et (\ref{efamdis}),  la distribution $z_{\mathbf{Eis},c,d}(\kappa^{\mathbf{univ}},j)$ possède la propri\'et\'e suivante:
\begin{prop}Si $k\geq 2$ et si $1\leq j\leq k-1$, on a 
\[\mathbf{Ev}_k( z_{\mathbf{Eis},c,d}(\kappa^{\mathbf{univ}},j))=z_{\mathbf{Eis},c,d}(k,j).\]
\end{prop}
\begin{proof}Il suffit de montrer que \[\textbf{Ev}_k(z'_{\mathbf{Eis},c}(\kappa^{\mathbf{univ}},j))=z'_{\mathbf{Eis},c}(k-j).\]  Si $ 0\neq r\in \Q^*$ et $(a,b)\in \Q^2$ tels que $a\in \Z_p^*$ et $v_p(r)\geq 1$, on a
\[\textbf{Ev}_k(\int_{(a+r\hat{\Z})\times(b+r\hat{\Z})}z'_{\mathbf{Eis},c}(\kappa^{\mathbf{univ}},j))=
(\ord r^{-1}a)^{j+1-k} r^{k-2-j} \textbf{Ev}_k(F_{c,r^{-1}a,r^{-1}b} (\kappa^{\mathbf{univ}},j)).\]
Par ailleurs, on a $\ord\frac{\langle\langle c\rangle\rangle a}{r}=\prod_l \ord_l\frac{\langle\langle c\rangle\rangle a}{r}=\ord \frac{a}{r}$ et donc 
\begin{equation*}
\textbf{Ev}_k(F_{c,r^{-1}a,r^{-1}b}(\kappa^{\mathbf{univ}},j))
=(\ord \frac{a}{r})^{k-j-1}F^{(k-j)}_{c, \frac{a}{r}, \frac{b}{r}}.
\end{equation*}
Ceci implique que 
\[\textbf{Ev}_k(\int_{(a+r\hat{\Z})\times(b+r\hat{\Z})}z'_{\mathbf{Eis},c}(\kappa^{\mathbf{univ}},j))=r^{k-j-2}F^{(k-j)}_{c, \frac{a}{r}, \frac{b}{r}}=\int_{(a+r\hat{\Z})\times(b+r\hat{\Z})}z'_{\mathbf{Eis},c}(k-j).\]
\end{proof}

\section{La loi de réciprocité explicite de Kato en famille }
\subsection{Préliminaire }
Dans les deux premiers paragraphes, on rappelle la méthode de Tate-Sen-Colmez utilisée dans $\cite{Wang}$. Ensuite, on établit une version entière de la description explicite de la cohomologie des représentations  analytiques du groupe $P_{\fK_M}$ dans $\S \ref{technique}$.
\subsubsection{L'anneau $\fK^+$}
Soit $\fK^+=\Q_p\{\frac{q}{p}\}$ l'alg\`ebre des fonctions analytiques sur la boule $v_p(q)\geq 1$ \`a coefficients dans $\Q_p$;  c'est un anneau principal complet pour la valuation $v_{p,\fK}$ d\'efinie par la formule:
\[ v_{p,\fK}(f)=\inf_{n\in\N}v_p(a_n), \text{ si } f=\sum_{n\in\N}a_n(\frac{q}{p})^n\in \fK^+ .\]
Dans la suite, on notera $v_p$ au lieu de $v_{p,\fK}$\footnote{ La restriction de la valuation $v_{p,\fK}$ \`a $\Q_p$ co\"incide avec la valuation $p$-adique $v_p$ sur $\Q_p$. }. On note $\fK$ le compl\'et\'e du corps des fractions de l'anneau $\fK^+$ pour la valuation $v_p$. Fixons une cl\^oture alg\'ebrique $\overline{\fK}$ de $\fK$. Comme $\fK$ est un corps complet pour la valuation $v_p$, on peut prolonger  $v_p$ sur $\fK$ \`a $\overline{\fK}$. 
On note le groupe de Galois de $\overline{\fK}$ sur $\fK$ par $\cG_{\fK}$.

Soit $M\geq 1$ un entier. On note $q_M$ (resp. $\z_M$) une racine
$M$-i\`eme $q^{1/M}$ (resp. $\exp(\frac{2i\pi}{M})$) de $q$ (resp.
$1$). On note $F_M=\Q_p[\z_M]$. Soit $\fK_M=\fK[q_{M},\z_M]$ ; c'est une extension galoisienne de $\fK$.
Soit $\fF_M=\fK[\z_M]$ la sous-extension galoisienne de $\fK_M$ sur $\fK$;  la cl\^oture int\'egrale $\fF_M^+$ de $\fK^+$ dans $\fF_M$ est $\fK^+[\z_M]$, qui est l'anneau des fonctions analytiques sur la boule $v_p(q)\geq 1$ \`a coefficients dans $F_M$. Alors, $\fK_M$ est une extension de Kummer de $\fF_M$ de groupe de Galois cyclique d'ordre $M$, dont un g\'en\'erateur $\sigma_M$ est d\'efini par son action sur $q_M$:
\[\sigma_Mq_M=\z_Mq_M.\]

On note $\fK_{\infty}$ (resp. $\fF_{\infty}, F_\infty$) la r\'eunion des $\fK_M$ (resp. $\fF_M, F_M$) pour tout $M\geq 1$. On note $P_{\Q_p}$ (resp. $P_{\ol{\Q}_p}$) le groupe de Galois de $\ol{\Q}_p\fK_{\infty}$ sur $\fK$ (resp. $\ol{\Q}_p\fK$). Le groupe $P_{\ol{\Q}_p}$ est un groupe profini qui est isomorphe à $\hat{\Z}$. De plus, on a une suite exacte:
\[0\ra P_{\ol{\Q}_p}\ra P_{\Q_p}\ra \cG_{\Q_p}\ra 0,\]
et le groupe $P_{\Q_p}$ préserve l'algèbre des formes modulaires $\cM(\ol{\Q})$, ce qui permet de définir une application de localisation $\cG_\fK\ra P_{\Q_p}\ra \Pi_{\Q}$. 

Fixons $M$ un entier $\geq 1$. On note $\fK_{Mp^{\infty}}$ (resp. $\fF_{Mp^\infty}, F_{Mp^{\infty}}$) la r\'eunion des $\fK_{Mp^n}$ (resp. $\fF_{Mp^n}, F_{Mp^\infty}$) pour tous $n\geq 1$. On note $P_{\fK_M}$ le groupe de Galois de $\fK_{Mp^{\infty}}$ sur $\fK_M$. On note $U_{\fK_M}$ le groupe de Galois de $\fK_{Mp^{\infty}}$ sur $\fF_{Mp^{\infty}}$, qui est isomorphe au groupe $\Z_p$, et on note $\Gamma_{\fK_M}$ le groupe de Galois de $\fF_{Mp^{\infty}}$ sur $\fK_M$, qui est isomorphe au groupe $\Gal(F_{Mp^{\infty}}/F_M)$. On a donc une suite exacte:
\[0\ra U_{\fK_M}\ra P_{\fK_M}\ra \Gamma_{\fK_M}\ra 0.\]

Soit $\ol{\fK}^{+}$ la cl\^oture int\'egrale  de $\fK^{+}$ dans $\ol{\fK}$. On a une inclusion $\ol{\Q}_p\subset \ol{\fK}^+$.
On note $\fK_M^+$ la cl\^oture int\'egrale de $\fK^+$ dans $\fK_M$, qui est aussi la cl\^oture int\'egrale de $\fF_M^+$ dans $\fK_M$.

\subsubsection{L'anneau $\B_{\dR}^+$ pour $\fK^+$ et ses sous-anneaux}

Soit $L$ un anneau de caract\'eristique $0$ muni d'une
valuation $v_p$ telle que
 $v_{p}(p)=1$ . On note $\cO_L=\{x\in L, v_p(x)\geq 0\}$ l'anneau
des entiers de $L$. On note $\cO_{\C(L)}$ le compl\'et\'e  de $\cO_{L}$ pour la valuation $v_p$. On pose $\C(L)=\cO_{\C(L)}[\frac{1}{p}]$.

\begin{defn}Soit $A_n=\cO_L/p\cO_L$ pour tout $n\geq 1$; alors le système $\{A_n\}$ muni de morphismes de transition $ A_n\ra
A_{n-1}$ d\'efinis par l'application de Frobenius absolu
$x_n\mapsto x_n^p$ forme un syst\`eme
projectif. On note $\R(L)$ sa limite projective 
\[\plim A_n=\{(x_n)_{n\in \N}|
x_n\in \cO_L/p\cO_L \text{ et } x_{n+1}^p=x_n, \text{ si }n\in \N \}.\]
\end{defn}

Si $x=(x_n)_{n\in\N}\in\R(L)$, soit $\hat{x}_n$ un rel\`evement de
$x_n$ dans $\cO_{\C(L)}$. La suite $(\hat{x}_{n+k}^{p^k})$, converge quand
$k$ tend vers l'infini. On note $x^{(n)}$ sa limite, qui ne
d\'epend pas du choix des rel\`evements $\hat{x}_n$. On obtient ainsi
une bijection: $\R(L)\ra\{(x^{(n)})_{n\in\N}| x^{(n)}\in\cO_{\C(L)},
(x^{(n+1)})^p=x^{(n)}, \forall n\}$. Si $x=(x^{(i)}), y=y^{(i)}$ sont deux \'el\'ements de $\R(L)$, alors leur somme $x+y$ et leur produit $xy$ sont donn\'es par:
\[ (x+y)^{(i)}=\lim_{j\ra\infty}(x^{(i+j)}+y^{(i+j)})^{p^j} \text{ et } (xy)^{(i)}=x^{(i)}y^{(i)}.\]

L'anneau $\R(L)$ est un anneau parfait  de caract\'eristique $p$ (i.e. le morphisme $x\mapsto x^p$ est bijectif.).
On note $\A_{\inf}(L)$ l'anneau des vecteurs de Witt \`a coefficients dans
$\R(L)$. Alors $\A_{\inf}(L)$ est un
anneau $p$-adique (i.e. un anneau s\'epar\'e et complet pour la
topologie $p$-adique), d'anneau r\'esiduel parfait de
caract\'eristique $p$. Si $x\in \R(L)$, on note
$[x]=(x,0,0,...)\in \A_{\Inf}(L)$ son repr\'esentant de Teichm\"uller.
Alors tout \'el\'ement $a$ de $\A_{\inf}(L)$ peut s'\'ecrire de
mani\`ere unique sous la forme $\sum\limits_{k=0}^{\infty}p^k[x_k]$
avec une suite $(x_k)\in (\R(L))^{\N}$.

On d\'efinit un morphisme d'anneaux $\theta: \A_{\inf}(L)\ra\cO_{\C(L)}$
par la formule
\[\sum_{k=0}^{+\infty}p^k[x_k]\mapsto
\sum_{k=0}^{+\infty}p^kx_k^{(0)}.\]
 On note
$\B_{\inf}(L)=\A_{\inf}(L)[\frac{1}{p}]$,  et on \'etend $\theta$ en
un morphisme
\[\B_{\inf}(L)\ra \C(L).\]
 On note $\B_m(L)=\B_{\inf}(L)/(\ker\theta)^m$. On fait de $\B_m(L)$ un anneau de Banach en prenant l'image de $\A_{\inf}(L)$ comme anneau d'entiers.

On d\'efinit $\B_{\dR}^{+}(L):=\plim \B_m(L)$ comme le compl\'et\'e
$\ker (\theta)$-adique de $\B_{\Inf}(L)$; on le munit de la topologie de la limite projective, ce qui en fait un anneau de Fr\'echet. Donc $\theta$ s'\'etend en
un morphisme continu d'anneaux topologiques
\[\B_{\dR}^{+}(L)\ra \C(L).\]

On peut munir $\B_{\dR}^{+}(L)$  d'une filtration de la fa\c{c}on
suivante: pour $i\in\N$, notons $\Fil^i \B_{\dR}^{+}(L)$
la $i$-i\`eme puissance de l'id\'eal $\ker\theta$ de
$\B_{\dR}^{+}(L)$.

L'anneau $\A_{\inf}(L)$ s'identifie canoniquement \`a un sous-anneau de $\B_{\dR}^+(L)$, et
si $k\in\N, m\in\Z$, on pose \[U_{m,k}=p^m\A_{\inf}(L)+(\ker \theta)^{k+1}\B_{\dR}^{+}(L),\]
alors les $U_{m,k}$ forment une base de voisinages de $0$ dans $\B_{\dR}^{+}(L)$.

Pour simplifier la notation, on note $\A_{\inf}$ (resp. $\B_{\inf}$ et $\B_{\dR}^+$) l'anneau $\A_{\inf}(\ol{\fK}^+)$ (resp. $\B_{\inf}(\ol{\fK}^+)$ et $\B_{\dR}^+(\ol{\fK}^+)$).

Soit $\tilde{q}$ (resp. $\tilde{q}_M$ si $M\geq 1$ est un entier) le repr\'esentant de Teichm\"uller dans $\A_{\Inf}$ de $(q,
q_p,\cdots,q_{p^n},\cdots)$ (resp. $(q_M,\cdots,q_{Mp^n},\cdots)$). Si $M|N$, on a
$\tilde{q}_N^{N/M}=\tilde{q}_M$.

On d\'efinit une application continue $\iota_{\dR}:\fK^+\ra \B_{\dR}^+$ par $f(q)\mapsto f(\tilde{q}) $; 
ce qui permet d'identifier $\fK^+$ \`a un sous-anneau de $\B_{\dR}^+(\ol{\fK}^+)$. 
 Mais il faut faire attention au fait que $\iota_{\dR}(\fK^+)$ n'est pas stable par $\cG_{\fK}$ car $\tilde{q}\sigma=\tilde{q}\tilde{\z}^{c_q(\sigma)}$ si
$\sigma\in\cG_{\fK}$, o\`u $c_{q}$ est le $1$-cocycle \`a valeurs dans $\Z_p(1)$ associ\'e \`a $q$
par la th\'eorie de Kummer.

Posons
$\tilde{\fK}^+=\iota_{\dR}(\fK^+)[[t]]$.  Si $M\geq 1$ est un entier, on note $\tilde{\fK}^+_M$ l'anneau
$\tilde{\fK}^+[\tilde{q}_M,\tilde{\z}_M]$. On peut \'etendre l'application $\iota_{\dR}$ en un morphisme continu 
de $\fK^+$-modules $\iota_{\dR}:\fK^+_M\ra \B_{\dR}^{+}(\ol{\fK}^+)$ en envoyant $\z_M$ et $q_M$ sur $\tilde{\z}_M$ et $\tilde{q}_M$ respectivement. 
Alors, on a $\tilde{\fK}_M^+=\iota_{\dR}(\fK_M^+)[[t]]$. On pose
$\tilde{\fK}_{Mp^{\infty}}^+=\bigcup_{n}\tilde{\fK}^+[\tilde{\z}_{Mp^n},\tilde{q}_{Mp^n}]$.

On définit une application $\iota_{\dR}(\fK)^+$-linéaire $\bR_M: \tilde{\fK}^+_{Mp^{\infty}}\longrightarrow\tilde{\fK}^+_M $ par la formule suivante:
\begin{equation*}
\begin{split}
\bR_M: \tilde{\fK}^+_{Mp^{\infty}}&\longrightarrow\tilde{\fK}^+_M \\
\tilde{\z}_{Mp^n}^{a}\tilde{q}_{Mp^n}^b &\mapsto
\left\{
\begin{aligned}
\tilde{\z}_{Mp^n}^{a}\tilde{q}_{Mp^n}^b & ;
\text{ si } p^n|a \text{ et
} p^n|b;\\
0 & ;   \text{ sinon}.
\end{aligned}\right.
\end{split}
\end{equation*}
\begin{prop} Si $M$ est un entier $\geq 1$, alors
\begin{itemize}
\item[(1)]$\rH^0(\cG_{\fK_{Mp^\infty}},\B_{\dR}^+)= \B_{\dR}^+(\fK^+_{Mp^{\infty}})$.
\item[(2)]$\tilde{\fK}^+_{Mp^{\infty}}$ est dense dans $\B_{\dR}^+(\fK^+_{Mp^{\infty}})$.
\item[(3)] Si $v_p(M)\geq v_p(2p)$, alors $\bR_M$ s'étend par continuité en une application $\fK^+$-linéaire $\bR_M: \B_{\dR}^+(\fK^+_{Mp^{\infty}})\ra \tilde{\fK}_M^+$ qui commute à l'action de $\cG_\fK$.
\end{itemize}
\end{prop}
Soit $A$ une $\Q_p$-algèbre de Banach et on note $A^+$ l'anneau des entiers de $A$.
\begin{prop}\label{trace}Si $v_p(M)\geq v_p(2p)$, si $V$ est une $A$-représentation de Banach de $P_{\fK_M}$ possédant une $A^+$-représentation $T$ telle que $P_{\fK_M}$ agisse trivialement sur $T/2pT$, et si $i\in \N$, alors $\bR_M$ induit un isomorphisme:
\[\bR_M: \rH^i(P_{\fK_M}, \B_{\dR}^+(\fK^+_{Mp^\infty})\hat{\otimes}_{\Q_p}V )\cong \rH^i(P_{\fK_M}, \tilde{\fK}_M^+\hat{\otimes}_{\Q_p}V).\] 
\end{prop}
\begin{proof}La démonstration est standard et le lecteur intéressé s'envoie à consulter $\cite{Wang}$ section $3$.
\end{proof}

\subsubsection{La cohomologie des représentations du groupe $P_{\fK_M}$}\label{technique}
Soit $M\geq 1$ tel que $v_p(M)=m\geq v_p(2p)$.
Le groupe de Galois $P_{\fK_M}$ de l'extension
$\fK_{Mp^{\infty}}/\fK_M$ est un groupe analytique $p$-adique compact de rang $2$,
isomorphe \`a
\[\rP_m=\{(\begin{smallmatrix}a& b\\ c& d\end{smallmatrix})\in\GL_2(\Z_p): a=1, c=0, b\in p^m\Z_p, d\in 1+p^m\Z_p \},\]
et si $u,v\in p^m\Z_p$, on note $(\begin{smallmatrix}1&u\\ 0& e^v\end{smallmatrix})$ de $\rP_m$. La loi de groupe s'écrit sous la forme 
\[(u_1,v_1)(u_2,v_2)=(e^{v_2}u_1+u_2, v_1,v_2).\]
Soient $U_m$ et $\Gamma_m$ les sous-groupes de $\rP_m$ topologiquement engendrés par $(p^m,0)$ et $(0,p^m)$ respectivement. Ces deux sous-groupes sont isomorphes à $\Z_p$. De plus,  $U_m$ est distingué dans $\rP_m$ et on a $\rP_m/U_m=\Gamma_m$. Comme $U_m$ et $\Gamma_m$ n'ayant pas de $\rH^2$, la suite spectrale de Hochschild-Serre nous fournit, si $V$ est une $\Q_p$-représentation de dimension finie de $\rP_m$, un isomorphisme
\[\rH^2(\rP_m, V)\cong \rH^1(\Gamma_m, \rH^1(U_m, V))\cong V/( (p^m,0)-1, (0,p^m)-e^{p^m}). \] 

 Soit $\gamma\in\rP_m$; l'image de la fonction analytique $\alpha_{\gamma}:\Z_p\ra\rP_m, \alpha_{\gamma}(x)=\gamma^x$ est un \emph{sous-groupe \`a un param\`etre}. On dira que l'action de $\rP_m$ sur une $\Q_p$-représentation de dimension finie $V$ est analytique si pour tout $\gamma\in\rP_m$ et $v\in V$, la fonction $x\mapsto \gamma^xv=\sum_{n=0}^{+\infty}\binom{x}{n}(\gamma-1)^nv$ est une fonction analytique sur $\Z_p$ à valeurs dans $V$.

Soit $V$ une $\Q_p$-représentation de $\rP_m$.
Si l'action de $\rP_m$ sur $V$ est analytique, pour tout $\gamma\in\rP_m$, on peut d\'efinir une d\'erivation $\partial_{\gamma}:V\ra V$ par rapport \`a $\alpha_{\gamma}$ par la formule:
\[\partial_{\gamma}(x)=\lim_{n\ra\infty}\frac{x*\gamma^{p^n}-x}{p^{n}} .\]
En particulier, on note $\partial_{m,i}$, $i=1,2$, les dérivations par rapport à $(p^m,0)$ et $(0,p^m)$ respectivement.

\begin{defn}\label{def-an}On dira qu'une $\Q_p$-représentation de dimension finie $V$ de $\rP_m$ est analytique si l'action de $\rP_m$ est analytique sur $V$. 
\end{defn} 
On a le r\'esultat tr\`es utile:
\begin{prop}\label{analytic}
Soit $V$ une repr\'esentation analytique de $\rP_m$ munie d'un $\Z_p$-réseau $T$ stable sous l'action de $\rP_m$. Alors, ,
\begin{itemize}
\item[$(i)$] tout \'el\'ement de $ \rH^2(\rP_m,T) $ est repr\'esentable par un  $2$-cocycle analytique à un élément de $p^{2m}$-torsion près;
\item[$(ii)$] on a $ \rH^2(\rP_m,T)\cong T/(\partial_{m,1},\partial_{m,2}-p^m)$,
et l'image d'un $2$-cocycle analytique
\[((u,v),(x,y))\ra c_{(u,v),(x,y)}=\sum_{i+j+k+l\geq 2}c_{i,j,k,l}u^iv^jx^ky^l,\]
avec $p^{i+j+k+l}c_{i,j,k,l}\in T$, par cet isomorphisme, est celle de
$\delta^{(2)}(c_{(u,v),(x,y)})=p^{2m}(c_{1,0,0,1}-c_{0,1,1,0})$ à un élément de $p^{2m}$-torsion près.
\end{itemize}
\end{prop}
\begin{proof}
On dispose des opérateurs $\partial_i:V\ra V$, $i=1,2$, définis par $x*(u,v)=x+u\partial_1x+v\partial_2x+O((u,v)^2)$. 
Ces opérateurs ont des propriété de dérivations:
si
$x_1\in V_1, x_2\in V_2$, o\`u $V_1,V_2$ sont des repr\'esentations
analytiques de $\rP_m$, et si $i=1,2$, on a
\begin{equation*}
\partial_i(x_1\otimes x_2)=(\partial_i
x_1)\otimes x_2+x_1\otimes\partial_i x_2.
\end{equation*}
On a les relations suivantes $\partial_1=p^{-m}\partial_{u_m}, \partial_2=p^{-m}\partial_{\gamma_m}.$

La proposition $4.10$ dans $\cite{Wang}$ dit que si $V$ est une repr\'esentation analytique de $\rP_m$, alors\\
$(i)$ tout \'el\'ement de $ \rH^2(\rP_m,V) $ est repr\'esentable par un  $2$-cocycle analytique;\\
$(ii)$on a un isomorphisme $ \rH^2(\rP_m,V)\cong V/(\partial_1,\partial_2-1)$, et l'image d'un $2$-cocycle analytique,
\[((u,v),(x,y))\ra c_{(u,v),(x,y)}=\sum_{i+j+k+l\geq 2}c_{i,j,k,l}u^iv^jx^ky^l,\]
sous cet isomorphisme  est aussi celle de
$\tilde{\delta}^{(2)}(c_{(u,v),(x,y)})=c_{1,0,0,1}-c_{0,1,1,0}$ dans
$V/(\partial_1,\partial_2-1)$.\\

Comme $\partial_1$ et $\partial_2-1$ introduissent des dominateurs, ils ne préservent pas $T$. Par contre, $\partial_{m,1}$ et $\partial_{m,2}-p^m$ le préservent. La démonstration de la proposition $4.10$ dans $\cite{Wang}$ s'adapte à une démonstration du théorème ci-dessus. Pour faciliter la lecture, on donne l'idée de la démonstration; les détails du calcul se trouvent dans $\cite{Wang}$.

$(1)$ Le groupe $\rP_m$ est de dimension cohomologique $\leq 2$. En utilisant la suite spectrale de Hochschild-Serre, on déduit de la suite exacte $1\ra U_m\ra \rP_m\ra \Gamma_m\ra 1$:
\[\rH^2(\rP_m,T)\cong\rH^1(\Gamma_m,\rH^1(U_m,T))\cong T/(u_m-1,\gamma_m-e^{p^m}).\]
Par ailleurs, l'application surjective naturelle \[\phi: T/(\partial_{m,1}, \partial_{m,2}-p^m)\ra T/(u_m-1,\gamma_m-e^{p^m})\] est un isomorphisme. On en déduit l'isomorphisme  $ \rH^2(\rP_m,T)\cong T/(\partial_{m,1},\partial_{m,2}-p^m)$.
Par cet isomorphisme, il suffit de montrer  le point $(ii)$ du théorème et de montrer que l'application
\[\delta^{(2)}: \{ 2\text{-cocycles analytiques}\}\ra T\] induit une surjection
$\rH^{\an,2}(\rP_m,T)\ra T/(\partial_{m,1},\partial_{m,2}-p^m)$ à un élément de $p^{2m}$-torsion près.\\
$(2)$ Le calcul délicat dans $\cite{Wang}~\S 4.1$, qui montre que l'application
\[\tilde{\delta}^{(2)}: \{ 2\text{-cocycles analytiques à valeurs dans } V\}\ra V\] induit une surjection
$\rH^{\an,2}(\rP_m,V)\ra V/(\partial_{m,1},\partial_{m,2}-p^m)$, s'adapte à notre cas en constatant que l'image d'un $2$-cocycle analytique à valeurs dans $T$ sous l'application $\tilde{\delta}^{(2)}$ est dans $p^{-2m}T$ (c.f. \cite[Lemme 4.14]{Wang}). Plus précisément,
 si $c_{(u,v),(x,y)}=\sum_{i+j+k+l\geq 2}c_{i,j,k,l}u^iv^jx^ky^l$ est un $2$-cocycle analytique sur $\rP_m$ à valeurs dans $T$, on construit un $2$-cobord formel $d(b_{(x,y)})$ tel que $c_{(u,v),(x,y)}-(c_{1,0,0,1}-c_{0,1,1,0})uy= db$. Notons que $b_{(x,y)}$ ne converge pas sur $\rP_m$, mais il converge si on le restreint à $\rP_{m+v_p(2p)}$. Par ailleurs, on a la suite exacte d'inflation-restriction,
 \[0\ra \rH^2(\rP_m/\rP_{m+v_p(2p)}, T^{\rP_{m+v_p(2p)}})\ra \rH^2(\rP_m, T)\ra\rH^2(\rP_{m+v_p(2p)}, T),\]
 où $ \rH^2(\rP_m/\rP_{m+v_p(2p)}, T^{\rP_{m+v_p(2p)}})$ est un groupe de $p^{2m}$-torsion. Ceci permet de prouver le théorème.

\end{proof}
\subsection{L'application exponentielle duale de Kato en famille }\label{constructiondeloi}
\subsubsection{La structure entière de la représentation $\D_{1,j,\sW}$ de $\rP_m$} 
Soient  $L$ une extension finie de $\Q_p$ et $\cO_L$ son anneau des entiers.
À un élément $\mu\in \D_0(\Z_p,L)$, on associe une série formelle:
\[ \cA_\mu(T)=\int_{\Z_p}(1+T)^z\mu,\]
appelée \emph{transformée d'Amice} de $\mu$. On a le lemme suivant:
\begin{lemma}L'application $\mu\mapsto \cA_\mu$ est une isométrie d'espaces de Banach de $\D_0(\Z_p,L)$  sur $\cE_L^+=\cO_L[[T]]\otimes_{\Z_p}\Q_p$.
\end{lemma}

La transformée d'Amice induit un isomorphisme de familles de représentations de Banach de $\D_{0,\rho^{\mathbf{univ}}_{j}}(\sW)$ sur $\cE_{\Q_p}^+\hat{\otimes}\cO(\sW)$.
\begin{lemma}
Soient $u,v\in p^m\Z_p $. Si $\mu$ est une mesure sur $\Z_p$ à valeurs dans $\Q_p$,  l'action de $(u,v)\in \rP_m$ sur $\cA_\mu(T)$ induite par celle sur $\D_{0,\rho^{\mathbf{univ}}_{j}}(\sW)$ est donne par la formule suivante:
\begin{equation}\label{newfor}
(u,v)\cdot \cA_\mu(T)=e^{-vj} (1+T)^u\cA_\mu((1+T)^{e^v}-1).
\end{equation}
\end{lemma}

L'anneau $\cR_{1,L}^+$ des fonctions analytiques sur le disque $v_p(T)\geq1$, noté par $C_1$ dans l'exemple $\ref{espacepoids}$, est un $L$-Banach pour la valuation $v_{\cR_1}$ définie par la formule 
  \[v_{\cR_1}(f)=\inf_{n\in\N} (v_p(b_n)+n), \text{ si }  \sum_{n=0}^{+\infty}b_nT^n\in \cR_{1,L}^+ .\] 
On note $\cR_{1,L}^{++}$ son anneau des entiers pour la valuation $v_{\cR_1}$. 

On note $\D_{1}$ le faisceau de modules de Banach sur l'espace des poids défini par:
$\D_{1}(\sW_n)=\cR^+_{1,\Q_p}\hat{\otimes}\cO(\sW_n)$ pour tout entier $n\geq 1$. 
Comme le $\cO(\sW_n)$-module $\cE_{\Q_p}^+\hat{\otimes}\cO(\sW_n)$ est dense dans $\D_{1}(\sW_n)$, 
l'action de $\rP_m$ sur $\cE_{\Q_p}^+\hat{\otimes}\cO(\sW_n)$
s'étend en une action continue sur $\D_{1}(\sW_n)$. Ceci nous fournit un faisceau de représentations de Banach sur $\sW$, noté par $\D_{1,j}$.
On note $\D^+_{1,j}(\sW_n)=\cR_{1,\Q_p}^{++}\hat{\otimes}_{\Z_p}\cO(\sW_n)^+$
 le sous-$\cO(\sW_n)^+$-module de Banach de $\D_{1,j}(\sW_n)$, qui est la structure entière de la $\cO(\sW_n)$-représentation de Banach $\D_{1,j}(\sW_n)$ de $\rP_m$.

\begin{lemma} L'action de $\rP_m$ sur $\D_{1,j}(\sW_n)$ est analytique pour tout $n$.
\end{lemma}
\begin{proof}
Comme le groupe $\rP_m$ est un groupe analytique compact de rang $2$ engendré par $u_m=(p^m,0)$ et $\gamma_m=(p^m,0)$, il suffit de montrer que, pour tout $f(T)\in\D_{1,j}(\sW_n)$, la fonction $x\mapsto u_m^xf(T)$ (resp. $x\mapsto \gamma_m^xf(T)$) est une fonction analytique sur $\Z_p$ à valeurs dans $\D_{1,j}(\sW_n)$. 

Comme  $\gamma^x=\sum_{n=0}^{+\infty}\binom{x}{n}(\gamma-1)^n$ pour $\gamma\in \rP_m$, on se ramène à estimer la valuation $\frac{(\gamma-1)T^i}{T^i}$ pour tout $i\in\N$ et $\gamma=u_m, \gamma_m$ respectivement. En effet, de la formule (\ref{newfor}), on déduit que pour $i\in\N$ et $\gamma=u_m$ ou $\gamma_m$, la valuation de $\frac{(\gamma-1)T^i}{T^i}$ est $\geq 1$. On donne l'estimation seulement pour $\gamma=\gamma_m$ et l'estimation pour $\gamma=u_m$ se déduit de la même manière.

De la formule (\ref{newfor}) pour $\gamma_m$, on a 
\begin{equation*}
(\gamma_m-1)T^i=((1+T)^{e^{p^m}}-1)^i-T^i=T^i\left((\sum\limits_{k=1}^{+\infty}\binom{e^{p^m}}{k}T^{k-1} )^i-1   \right).
\end{equation*}
On conclut le lemme du fait que la valuation de $(\sum\limits_{k=1}^{+\infty}\binom{e^{p^m}}{k}T^{k-1} )^i-1$ est $\geq1$ dans $\D_{1,j}(\sW_n)$.
\end{proof}
Pour terminer ce paragraphe, on établit les formules pour les actions de $\partial_{m,1}$ et $\partial_{m,2}$ sur $\D_{1,j}(\sW_n)$ pour tout $n$.
\begin{lemma}\label{T}Si $f(T)=\sum_{i=0}^{+\infty}a_i T^i\in\D_{1,j}(\sW_n)$ avec $a_i \in \cO(\sW_n)$, les actions de $\partial_{m,1}$ et $\partial_{m,2}-p^m$ sur $f(T)$  sont donn\'ees par les formules suivantes: 
\begin{equation}
\begin{split}
\partial_{m,1}f(T)&=p^m\log(1+T)f(T),\\ 
(\partial_{m,2}-p^m)f(T)&=p^m(\sum_{i=1}^{+\infty} a_i i(1+T)T^{i-1}\log(1+T)) -p^m(j+1)f(T).
\end{split}
\end{equation}
\end{lemma}
\begin{proof}On donne seulement le calcul pour l'action de $\partial_{m,2}$, et la formule pour l'action de $\partial_{m,1}$ se déduit de la même manière: 
\begin{equation*}
\begin{split}
\partial_{m,2}(T^i)&=\lim_{n\ra \infty}\frac{(\gamma_m^{p^n}-1)T^i }{p^n}=\lim_{n\ra \infty}\frac{\left((1+T)^{e^{p^{m+n}}}-1\right)^i \det^{-j}(\gamma_m^{p^n})-T^i}{p^n}\\
&=\lim_{n\ra \infty}\frac{\det^{-j}(\gamma_m^{p^n})(\sum\limits_{k=0}^{i}\binom{i}{k}(1+T)^{e^{p^{m+n}}k}(-1)^{i-k}) -T^i}{p^n}\\
&=-jp^mT^i+p^m\log(1+T)(\sum\limits_{k=0}^{i}\binom{i}{k} k(1+T)^{k}(-1)^{i-k})\\
&=p^m\left(i(1+T)T^{i-1}\log(1+T)-jT^i\right). 
\end{split}
\end{equation*}

\end{proof}

\subsubsection{La construction de l'application $\exp^*_{\mathbf{Kato},\nu}$ }

On note $\fK^{++}=\Z_p\{\frac{q}{p}\}$ l'anneau des entiers de $\fK^{+}$ pour la valuation $v_p$, ainsi que 
$\cK^{++}$ son compl\'et\'e $q$-adique. 
On note $\fK_M^{++}$ l'anneau des entiers de $\fK_M^+$, qui est l'anneau
\[\{ \sum_{n=0}^{+\infty}a_nq_M^n\in \fK_M^+: a_n\in F_M \text{ tel que } v_p(a_n)+\frac{n}{M} \geq 0  \},\]
et  on note $\cK_M^{++}$ son compl\'et\'e $q$-adique, 
ainsi que $\cK_M^+=\cK_M^{++}\otimes \Q_p$.

Rappelons que l'application $\iota_{\dR}:\fK^+\ra \B_{\dR}^+; f(q)\mapsto f(\tilde{q})$ identifie $\fK^+$ \`a un sous-anneau de $\B_{\dR}^+$. On note 
$\tilde{\fK}^+=\iota_{\dR}(\fK^+)[[t]]$ et $\tilde{\cK}^+=(\widehat{\iota_{\dR}(\fK^{++})}\otimes\Q_p)[[t]]$, où $\widehat{\iota_{\dR}(\fK^{++})}$ est le complété $\tilde{q}$-adique de $\iota_{\dR}(\fK^{++})$. De même, on note $\tilde{\fK}_M^+=\iota_{\dR}(\fK_M^+)[[t]]$ et $\tilde{\cK}_{M}^+=(\widehat{\iota_{\dR}(\fK_M^{++})}\otimes\Q_p)[[t]]$. 
On a bien 
\[\widehat{\iota_{\dR}(\fK_M^{++})}=\{ \sum_{n=0}^{+\infty}a_{n}\tilde{q}_M^n\in F_M[[\tilde{q}_M]]: a_{n}\in F_M \text{ tel que } v_p(a_{n})+\frac{n}{M} \geq 0\}.\]
On d\'efinit une application $\theta: \tilde{\cK}^{+}_M\ra \cK^{+}_M$ par r\'eduction modulo $t$, qui co\"incide avec celle sur $\tilde{\fK}_M^{+}$.
On constate que $\tilde{\fK}_M^+$ est la limite projective $\plim_n(\tilde{\fK}_M^+/t^n)$, où les $\tilde{\fK}_M^+/t^n$ sont des $\fK^+$-modules de rang fini munis de la topologie $p$-adique. 

Fixons un ouvert affinoïde $\sW_n$ de l'espace des poids $\sW$ pour $n\geq 1$ un entier. On a un isomorphisme $\cO(W_n)\cong\Z_p[\Delta]\otimes C_n$,
où $C_n$ est le sous-anneau de $\Q_p[[T_1-1]]$ consistant des fonctions analytiques sur le disque $v_p(T_1-1)\geq \frac{1}{n}$ et $\Delta$ est un groupe cyclique d'ordre $p-1$ engendré par $X_1$. Dans la suite, on identifie $\cO(\sW_n)$ avec $\Z_p[\Delta]\otimes C_n$.

Posons $\tilde{\D}=\plim_{n_1} (\tilde{\fK}_M^+/t^{n_1})\hat{\otimes} \D_{1,j-2}(\sW_n)$, qui est une $\cO(\sW_n)$-représentation de $\rP_m$. Elle n'est pas une $\Q_p$-repr\'esentation analytique. 
 On note $\tilde{\D}^{+,n_1}=(\widehat{\iota_{\dR}(\fK_M^{++})}\hat{\otimes}\D^+_{1,j-2}(\sW_n)[[t]])/t^{n_1}$ le $\Z_p$-réseau de  $(\tilde{\fK}_M^+/t^n)\hat{\otimes} \D_{1,j-2}(\sW_n)$ qui est stable sous l'action de $\rP_m$. 
 On déduit de la formule ($\ref{newfor}$)  
que l'idéal $\m=(T_1-1,T,\tilde{q}_M)$ de $\tilde{\D}^{+,n_1}$ est stable sous l'action de $\rP_m$.  Ceci nous permet de définir des $\Z_p$-représentations analytiques $\tilde{\D}^{+,n_1,n_2}$ de $\rP_m$, pour tous $n_1,n_2\geq 1$,
\[\tilde{\D}^{+,n_1,n_2}=(\widehat{\iota_{\dR}(\fK_M^{++})}\hat{\otimes}\D^+_{1,j-2}(\sW_n)[[t]])/t^{n_1})/\m^{n_2}.\]
  
L'inclusion $ \tilde{\D}\subset\plim_{n_1}\left((\plim_{n_2} \tilde{\D}^{+,n_1,n_2})\otimes\Q_p\right)$ de $\rP_m$-représentations, nous permet de définir un morphisme:
\[\rH^i(\rP_m, \tilde{\D})\ra \plim_{n_1}\rH^i(\rP_m, (\plim_{n_2} \tilde{\D}^{+,n_1,n_2})\otimes\Q_p) \ra \plim_{n_1}\left((\plim_{n_2}\rH^i(\rP_m,  \tilde{\D}^{+,n_1,n_2}))\otimes\Q_p\right).  \]

\begin{lemma}Les actions de $\partial_{m,1}$ et $\partial_{m,2}-p^m$ sur $t$ et $\tilde{q}_M$ sont donn\'ees par les formules suivantes:
\begin{equation*}
\begin{split}
\partial_{m,1}(t)=0, &\partial_{m,1}(\tilde{q}_M)=\frac{p^mt}{M}\tilde{q}_M;
\partial_{m,2}(t)=p^mt,\partial_{m,2}(\tilde{q}_M)=0.\end{split}
\end{equation*}

\end{lemma}
\begin{proof}Le lemme se déduit d'un calcul direct, qui se trouve dans $\cite{Wang}$ $\S 5.1$.
\end{proof}

\begin{prop}\label{reskj}Si $v_p(M)=m\geq v_p(2p)$ et si $j\geq 1$, alors l'application 
\[f(q_M)\mapsto \cA_{\nu_{j-2}}  t^{j-1} f(\tilde{q}_M) \] induit un isomorphisme de $\cK_{M}^{++}\hat{\otimes}_{\Z_p}\cO(\sW_n)$ sur $\plim_{n_1}\left((\plim_{n_2}\tilde{\D}^{+,n_1,n_2}/(\partial_{m,1},\partial_{m,2}-p^m))\otimes \Q_p\right)$, où $\cA_{\nu_{j-2}}$ est la transformée d'Amice de la section globale $\nu_{j-2}$ de $\D_{0,\rho^{\mathbf{univ}}_{j-2}}(\sW_n)\subset \D_{1,j-2}(\sW_n)$.
\end{prop}
\begin{remark} La section globale $\nu_{j-2}$ est la masse de Dirac en $0$ (c.f. lemme$~\ref{intervect}$). On a $\cA_{\nu_{j-2}}=1$, mais il faut faire attention que l'action de $\I_0(p)$ sur $\tilde{\D}^{+,n_1,n_2}$ n'est pas triviale:  si $\gamma=(\begin{smallmatrix}a&b\\ c&d\end{smallmatrix})\in\I_0(p)$, on a 
$\cA_{\gamma\nu_{j-2}}=\cA_{\delta_{\frac{b}{a}}}\cdot\kappa^{\mathbf{univ}}(a)\det(\gamma)^{2-j}$,  où $\delta_{\frac{b}{a}}$ est la masse de Dirac en $\frac{b}{a}$ (si $\gamma\in \rP_m$, on a $a=1$ et donc $\kappa^{\mathbf{univ}}(a)=1$).

\end{remark}
Pour démontrer la proposition $(\ref{reskj})$, on a besoin d'un lemme préparatoire. 
On définit une valuation $p$-adique $v_p$ sur $F_M[\Delta]$ par la formule: 
\[v_p(x)=\inf_{0\leq n\leq p-2}v_p(a_n), \text{ si } x=\sum_{n=0}^{p-2}a_nX_1^n \text{ avec } a_n\in F_M.\] 
Alors tous les éléments de $\tilde{\D}^{+,n_1,n_2}$ sont de la forme
\[ \sum\limits_{\substack{0\leq k,l,r,s\\ k+l+r\leq n_2-1;\\ s\leq n_1-1}}a_{k,l,r,s}T^k(T_1-1)^l\tilde{q}_M^rt^s,\] 
 où $a_{k,l,r,s}\in F_M[\Delta]$ vérifie $v_p(a_{k,l,r,s})+\frac{r}{M}+k\geq 0$.

On note $\bM_{n_1,n_2}$ le sous-$\Z_p$-module de $\tilde{\D}^{+,n_1,n_2}$ des éléments de la forme
\[ \sum\limits_{\substack{0\leq l,r,s\\ l+r\leq n_2-1;\\ s\leq n_1-1}}a_{l,r,s}\cA_{\nu_{j-2}}(T_1-1)^l\tilde{q}_M^rt^s  \text{ où } a_{k,r,s}\in F_M \text{ vérifie } v_p(a_{l,r,s})+\frac{r}{M}\geq 0 .\]
 On constate qu'il n'existe pas d'élément de $\tilde{\D}^{+,n_1,n_2}$ tel que $\partial_{m,1} x$ appartient à $\bM_{n_1,n_2}$. Par conséquent, l'application naturelle
\[\phi_1: \bM_{n_1,n_2}\ra \tilde{\D}^{+,n_1,n_2}/\partial_{m,1} \] 
est injective.   

\begin{lemma}Le $\Z_p$-module $\bM_{n_1,n_2}+\partial_{m,1}\tilde{\D}^{+, n_1,n_2}$ contient $p^{2m(n_1+1)}T\tilde{\D}^{+, n_1,n_2}$, pour tout $n_1,n_2\geq 1$.
\end{lemma}
\begin{proof} 
D'après le lemme $\ref{T}$, on a la formule $\partial_{m,1}(T^k)=p^m\log(1+T)T^k$ et donc 
\begin{equation}\label{parti_M1}
 \partial_{m,1}(T^k\tilde{q}_M^rt^s)=p^mT^k \log(1+T) \tilde{q}_M^rt^{s}+rT^k\frac{p^m}{M}\tilde{q}_M^rt^{s+1}.
 \end{equation}
Si $x\in \tilde{\D}^{+,n_1,n_2}$, alors $x$ est de la forme  
 \[\sum\limits_{\substack{0\leq k,l,r,s;\\ k+l+r\leq n_2-1;\\ s\leq n_1-1}}a_{k,l,r,s}T^k(T_1-1)^l\tilde{q}_M^rt^s, \text{ avec } a_{k,l,r,s}\in F_M[\Delta] \text{ vérifiant } v_p(a_{k,l,r,s})+\frac{r}{M}+k\geq 0.\]   
Il s'agit de montrer que tous les termes $p^{2m(n_1+1)}a_{k,l,r,s}T^k(T_1-1)^l\tilde{q}_M^rt^s$, où  $k\geq 1$ et $a_{k,l,r,s}\in F_M[\Delta]$ satisfait $v_p(a_{k,l,r,s})+\frac{r}{M}+k\geq 0$, sont dans $\partial_{m,1}\tilde{\D}^{+,n_1,n_2}$.

(i) Si $r=0$,  la formule $(\ref{parti_M1})$ s'écrit alors sous la forme: 
\[  \partial_{m,1}(a_{k,l,0,s}T^k(T_1-1)^{l}t^s)=p^{m}a_{k,l,0,s}T^k(T_1-1)^l \log(1+T)t^s .\]
Par ailleurs, on a $\log(1+T)=T(1+\sum_{i=1}^{+\infty}\frac{(-T)^i}{i+1})$. On constate que $1+\sum_{i=1}^{+\infty}\frac{(-T)^i}{i+1}$ est inversible dans $\tilde{\D}^{+, n_1,n_2}$ et on note son inverse par $g(T)$.

Supposons que $x\in T\tilde{\D}^{+,n_1,n_2}$.
Comme $v_p(a_{k,l,0,s})+k\geq 0$, on a $v_p(a_{k,l,0,s})+m+ k-1\geq 0$; ceci implique que 
\[y=\sum\limits_{\substack{0\leq l,s;1\leq k; \\ k+l\leq n_2-1;\\ s\leq n_1-1}} p^ma_{k,l,0,s}T^{k-1}(T_1-1)^{l}t^s\] appartient à  $\tilde{\D}^{+,n_1,n_2}$ et on a $\partial_{m,1}(yg(T))=p^{2m}x.$ 

(ii) Dans la suite, on suppose que $r\neq 0$. On démontrera le lemme dans ce cas par récurrence descendante sur $s$.\\
Si $s=n_1-1$, la formule $(\ref{parti_M1})$ s'écrit alors sous la forme: 
\[  \partial_{m,1}(a_{k,l,r,s}T^{k}(T_1-1)^l\tilde{q}_M^rt^{s})=p^{m}a_{k,l,r,s}T^{k}(T_1-1)^l \log(1+T)\tilde{q}_M^rt^s.\]
On en déduit que $\partial_{m,1}\tilde{\D}^{+,n_1,n_2}$ contient tous les termes $p^{m}a_{k,l,r,s}T^k(T_1-1)^l\tilde{q}_M^rt^s$ avec $s=n_1-1$ et $k\geq 1$; et donc tous les termes $p^{2m}a_{k,l,r,s}T^k(T_1-1)^l\tilde{q}_M^rt^s$ avec $s=n_1-1$ et $k\geq 1$.
 
Fixons $s_0$ un entiers tel que $0<s_0\leq n_1-2$.
Supposons que $\bM_{n_1,n_2}+\partial_{m,1}\tilde{\D}^{+,n_1,n_2}$ contient tout les termes $p^{2m(n_1-s+1)}a_{k,l,r,s}T^k(T_1-1)^l\tilde{q}_M^rt^s$ avec $s\geq s_0+1$. Il s'agit à montrer que $\bM_{n_1,n_2}+\partial_{m,1}\tilde{\D}^{+,n_1,n_2}$ contient tout les termes $x_{k,l,r,s}=p^{2m(n_1-s+1)}a_{k,l,r,s}T^k(T_1-1)^l\tilde{q}_M^rt^s$ avec $s=s_0$ et $k\geq 1$.

Comme $v_p(a_{k,l,r,s})+\frac{r}{M}+k\geq 0 $ avec $k\geq 1$, on a $v_p(a_{k_0,l,r,s})+\frac{r}{M}+m+k-1\geq 0$ et donc $y_{k, l,r,s_0}=p^{m+2m(n_1-s_0)}a_{k,l,r,s_0}T^{k-1}(T_1-1)^l\tilde{q}_M^rt^{s_0}$ appartient à $\tilde{\D}^{+,n_1,n_2}$. Par récurrence, on a $ty_{k, l,r,s_0}\in \bM_{n_1,n_2}+\partial_{m,1}\tilde{\D}^{+,n_1,n_2}$.

De la formule $(\ref{parti_M1})$, on a: 
\begin{equation*} 
\begin{split}\partial_{m,1}(y)=&p^{2m(n_1-s_0+1)}a_{k,l,r,s_0}T^{k-1}(T_1-1)^l\log(1+T)\tilde{q}_M^rt^{s_0}+\frac{rp^m}{M}\tilde{q}_M^r y t
\\
=&x_{k,l,r,s_0}g(T)+\frac{rp^m}{M}\tilde{q}_M^r y t.
\end{split}
\end{equation*}
On en déduit que le lemme est vrai pour les termes avec $s=s_0$. 

 \end{proof}
\begin{coro}\label{prepare}Si $n_1,n_2$ sont deux entiers $\geq1$, le conoyau $\coker\phi_1$ de $\phi_1$ est un $\Z_p$-module de $p^{2m(n_1+1)}$-torsion. 
\end{coro}

On revient à la démonstration de la proposition $(\ref{reskj})$.
\begin{proof}[Démonstration de la prop. $(\ref{reskj})$]
Comme on a  
\begin{equation}\label{partial2}(\partial_{m,2}-p^m)( \cA_{\nu_{j-2}} \tilde{q}_M^r t^s)=p^m(s-j+1)\cA_{\nu_{j-2}}\tilde{q}_M^rt^s,
\end{equation} 
le $\Z_p$-module $\bM_{n_1,n_2}$ est stable sous l'action de $\partial_{m,2}-p^m$.
Donc l'application $\phi_1$ induit une application injective, que l'on note encore par $\phi_1$,
\[\phi_1:\bM_{n_1,n_2}/(\partial_{m,2}-p^m)\ra \tilde{\D}^{+,n_1,n_2}/(\partial_{m,1},\partial_{m,2}-p^m).\]
En plus, on déduit du corollaire (\ref{prepare}) que son conoyau est un $\Z_p$-module de $p^{2m(n_1+1)}$-torsion.

Si $n_1\geq 1$, on note
\[(\cK_{M}^{++}\hat{\otimes}_{\Z_p}\cO(\sW_n)^+)_{n_1}=(\cK_M^{++}\hat{\otimes}_{\Z_p}\cO(\sW_n)^+)/(T_1-1,q_M)^{n_1}.\]
Pour tout $n_1\geq j+1$, on dispose d'une application  $\phi_0: (\cK_M^{++}\hat{\otimes}\cO(\sW_n)^+)_{n_2}\ra \bM_{n_1,n_2}$ en envoyant $f(q_M)$ sur $f(\tilde{q}_M)t^{j-1}\cA_{\nu_{j-2}}$, qui est une injection. De la formule $(\ref{partial2})$ pour $s=j-1$, on déduit que $\phi_0$ induit une application injective, notée encore par $\phi_0$,
 \[\phi_0:(\cK_M^{++}\hat{\otimes}\cO(\sW_n)^+)_{n_2}\ra\bM_{n_1,n_2}/(\partial_{m,2}-p^m).\]
 En composant avec l'application $\phi_1$, on obtient une application injective 
  \[\phi=\phi_1\circ\phi_0: (\cK_M^{++}\hat{\otimes}\cO(\sW_n)^+)_{n_2}\ra \tilde{\D}^{+,n_1,n_2}/(\partial_{m,1},\partial_{m,2}-p^m), \]
  
 En prenant la limite projective sur $n_2$, on obtient une injection 
 \[\cK_M^{++}\hat{\otimes}\cO(\sW_n)^+ \ra (\plim_{n_2}\tilde{\D}^{+,n_1,n_2}/(\partial_{m,1},\partial_{m,2}-p^m)) \text{ pour } n_1\geq  j-1 .\] 
 Il ne reste qu'à montrer la surjectivité de
  \[\cK^{++}\hat{\otimes}\cO(\sW_n)\ra (\plim_{n_2}\tilde{\D}^{+,n_1,n_2}/(\partial_{m,1},\partial_{m,2}-p^m))\otimes \Q_p.\] 
 Cela se ramène à montrer que les applications 
 \[\plim_{n_2}\phi_0: \cK^{++}\hat{\otimes}\cO(\sW_n)\ra (\plim_{n_2}\bM_{n_1,n_2}/(\partial_{m,2}-p^m))\otimes\Q_p\] et 
 \[\plim_{n_2}\phi_1: (\plim_{n_2}\bM_{n_1,n_2}/(\partial_{m,2}-p^m))\otimes\Q_p\ra (\plim_{n_2}\tilde{\D}^{+,n_1,n_2}/(\partial_{m,1},\partial_{m,2}-p^m))\otimes \Q_p \] 
 sont surjectives.  La surjectivité de $\plim_{n_2}\phi_0$ découle de la formule  $(\ref{partial2})$ et celle de $\plim_{n_2}\phi_1$ découle du lemme précédent qui dit que, pour tout $n_1>j+1$, le conoyau de $\phi_1$ est de $p^{2m(n_1+1)}$-torsion.  
\end{proof}

En composant les applications obtenues dans les paragraphes pr\'ec\'edents, on obtient le diagramme suivant:
\[
\xymatrix{
 \rH^2(  \cG_{\fK_M},\B_{\dR}^{+}\hat{\otimes} \D_{1,j-2}(\sW_n)         ) &
\\
\rH^2(  P_{\fK_M},\B_{\dR}^{+}(\fK^+_{Mp^{\infty}})\hat{\otimes}\D_{1,j-2}(\sW_n)  ) \ar[u]^-{(1)}\ar[r]^-{(2)}\ar@{..>}[dd]^{\exp^{*}_{\mathbf{Kato},\nu}} &
\rH^2\left(  P_{\fK_M}, \tilde{\D} \right)\ar[d]^-{(3)}
\\
 & \plim_{n_1}\left((\plim_{n_2}\rH^2( P_{\fK_M},\tilde{\D}^{+,n_1,n_2}) ) \otimes_{\Z_p} \Q_p  \right) \ar[d]^-{(4)}
\\
\cK_{M}^{++}\hat{\otimes}_{\Z_p}\cO(\sW_n) &\plim_{n_1}\left((\plim_{n_2} (\tilde{\D}^{+,n_1,n_2}/(\partial_{m,1},\partial_{m,2}-p^m))\otimes _{\Z_p}\Q_p\right)\ar[l]^-{\cong}_-{(5)},
}
\]
o\`u

$\bullet$ l'application $(1)$, d'inflation, est injective car $(\B_{\dR}^{+})^{\cG_{\fK_{Mp^{\infty}}}}=\B_{\dR}^{+}(\fK_{Mp^{\infty}}^+)$ et $\cG_{\fK_{Mp^{\infty}}}$ agit trivialement sur $ \D_{1,j-2}(\sW_n)$;

$\bullet$ $(2)$ est l'isomorphisme induit par "la trace de
Tate normalis\'ee" $\bR_M$ (c.f. prop. $\ref{trace}$ );

$\bullet$ $(3)$ est l'application naturelle induite par la projection ; 

$\bullet$ $(4)$ est l'isomorphisme de la proposition $\ref{analytic}$ car $\tilde{\D}^{+,n_1,n_2} $ est analytique pour tout $(n_1,n_2)$;

$\bullet$ $(5)$ est l'inverse de l'isomorphisme dans la proposition $\ref{reskj}$.

On d\'efinit l'application $\exp^{*}_{\mathbf{Kato},\nu}$  en composant les applications $(2),(3), (4), (5)$.

\subsection{Application à la famille de systèmes d'Euler de Kato}
Dans ce paragraphe, on montrera le théorème \ref{theo}. Soient $M\geq 1$ tel que $v_p(M)\geq v_p(2p)$ et $A=\bigl(\begin{smallmatrix}\alpha&\beta\\ \gamma&
\delta\end{smallmatrix}\bigr)\in \Gamma_0(p)$ avec
$\alpha,\beta,\gamma,\delta\in\{1,\cdots,M\}$. On note $\psi_{M,A}=1_{A+M\bM_2(\hat{\Z})}$ la fonction caract\'eristique de $A+M\bM_2(\hat{\Z})$. C'est une fonction invariante sous l'action de $\cG_{\fK_M}$. Par ailleurs, la distribution $z_{\mathbf{Kato},c,d}(\nu_j)$ appartient \`a $\rH^2(\cG_{\fK_M},\fD_{\alg}(\bM_2(\Q\otimes \hat{\Z})^{]p[}\times\I_0(p),\D_{1,j-2}(\sW_n)))$. Alors, on a
\[\int\psi_{M,A}z_{\mathbf{Kato},c,d}(\nu)\in\rH^2(\cG_{\fK_M},\D_{1,j-2}(\sW_n))\]
et on note son image dans $\rH^2(\cG_{\fK_M},\B_{\dR}^{+}(\fK_{Mp^{\infty}}^+)\hat{\otimes}
\D_{1,j-2}(\sW_n))$ par $z_{M,A}$.
Pour montrer le th\'eor\`eme (\ref{theo}),
il suffit de prouver:
\begin{prop}\label{exam}
Pour toute paire $(M,A)$ ci-dessus, on a
\begin{equation*}
\exp^{*}_{\mathbf{Kato},\nu}(z_{M,A})
=\frac{M^{-2j}}{(j-1)!}
\ord(\frac{\alpha}{M})^{j-1}\kappa^{\mathbf{univ}}(\frac{M}{\ord(\alpha/M)})F_{c,\alpha/M,\beta/M}(\kappa^{\mathbf{univ}},j)E_{d,\gamma/M,\delta/M}^{(j)}
.
\end{equation*}
\end{prop}
Pour d\'emontrer ceci, nous aurons besoin d'\'ecrire un $2$-cocycle
explicite repr\'esentant $z_{M,A}$ et le suivre \`a travers les
\'etapes de la construction de l'application $\exp^{*}_{\mathbf{Kato},\nu}$.

\subsubsection{Construction d'un $2$-cocycle}

Rappelons que notre famille de systèmes d'Euler de Kato est construite selon le même chemin que celui du système d'Euler de Kato classique. Ceci nous permet d'utiliser la construction d'un $2$-cocycle explicite pour le système d'Euler de Kato classique dans $~\cite{Wang}$, ce qui est reliée à la construction du système d'Euler de Kato, sauf que on doit transféfer l'action de groupe à droite en celle à gauche. On expose le résultat ci-dessous, et renvoie le lecteur intéressé à $\cite{Wang}~\S 5.2$ pour les détails de la construction.

Soient $\Lambda_1, \Lambda_2$ deux $G$-modules \`a gauche. Si
$x_1\in\Lambda_1, x_2\in\Lambda_2$ et $\sigma,\tau\in G$, on d\'efinit un
\'el\'ement $\{x_1\otimes
x_{2}\}_{\sigma,\tau}:=((\tau\sigma-\sigma)\circ x_1\otimes((\sigma-1)\circ x_2))\in\Lambda_1\otimes\Lambda_2$.
Un cocycle explicite, qui présente l'image de $z_{M,A}$ dans $\rH^2(P_{\fK_M},\tilde{\fK}^+_{M}\hat{\otimes}\D_{1,j-2,\sW})$, est donné comme suit: 
\begin{theo}\label{ZM}
Si on note
\[\log_{\left(\begin{smallmatrix}a_0&b_0\\c_0&d_0\end{smallmatrix}\right)}^{(\sigma,\tau)}=\left\{\log\left((c^2-<c>)\theta(\tilde{q}^{p^n},\tilde{q}_M^{a_0}\tilde{\z}_M^{b_0})\right)
\otimes\log\left((d^2-<d>)\theta(\tilde{q}^{p^n},\tilde{q}^{c_0}_M\tilde{\z}_M^{d_0})\right)\right\}_{\sigma,\tau},\]
 $z_{M,A}$ peut se repr\'esenter par le $2$-cocycle
\begin{equation*}
\begin{split}
(\sigma,\tau)\mapsto \lim_{n\ra+\infty}p^{-2n}
\sum\cA_{(\begin{smallmatrix}a_0&b_0\\ c_0& d_0\end{smallmatrix})\nu_j}\log_{\left(\begin{smallmatrix}a_0&b_0\\c_0&d_0\end{smallmatrix}\right)}^{(\sigma,\tau)},
\end{split}
\end{equation*}
 la somme portant sur l'ensemble
\[U^{(n)}:=\{(a_0,b_0,c_0,d_0)\in\{1,\cdots,Mp^n\}^4|a_0\equiv\alpha, b_0\equiv\beta, c_0\equiv\gamma, d_0\equiv\delta\mod M\}.\]
\end{theo}
\subsubsection{Passage à l'algèbre de Lie}
On utilise les techniques différentielles pour calculer l'image du $2$-cocycle
\[(\sigma,\tau)\mapsto \lim_{n\ra+\infty}p^{-2n}
\sum\cA_{(\begin{smallmatrix}a_0&b_0\\ c_0& d_0\end{smallmatrix})\nu_j}\log_{\left(\begin{smallmatrix}a_0&b_0\\c_0&d_0\end{smallmatrix}\right)}^{(\sigma,\tau)},
\]
obtenu dans le théorème ci-dessus, dans $\cK_M^+\hat{\otimes}\cO(\sW_n)$ par l'application $\exp^*_{\textbf{Kato},\nu}$. Plus pr\'ecis\'ement, cela se fait comme suit:
   \begin{recette}\label{res} Si $n_1>j+1$, on d\'efinit une application $\res_{\nu,j}^{(k)}: \tilde{\D}\ra (\cK_{M}^{++}\hat{\otimes}_{\Z_p}\cO(\sW_n))/(q_M)^k$ en composant la
projection $\tilde{\D}\ra \tilde{\D}/t^{n_1}\ra (\plim_{n_2}\tilde{\D}^{+,n_1,n_2}/(\partial_{m,1},\partial_{m,2}-p^m))\otimes\Q_p$ avec l'inverse de l'isomorphisme dans la proposition
$\ref{reskj}$. En prenant la limite projective sur $k$, on obtient un morphisme $\res_{\nu,j}:\tilde{\D}\ra \cK_{M}^{++}\hat{\otimes}_{\Z_p}\cO(\sW_n)$. Si la classe de cohomologie
\[c=(c^{(n_1,n_2)})\in \plim_{n_1} ((\plim_{n_2}\rH^2(P_{\fK_M},\tilde{\D}^{+,n_1,n_2} )\otimes\Q_p) \] est repr\'esent\'e par une limite de $2$-cocycle analytique $(\sigma,\tau)\mapsto c^{(n_1,n_2)}_{\sigma,\tau}$ sur $P_{\fK_M}$ \`a valeurs dans $\tilde{\D}^{+,n_1,n_2}$, alors l'image de $c$ sous l'application $(5)$ est $\res_{\nu,j}(p^{-2m}\delta^{(2)}(c))$, o\`u $\delta^{(2)}$ est l'application d\'efinie dans la proposition $\ref{analytic}$.
\end{recette}

Comme le $2$-cocycle $(\sigma,\tau)\mapsto \lim\limits_{n\ra+\infty}p^{-2n}
\sum(\begin{smallmatrix}a_0&b_0\\ c_0& d_0\end{smallmatrix})\nu_j\log_{\left(\begin{smallmatrix}a_0&b_0\\c_0&d_0\end{smallmatrix}\right)}^{(\sigma,\tau)}$ obtenu dans le théorème ci-dessus est la
 limite de $2$-cocycles analytiques \`a valeurs dans $\tilde{\D}$,  on peut utiliser les techniques diff\'erentielles
 pour calculer son image dans $\cK^{++}_M\hat{\otimes}\cO(\sW_n)$ par l'application exponentielle duale $\exp^*_{\mathbf{Kato},\nu}$.

Si $f(x_1,x_2)$ est une fonction en deux variables,
on note $D_1$ ( resp. $D_2$ ) l'op\'erateur $x_1\frac{d}{dx_1}$ (
resp. $x_2\frac{d}{dx_2}$ ). Si $n\in\N$ et $a,b\in\Z$, on pose
$f_{a,b}^{(n)}=f(\tilde{q}^{p^n},\tilde{q}_M^a\tilde{\z}_M^b)$. Du développement limité du terme de $(u,v)f_{a,b}^{(n)}=f(\tilde{q}^{p^n}, \tilde{q}_M^a\tilde{\z}_M^{au+be^v})$ en $u$ et $v$, on déduit que:
\begin{equation}
\partial_{m,1}f_{a,b}^{(n)}=\frac{ap^mt}{M}D_2f_{a,b}^{(n)} \text{ et } 
\partial_{m,2}f_{a,b}^{(n)}=\frac{bp^mt}{M}D_2f_{a,b}^{(n)} ,
\end{equation}
ce qui joueront un rôle dans les démonstrations du lemme $\ref{2co}$ et $\ref{negligible}$ ci-dessous.
\begin{lemma}[\cite{Wang} Lemme 5.15]\label{2co}
On note $\delta^{(2)}_{a_0,b_0,c_0,d_0}=\tilde{\delta}^{(2)}\left(
\left\{\log \left(r_c\theta_{a_0,b_0}^{(n)}\right)\otimes\log\left(r_d\theta_{c_0,d_0}^{(n)}\right)\right\}_{\sigma,\tau}\right)$.
 On a
\begin{equation*}
\begin{split}
\delta^{(2)}_{a_0,b_0,c_0,d_0}=\frac{(a_0d_0-b_0c_0)t^2}{M^2}\cdot
D_2\log\left(r_c\theta_{a_0,b_0}^{(n)}\right)\cdot
D_2\log\left(r_d\theta_{c_0,d_0}^{(n)}\right).
\end{split}
\end{equation*}
\end{lemma}
D'après le lemme $\ref{T}$, on a les formules suivantes pour $\partial_{m,1}$ et $\partial_{m,2}$ sur $\cA_{(\begin{smallmatrix}a&b\\ c& d\end{smallmatrix})\nu_j}$:
\begin{lemma}On a les formules suivantes:
\begin{equation}\label{neg}
\begin{split}
\partial_{m,1}(\cA_{(\begin{smallmatrix}a&b\\ c& d\end{smallmatrix})\nu_j})&=p^m\log(1+T) \cA_{(\begin{smallmatrix}a&b\\ c& d\end{smallmatrix})\nu_j},\\
\partial_{m,2}(\cA_{(\begin{smallmatrix}a&b\\ c& d\end{smallmatrix})\nu_j})&=p^m(\frac{b}{a}\log(1+T)-j) \cA_{(\begin{smallmatrix}a&b\\ c& d\end{smallmatrix})\nu_j}.
\end{split}
\end{equation}  
\end{lemma}
\begin{proof} Ce lemme est une traduction du lemme $\ref{T}$. On donne seulement le calcul pour $\partial_{m,2}$: on note $\gamma=(\begin{smallmatrix}a&b\\ c& d\end{smallmatrix})$ et on a
\begin{equation*}
\begin{split}
\partial_{m,2}(\cA_{\gamma\nu_j})=&\partial_{m,2}(\kappa^{\mathbf{univ}}(a)\det(\gamma)^{-j}\sum_{k=0}^{+\infty} \binom{\frac{b}{a}}{k}T^k)\\
=&\kappa^{\mathbf{univ}}(a)\det(\gamma)^{-j}\sum_{k=0}^{+\infty} \binom{\frac{b}{a}}{k} k(1+T)T^{k-1}\log(1+T)\\
=&\frac{b}{a}\kappa^{\mathbf{univ}}(a)\det(\gamma)^{-j}\sum_{k=1}^{+\infty} \binom{\frac{b}{a}-1}{k-1} (1+T)T^{k-1}\log(1+T)\\
=&\frac{b}{a}\kappa^{\mathbf{univ}}(a)\det(\gamma)^{-j}(1+T)\log(1+T) (1+T)^{\frac{b}{a}-1}=\frac{b}{a}\log(1+T)\cA_{\gamma\nu_j}.
\end{split}
\end{equation*}

\end{proof}
\begin{lemma}\label{negligible} Si $s\geq j+1$ et $a,b,c,d\in\Z$, alors, dans $\plim_{n_1}\left((\plim_{n_2} (\tilde{\D}^{+,n_1,n_2}/(\partial_{m,1},\partial_2-1))\otimes _{\Z_p}\Q_p\right)$, on a 
\[\cA_{(\begin{smallmatrix}a&b\\ c& d\end{smallmatrix})\nu_j} t^sf_{a,b}^{(n)} g_{c,d}^{(n)}=\frac{(ad-bc)t^{s+1}}{aM(j+1-s)}\cA_{(\begin{smallmatrix}a&b\\ c& d\end{smallmatrix})\nu_j }   \cdot f_{a,b}^{(n)}\cdot D_2g_{c,d}^{(n)}.\]
\end{lemma}
\begin{proof}On en déduit que 
\begin{equation*}
\begin{split}
&(a(p^m-\partial_{m,2}) +b\partial_{m,1})( \cA_{(\begin{smallmatrix}a&b\\ c& d\end{smallmatrix})\nu_j} t^sf_{a,b}^{(n)} g_{c,d}^{(n)})\\
=&\left(a(j+1-s)p^mt^sf_{a,b}^{(n)}g_{c,d}^{(n)}-p^m\frac{(ad-bc)t^{s+1}}{M} f_{a,b}^{(n)}D_2g_{c,d}^{(n)} \right)\cA_{(\begin{smallmatrix}a&b\\ c& d\end{smallmatrix})\nu_j},
\end{split}
\end{equation*}
 est nul dans $\plim_{n_1}\left((\plim_{n_2} (\tilde{\D}^{+,n_1,n_2}/(\partial_{m,1},\partial_{m,2}-p^m))\otimes _{\Z_p}\Q_p\right)$, pour $s\geq j+1$.  
\end{proof}
\begin{coro}\label{olla}On a 
\[\res_{\nu,j}\left(   \cA_{(\begin{smallmatrix}a_0&b_0\\ c_0& d_0\end{smallmatrix})\nu_j} \delta^{(2)}_{a_0,b_0,c_0,d_0}  \right)=\frac{\kappa^{\mathbf{univ}}(a_0)}{a_0^{j-1}M^{j+1}(j-1)!}\cdot D_2\log(r_c\theta_{a_0,b_0}^{(n)})\cdot
D_2^j\log(r_d\theta_{c_0,d_0}^{(n)})\]

\end{coro}
\begin{proof}
D'apr\`es le lemme \ref{2co} et le lemme \ref{negligible}, on a
\begin{equation*}
\begin{split}
&\cA_{(\begin{smallmatrix}a_0&b_0\\ c_0& d_0\end{smallmatrix})\nu_j}\cdot\delta^{(2)}_{a_0,b_0,c_0,d_0}\\
=&\cA_{(\begin{smallmatrix}a_0&b_0\\ c_0& d_0\end{smallmatrix})\nu_{j-2}} \frac{(a_0d_0-b_0c_0)^{-1}}{M^2}\cdot
D_2\log\left(r_c\theta_{a_0,b_0}^{(n)}\right)\cdot
D_2\log\left(r_d\theta_{c_0,d_0}^{(n)}\right)\\
=&M^{-1-j}\frac{(a_0d_0-b_0c_0)^{j-2}t^{j+1}}{a_0^{j-1}(j-1)!}\cA_{(\begin{smallmatrix}a_0&b_0\\ c_0& d_0\end{smallmatrix})\nu_j}\cdot D_2\log\left(r_c\theta_{a_0,b_0}^{(n)}\right)\cdot
D_2^j\log\left(r_d\theta_{c_0,d_0}^{(n)}\right)
.
\end{split}
\end{equation*}
D'autre part, on a $\cA_{(\begin{smallmatrix}a_0&b_0\\ c_0& d_0\end{smallmatrix})\nu_{j-2}}=(1+T)^{\frac{b_0}{a_0}} (a_0d_0-b_0c_0)^{2-j}\kappa^{\mathbf{univ}}(a_0)$.
Le corollaire se déduit de la définition de $\res_{\nu,j}$.
\end{proof}

On rappelle le fait que $D^r_2\log\left(r_c\theta_{a_0,b_0}^{(n)}\right)=c^2E_r(x_1,x_2)-c^r E_r(x_1,x_2^c)$, noté par $E_{c,r}(x_1,x_2)$. Si $b\equiv\beta\mod M$ et $d\equiv \delta\mod M$, on a $\z_M^{b}=\z_M^\beta$ et $\z_M^d=\z_M^\delta$. Donc par le corollaire \ref{olla} et le calcul que nous avons fait, on obtient:
\[\exp_{\mathbf{Kato},\nu}^{*}(z_{M,A})=\frac{M^{-1-j}}{(j-1)!}\lim\limits_{n\ra\infty}\sum_{\substack{a_0\equiv\alpha [M]\\
c_0\equiv\gamma[M]\\ 1\leq a_0,c_0\leq Mp^n}}\eta(a_0)a_0^{1-j}E_{c,1}(q^{p^n},q_M^{a_0}\z_M^\beta) E_{d,j}(q^{p^n},q_M^{c_0}\z_M^\delta).\]

Enfin, on utilise le lemme suivant pour terminer le calcul:
\begin{lemma}\label{final}
Si $1\leq r\in\N$ et $c,d\in\Z_p^*$, on a
\begin{itemize}
\item[(1)] $\sum\limits_{\substack{c_0\equiv \gamma [M]\\ 1\leq c_0\leq Mp^n }}E_j(q^{p^n},q_M^{c_0}\z_M^{\delta})=E_{\gamma/M,\delta/M}^{(j)}$, et 
$\sum\limits_{\substack{c_0\equiv \gamma [M]\\ 1\leq c_0\leq Mp^n }}E_{d,j}(q^{p^n},q_M^{c_0}\z_M^{\delta})=E_{d,\gamma/M,\delta/M}^{(j)}$;
\item[(2)]
\begin{equation*}
\begin{split} 
&\lim\limits_{n\ra\infty}\sum_{\substack{a_0\equiv\alpha[M]\\ 1\leq a_0\leq Mp^n}}\kappa^{\mathbf{univ}}(a_0)a_0^{1-j}E_{c,1}(q^{p^n},q_M^{a_0}\z_M^\beta)
   \\
   =&
M^{1-j}\ord(\frac{\alpha}{M})^{j-1}\kappa^{\mathbf{univ}}(\frac{M}{\ord(\alpha/M)})F_{c,\alpha/M,\beta/M}(\kappa^{\mathbf{univ}},j).
\end{split}
\end{equation*}
\end{itemize}
\end{lemma}
\begin{proof}
Le $(1)$ est montré dans $\cite{Wang}$ lemme 5.17; on ne montre que le $(2)$ dans la suite.  On remarque que la limite dans $(2)$ est une limite $p$-adique. La démonstration se divise en deux parties: la première consiste à comparer les séries de Dirichlet formelles associées aux $q$-développements de ces familles de formes modulaires $p$-adiques et la deuxième consiste à comparer les termes constants de ces deux familles.  

(i) (Comparaison de deux séries de Dirichlet formelles) 
On se ramène à montrer que la série formelle associée à $\lim\limits_{n\ra\infty}\sum_{\substack{a_0\equiv\alpha[M]\\ 1\leq a_0\leq Mp^n}}\kappa^{\mathbf{univ}}(a_0)a_0^{1-j}E_{1}(q^{p^n},q_M^{a_0}\z_M^\beta)
$ et celle associée à $M^{1-j}\ord(\frac{\alpha}{M})^{j-1}\kappa^{\mathbf{univ}}(\frac{M}{\ord(\alpha/M)})\tilde{F}_{\alpha/M,\beta/M}({\kappa^{\mathbf{univ}},j})$ sont les mêmes.
  
D'apr\`es la proposition \ref{q-deve}, on a:
\begin{equation*}
E_1(q^{p^n},
q_M^{a_0}\z_M^{\beta})=E_{a_0/Mp^n,\beta/M}^{(1)}(q^{p^n})=F_{a_0/Mp^n,\beta/M}^{(1)}(q^{p^n}).
\end{equation*}
Donc:
\begin{equation*}
\begin{split}
\lim\limits_{n\ra\infty}\sum_{\substack{a_0\equiv\alpha[M]\\ 1\leq
a_0\leq Mp^n}}\kappa^{\mathbf{univ}}(a_0)a_0^{1-j}E_1(q^{p^n},q_M^{a_0}\z_M^\beta)=
\lim\limits_{n\ra\infty}\sum_{\substack{a_0\equiv\alpha[M]\\ 1\leq
a_0\leq Mp^n}}\kappa^{\mathbf{univ}}(a_0)a_0^{1-j}F^{(1)}_{a_0/Mp^n,\beta/M}(q^{p^n}).
\end{split}
\end{equation*}

Soit $F^{(1)}_{a_0/Mp^n,\beta/M}(q)=\sum_{m\in\Q^{+}}b_m q^m$. La s\'erie de Dirichlet formelle  $\sum_{m\in\Q^{+}}\frac{b_m}{m^s}$,  \`a coefficients dans $\Q^{\cycl}$, associ\'ee \`a $F^{(1)}_{a_0/Mp^n,\beta/M}(q)$ v\'erifie
la relation suivante:
\begin{equation*}
\sum_{m\in\Q^{+}}\frac{b_m}{m^s}=\z(a_0/Mp^n,s)\z^{*}(\beta,s)-\z(-a_0/Mp^n,s)\z^{*}(-\beta,s).
\end{equation*}
Donc la s\'erie de Dirichlet formelle \`a coefficients dans $\Q^{\cycl}$ associ\'ee \`a $F^{(1)}_{a_0/Mp^n,\beta/M}(q^{p^n})$
satisfait:
\begin{equation}\label{formelle}
\begin{split}
&\sum_{m\in\Q^{+}}\frac{b_m}{(p^nm)^s}
=p^{-ns}\sum_{m\in\Q^{+}}\frac{b_m}{n^s}
\\&=p^{-ns}(\z(a_0/Mp^n,s)\z^{*}(\beta/M,s)-\z(-a_0/Mp^n,s)\z^{*}(-\beta/M,s)).
\end{split}
\end{equation}

Soit $\sum_{m\in \Q^{+}}A_{m,n} m^{-s}$ la s\'erie de Dirichlet formelle \`a coefficients dans $\Q^{\cycl}\otimes\cO(\sW)$ associ\'ee au $q$-d\'eveloppement
de $\sum_{\substack{a_0\equiv\alpha[M]\\ 1\leq
a_0\leq Mp^n}}\kappa^{\mathbf{univ}}(a_0)a_0^{1-j}F^{(1)}_{a_0/Mp^n,\beta/M}(q^{p^n})$.

Alors, de la formule $(\ref{formelle})$, on obtient:

\begin{equation}\label{fin}
\begin{split}
\sum_{m\in\Q^{*}_{+}}\frac{A_{m,n}}{m^s}=&\sum_{i=0}^{p^n-1}\sum_{k=0}^{\infty}
(\frac{\kappa^{\mathbf{univ}}(\alpha+iM)(\alpha+iM)^{1-j}}{(kp^n+\frac{a_0}{M})^s}\z^{*}
(\beta/M,s)+\kappa^{\mathbf{univ}}(-1)(-1)^{2-j} \cdot\\ 
&\frac{\kappa^{\mathbf{univ}}(-\alpha-iM)(-\alpha-iM)^{1-j}}{(kp^n-\frac{\alpha}{M}+(p^n-i))^s}\z^{*}(-\beta/M,s)).
\end{split}
\end{equation}
Dans la suite, on donne les calculs pour le terme $\frac{\kappa^{\mathbf{univ}}(\alpha+iM)(\alpha+iM)^{1-j}}{(kp^n+\frac{a_0}{M})^s}\z^{*}
(\beta/M,s)$ ci-dessus, et le terme restant se calcule de la même manière:

\begin{equation}
\begin{split}
&\sum_{\substack{a_0\equiv\alpha[M]\\ 1\leq a_0\leq Mp^n}}\kappa^{\mathbf{univ}}(a_0)a_0^{1-j}p^{-ns}\z(a_0/Mp^n,s)\z^{*}(\beta/M,s)\\
=&\sum_{\substack{a_0\equiv\alpha[M]\\ 1\leq a_0\leq Mp^n}}\kappa^{\mathbf{univ}}(a_0)a_0^{1-j}p^{-ns}\sum_{k=0}^{+\infty}\frac{1}{(k+\frac{a_0}{Mp^n})^s}\z^{*}(\beta/M,s)\\
=&\sum_{\substack{a_0\equiv\alpha[M]\\ 1\leq a_0\leq Mp^n}}\sum_{k=0}^{+\infty}\frac{\kappa^{\mathbf{univ}}(a_0)a_0^{1-j}}{(kp^n+\frac{a_0}{M})^s}\z^{*}
(\beta/M,s).
\end{split}
\end{equation}

En prenant la limite $p$-adique de la formule $(\ref{fin})$ et en utilisant le fait $(p,\alpha)=1$, on a :
\begin{equation*}
\begin{split}
&\lim\limits_{n\ra\infty}\sum_{m\in\Q^{*}_{+}}\frac{A_{m,n}}{m^s}\\
=&\lim_{n\ra\infty}
( 
\sum_{\substack{ \frac{\alpha}{M}\leq\tilde{\alpha}< \frac{\alpha}{M}+p^n\\ \tilde{\alpha}\equiv \frac{\alpha}{M}\mod\Z }} 
\frac{\kappa^{\mathbf{univ}}(M\tilde{\alpha})(M\tilde{\alpha})^{1-j} }{ \tilde{\alpha}^s } \z^{*}(\beta/M,s)\\ &+\kappa^{\mathbf{univ}}(-1)(-1)^{2-j} \cdot
\sum_{\substack{ -\frac{\alpha}{M}<\tilde{\alpha}\leq -\frac{\alpha}{M}+p^n\\ \tilde{\alpha}\equiv -\frac{\alpha}{M}\mod\Z }}  
\frac{\kappa^{\mathbf{univ}}(M\tilde{\alpha})(M\tilde{\alpha})^{1-j}}{\tilde{\alpha}^s}
\z^{*}(-\beta/M,s)
).
\end{split}
\end{equation*}

C'est la m\^eme s\'erie de Dirichlet formelle associ\'ee au $q$-d\'eveloppement de (c.f. la formule (\ref{famq-dev}) pour le $q$-développement de $\tilde{F}_{\alpha/M,
\beta/M}(\kappa^{\mathbf{univ}},j)$)
\[
M^{1-j}\ord(\frac{\alpha}{M})^{j-1}\kappa^{\mathbf{univ}}(\frac{M}{\ord(\alpha/M)})\tilde{F}_{\alpha/M,
\beta/M}(\kappa^{\mathbf{univ}},j)\] .

(ii)(Comparaison de termes constants) On rappelle que la mesure $\mu_c$ sur $\Z_p$ dont sa transformée d'Amice est $\frac{c^2}{T}-\frac{c}{(1+T)^{c^{-1}}-1}$, satisfait les relations: 
\begin{equation}
\begin{split}
\int_{\Z_p}(x_{\alpha/M})^k\mu_c&=-\ord(\alpha/M)^k(c^2\z(\frac{\alpha}{M}, -k)-c^{1-k}\z(\frac{\langle \langle c\rangle\rangle \alpha}{M}, -k)); \\ 
\int_{p^n\Z_p}(x_{\alpha/M})^k\mu_c&=-\ord(\alpha/M)^k(c^2\z(\frac{\alpha}{p^nM}, -k)-c^{1-k}\z(\frac{\langle \langle c\rangle\rangle\alpha}{p^nM}, -k)).
\end{split}
\end{equation}
 On constate que 
le terme constant de la série à gauche de la relation voulue est la somme de Riemann  \[\lim\limits_{n\ra+\infty}
\sum_{\substack{a_0\equiv\alpha[M]\\ 1\leq
a_0\leq Mp^n}}\kappa^{\mathbf{univ}}(a_0)a_0^{1-j}(c^2\z(\frac{a_0}{p^nM}, 0)-c\z(\frac{\langle \langle c\rangle\rangle a_0}{p^nM}, 0)),\]
et donc il se traduit en l'intégration $p$-adique 
\[M^{1-j}\ord(\frac{\alpha}{M})^{j-2}\kappa^{\mathbf{univ}}(\frac{M}{\ord(\alpha/M)})\int_{\Z_p}\kappa^{\mathbf{univ}}(x_{\alpha/M})(x_{\alpha/M})^{2-j}\mu_c.\] On en déduit que le terme constant de la série à gauche de la relation voulue est 
\[M^{1-j}\ord(\frac{\alpha}{M})^{j-2}\kappa^{\mathbf{univ}}(\frac{M}{\ord(\alpha/M)})\omega(0_{\alpha/M})^{2-j}\z_{p,c}(\kappa^{\mathbf{univ}},\alpha/M).\]
Ceci permet de conclure.
\end{proof}
En appliquant ceci à  $\exp^{*}_{\mathbf{Kato},\nu}(z_{M,A})$, on obtient  
\[\exp^{*}_{\mathbf{Kato},\nu}(z_{M,A})=\frac{1}{(j-1)!}M^{-2j}\ord(\frac{\alpha}{M})^{j-1}\kappa^{\mathbf{univ}}(\frac{M}{\ord(\alpha/M)})F_{c,\alpha/M,
\beta/M}(\kappa^{\mathbf{univ}},j) E^{(j)}_{d, \gamma/M, \delta/M}. \]

Ceci termine la démontration du théorème $\ref{theo}$.


Dipartimento di Matematica Pura e Applicata, Via Trieste 63, I-35121, Padova, Italy; swang@math.unipd.it

\begin{thebibliography}{99}
\let\ts=\textsc
\bibitem{AI}\ts{F. Andreatta; A. Iovita}: Comparison isomorphisms for formal Schemes, preprint.
\bibitem{Be}\ts{J. Bellaïche}: Critical $p$-adic L-functions, Invent. Math. vol.\textbf{189} no. 1, p.1-60 (2012).
\bibitem{Ch}\ts{G. Chenevier}: Une correspondance de Jacquet-Langlands p-adique, 
Duke math. journal \textbf{126} no.1, 161-194 (2005).
\bibitem{PC1}\ts{P. Colmez}: La Conjecture de Birch et
Swinnerton-Dyer p-adique,~\emph{Ast\'erisque} \textbf{294} (2004).
\bibitem{PC2}\ts{P. Colmez}: Fonctions d'une variable $p$-adique, \textit{Ast\'erisque} \textbf{330} (2010), p. 13-59.
\bibitem{DJ}\ts{A.J. de Jong}: Crystalline Dieudonn\'e module theory via formal and rigid geometry, Publications Mathematiques I.H.E.S., \textbf{82} (1995), pp. 5-96.
\bibitem{DD}\ts{D. Delbourgo}: Elliptic curves and big Galois representations, London Mathematical Society Lecture Note Series, \textbf{356}. Cambridge University Press, Cambridge, (2008).
\bibitem{Fa} \ts{G. Faltings}: Almost \'etale extensions, Ast\'erisque \textbf{279} (2002), p.185-270.
\bibitem{Fu}\ts{T. Fukaya}: Coleman power series for $K_2$ and $p$-adic zeta functions of modular forms. Kazuya Kato's fiftieth birthday. Doc. Math. (2003), Extra Vol., 387–442 (electronic).
\bibitem{KK}\ts{K. Kato}: $p$-adic Hodge theory and values of zeta
functions of modular forms, Ast\'erisque \textbf{295} (2004).
\bibitem{Sen}\ts{S. Sen}: Lie algebras of Galois groups arising from Hodge-Tate modules, Ann. of Math.\textbf{97} (1973) 160-170.
\bibitem{Ta}\ts{J. Tate}: $p$-divisible groups, Proc. of a conference on local field, Nuffic summer school at Driebergen,  Springer, Berlin (1967) p.158-183.
\bibitem{AP}\ts{A. Panchishkin}: A new method of constructing $p$-adic $L$-functions associated with modular forms, Moscow Mathematical Journal Vol. \textbf{2} No.\textbf{2} (2002)
\bibitem{Wang}\ts{S. Wang}: Le syst\`eme d'Euler de Kato, pr\'epublication.
\bibitem{Wang1}\ts{S. Wang}: Le système d'Euler de Kato en famille (II), en préparation.
\bibitem{Weil}\ts{A. Weil}: Elliptic functions according to Eisenstein and Kronecker, Erg. der Math. \textbf{88}, Springer-Verlag, (1976).

\end{thebibliography}
\end{document}